\documentclass[10pt,hyperref,reqno,a4paper]{amsart}
\usepackage{graphicx} 
\usepackage{float} 
\usepackage{amsmath}
\usepackage{color}
\usepackage{amsfonts,amsthm,amssymb}
\usepackage{mathtools}
\usepackage{bm} 
\usepackage[colorlinks,linkcolor=cyan,citecolor=cyan]{hyperref}
\usepackage{extarrows}
\usepackage[hang,flushmargin]{footmisc} 
\usepackage{mathrsfs} 
\usepackage[font=footnotesize,skip=0pt,textfont=rm,labelfont=rm]{caption,subcaption} 
\usepackage{booktabs} 

\def\eqdefa{\buildrel\hbox{\footnotesize def}\over =}

\def\S{\mathbb{S}}
\def\R{\mathbb{R}}
\def\T{\mathbb{T}}

\newcommand{\1}{\mathbf{1}}

\newcommand{\g}{\bm{g}}
\newcommand{\h}{\bm{h}}

\renewcommand{\L}{\mathbf{L}}

\renewcommand{\S}{\mathbb{S}}

\newcommand{\ga}{\gamma}

\newcommand{\pa}{\partial}
\newcommand{\na}{\nabla}

\newcommand{\al}{\alpha}

\newcommand{\lam}{\lambda}

\renewcommand{\>}{\rangle}

\renewcommand{\P}{\bm{P}}
\newcommand{\PP}{\bm{P}_{\pm}}

\allowdisplaybreaks
\theoremstyle{plain}
\newtheorem{thm}{Theorem}[section]
\newtheorem{cor}[thm]{Corollary}
\newtheorem{lem}[thm]{Lemma}
\newtheorem{rmk}[thm]{Remark}

\usepackage[]{geometry}                           
\usepackage{fullpage}
\numberwithin{equation}{section} 

\usepackage{refcheck}
\newcommand{\vertiii}{{\vert\kern-0.25ex\vert\kern-0.25ex\vert}}

\begin{document}
\title{The Vlasov-Maxwell-Boltzmann/Landau system with polynomial perturbation near Maxwellian}
\date{}



\author{Chuqi Cao}
\address[Chuqi Cao]{Yau Mathematical Science Center and Beijing Institute of Mathematical Sciences and Applications, Tsinghua University\\
Beijing 100084,  P. R.  China.} \email{chuqicao@gmail.com}

\author{Dingqun Deng}
\address[Dingqun Deng]{Department of Mathematics, Pohang University of Science and Technology, Pohang, Republic of Korea (South), ORCID: \href{https://orcid.org/0000-0001-9678-314X}{0000-0001-9678-314X}} \email{dingqun.deng@postech.ac.kr} \email{dingqun.deng@gmail.com}

\author{Xingyu Li}
\address[Xingyu Li] {Institut de Math\'ematiques de Toulouse , Universit\'e de Toulouse III Paul Sabatier, Toulouse, France.} \email{xingyuli92@gmail.com}
    



\begin{abstract}
     In this paper, we study the Vlasov-Maxwell-Boltzmann system without angular cutoff and the Vlasov-Maxwell-Landau/Boltzmann system with polynomial perturbation $F=\mu+f$ near global Maxwellian. In particular, we prove the global existence, uniqueness and
large time behavior for solutions in a polynomial weighted space $H^N_{x,v}(\langle v\rangle^k)$. The method is based on Duhamel's principle with the crucial time-decay analysis on the particle distribution $f$ and the electromagnetic field $(E,B)$.

\end{abstract}

\maketitle
\tableofcontents

\section{Introduction}

The Vlasov-Maxwell-Boltzmann equation reads
\begin{equation}
\label{VMB1}
\begin{cases}
	\pa_tF_+ + v\cdot \na_xF_++ (E+v\times B)\cdot\na_vF_+ = Q(F_+,F_+)+Q(F_-,F_+),\\
	\pa_tF_- + v\cdot \na_xF_- -(E+v\times B)\cdot\na_vF_- = Q(F_+,F_-)+Q(F_-,F_-),\\
	F_\pm(0,x,v)=F_{0,\pm}(x,v),\quad E(0,x)=E_0(x), \quad B(0,x)=B_0(x)
\end{cases}
\end{equation}
where $E, B$ satisfies 
\begin{equation}\left\{
\begin{aligned}
\label{VMB2}
&\partial_t E - \nabla_x \times B = -\int_{\R^3} v(F_+-F_{-}) dv , \quad \partial_t B + \nabla_x \times E =0,\\
&\nabla_x \cdot E = \int_{\R^3} (F_+-F_{-}) dv, \quad  \nabla_x \cdot B =0. 
\end{aligned}\right.
\end{equation}
Here $F(t, x, v) \ge 0$ is the distribution function of colliding particles at time $t>0$ and position $x \in \T^3$, moving with at velocity $v \in \R^3$, and $[E(t,x), B(t,x)]$ is the self-consistent, spatially periodic electromagnetic field coupled to $F(t,x,v)$.

\subsubsection{Boltzmann collision operator}
The Boltzmann collision operator $Q$ is a bilinear operator which acts only on the velocity variable $v$ given by 
\[
 Q(G,F)(v)=\int_{\R^3}\int_{\mathbb{S}^2}B(v-v_*,\sigma )(G'_*F'-G_*F)d\sigma dv_*.
\]
Let us give some explanations on the collision operator.
\begin{enumerate}
\item  We use the standard shorthand $F=F(v),G_*=G(v_*),F'=F(v'),G'_*=G(v'_*)$, where $v',v'_*$ are given by
\[
v'=\frac{v+v_*}{2}+\frac{|v-v_*|}{2}\sigma, \quad v_*'=\frac{v+v_*}{2}-\frac{|v-v_*|}{2}\sigma, \quad \sigma\in\mathbb{S}^2.
\]
 This representation follows the parametrization of the set of solutions of the physical law of elastic collision:
\[
  v+v_*=v'+v'_*,  \quad
  |v|^2+|v_*|^2=|v'|^2+|v'_*|^2.
\]

\item  The nonnegative function $B(v-v_*,\sigma)$ in the collision operator is called the Boltzmann collision kernel. It is  assumed depending  only on $|v-v_*|$ and the deviation angle $\theta$ through \[\cos\theta\eqdefa \frac{v-v_*}{|v-v_*|}\cdot\sigma .\]
\item  In the present work,  our {\bf basic assumptions on the kernel $B$}  can be concluded as follows:
\begin{itemize}
  \item[$\mathbf{(A1).}$] The Boltzmann kernel $B$ takes the form
$B(v-v_*,\sigma)=|v-v_*|^\gamma b(\frac{v-v_*}{|v-v_*|}\cdot\sigma )$,
 where   $b$ is a nonnegative function.

  \item[$\mathbf{(A2).}$] The angular function $b(\cos \theta)$ is not locally integrable and it satisfies
\[
  \mathcal{K}\theta^{-1-2s}\leq \sin\theta b(\cos\theta) \le \mathcal{K}^{-1}\theta^{-1-2s},~\mbox{with}~0<s<1,~\mathcal{K}>0.
\]

  \item[$\mathbf{(A3).}$]
  The parameter $\gamma$ and $s$ satisfy the condition $-3 < \gamma  \le 1$ and $s \ge 1/2.$

  \item[$\mathbf{(A4).}$]  Without lose of generality, we may assume that $B(v-v_*,\sigma )$ is supported in the set $0\leq \theta \le \pi/2$, i.e.$\frac{v-v_*}{|v-v_*|}   \cdot  \sigma  \ge 0$, for otherwise $B$ can be replaced by its symmetrized form:
\[
\overline{B}(v-v_*,\sigma )=|v-v_*|^\gamma\big(b(\frac{v-v_*}{|v-v_*|}\cdot\sigma) + b(\frac{v-v_*}{|v-v_*|}\cdot(-\sigma ))\big)  \mathrm{1}_{\frac{v-v_*}{|v-v_*|}\cdot\sigma \ge 0},
\]
where $\mathrm{1}_A$ is the characteristic function of the set $A$.
\end{itemize}
\end{enumerate}

\subsubsection{Landau collision operator}
We also consider \eqref{VMB1} as the Vlasov-Maxwell-Landau equation, and $Q$ as the Landau operator
\begin{equation}
\begin{aligned}
\label{vpl3}
Q(g, f)(v)&=\nabla_v\cdot\int_{\R^3}   \Phi(v-v_*)(g(v_*)  \nabla_v  f(v)-f(v)  \nabla_{v_*}g(v_*)) dv_*,\\
&\text{with}\,\, \Phi(u)=|u|^{\gamma+2}\left(I-\frac{u\otimes u}{|u|^2}\right),\,\, \gamma\in[-3,1].
\end{aligned}
\end{equation}
The Landau operator \eqref{vpl3} can be equivalently written in the form
\[
Q(g, f)(v)  = \partial_i  \int_{\R^3}  a_{ij} (v-v_*) (g_* \partial_j f - f  \partial_j g_*  )    dv_*,
\]
where we use the convention of summation of repeated indices, and the derivatives are in the velocity variable $\partial_i = \partial_{v_i}$. Hereafter we use the shorthand notations $g_* = g(v_*), f = f(v), \partial_j g_* = \partial_{v_{*j}}g(v_*), \partial_j f = \partial_{v_j} f(v) $,
and 
\[
a_{i j} (v) = |v|^{\gamma+2} \left( \delta_{ij} - \frac {v_i v_j}  {|v|^2}\right), \quad -3 \le \gamma \le 1. 
\]
Then we define
\begin{equation*}
\begin{aligned}
&b_i(v) = \partial_j a_{ij} (v) = -2|v|^\gamma v_i,\\
& c(v) = \partial_{ij} a_{ij}(v) = -2(\gamma + 3) |v|^\gamma,  \quad\text{ when }-3 < \gamma \le 1,\\
& c(v) = - 8 \pi \delta_0,  \quad\text{ when } \gamma=-3,
\end{aligned}
\end{equation*}
Then we have $Q(g, f) = \partial_i \big( (a_{ij} *g )\partial_j f - (b_i * g)  f   \big) = (a_{ij} *g ) \partial_{ij} f -(c * g) f.$
\begin{rmk}
We will call the cases $\gamma\ge 0$ and $\gamma<0$ hard potential and soft potential respectively.
\end{rmk}
\subsection{Basic properties and the perturbation equation} We recall some basic facts on the Boltzmann/Landau equation.
 \smallskip

\noindent$\bullet$ {\bf Conservation Law.}  Formally if $F$ is the solution to \eqref{VMB1} with the initial data $F_0$, then it enjoys the conservation of mass, momentum, and energy, that is,
\begin{equation}
\label{conse1}
\frac d {dt} \int_{\T^3} \int_{  \R^3} F_+ dv dx = \frac d {dt} \int_{\T^3} \int_{  \R^3} F_{-} dv dx =0, \quad\frac d {dt}\int_{\T^3} \int_{  \R^3} v( F_++F_{-}) dv dx =0, \,\,
\end{equation}
\begin{equation}
\label{conse2a}
 \frac d {dt} \int_{\T^3} \int_{  \R^3} |v|^2( F_++F_{-}) dv dx+\frac {d} {dt}  \int_{\T^3} |E  (t, x)|^2 +  |B  (t, x)|^2dx=0.
 \end{equation}
Taking the inner product of \eqref{VMB1} with 1 over $\mathbb R^3_v$, we get the continuity equation 
\[
\frac {\partial} {\partial t} \int_{\R^3} F_\pm dv + \nabla_x \cdot \int_{\R^3} v F_\pm dv =0,
\]
For simplicity, we assume the initial data $F_0 $ is normalized such that  the  equilibrium associated with the equation \eqref{VMB1} will be the standard Gaussian function, i.e.
\[
\label{DefM} \mu(v)\eqdefa (2\pi)^{-3/2} e^{-|v|^2/2}, 
\] which enjoys the same mass, momentum, and energy as $F_0$.
\smallskip

\noindent$\bullet$ {\bf Perturbation Equation.} In the perturbation framework, let $f=[f_+,f_-]$ be the perturbation such that
\[
F_\pm=\mu+f_\pm.
\]
Then the equation \eqref{VMB1} becomes
\begin{equation}\label{VMB3}
\partial_t f_\pm+v \cdot \nabla_x f_\pm \pm (E + v \times B) \cdot \nabla_v f_\pm \mp ( E  \cdot v)   \mu =Q(f_\pm+f_\mp,\mu)+Q(2\mu+f_\pm+f_\mp, f_\pm)
\end{equation}
where $E, B$ satisfies 
\begin{equation}\label{VMB4}
\partial_t E - \nabla_x \times B = -\int_{\R^3} v (f_+-f_{-}) dv:=G, \quad \partial_t B +  \nabla_x \cdot E =0,\quad
\nabla_x \cdot E = \int_{\R^3} (f_+-f_{-}) dv   , \quad  \nabla_x \cdot B =0,
\end{equation}
and we define the linearized operators $L=[L_+, L_{-}], \mathcal L=[\mathcal L_+, \mathcal L_{-}]$ by 
\[L_\pm f=2Q(\mu, f_\pm)+Q(f_++f_{\mp},\mu),\quad \mathcal Lf=-v\cdot\nabla_xf\mp E\cdot v\mu+Lf
\]
and the nonlinear collision operator $\Gamma=[\Gamma_+, \Gamma_{-}]$ by
\[
\Gamma_{\pm}(f,g)=Q(g_\pm+g_\mp, f_{\pm}).
\]
By assuming $F_0$ have the same mass, total momentum and total energy as $\mu$, the conservation law \eqref{conse1} turns to
\begin{equation}
\label{conse2}
\begin{aligned}
\int_{\T^3} \int_{ \R^3}  f_\pm dv dx  =0,\,\,\int_{\T^3} \int_{ \R^3} v (f_++f_{-})dv dx = 0,\\\int_{\T^3} \int_{ \R^3} |v|^2 (f_++f_{-}) dv dx  +  \int_{\T^3} |E (t, x)|^2  +|B(t, x)|^2 dx=0.
\end{aligned}
\end{equation}
Then the continuity equation is 
\[
\frac {\partial} {\partial t} \int_{\R^3}  f_\pm( v) dv + \nabla_x \cdot \int_{\R^3} v f_\pm(v)  dv =0.
\]
and the kernel of $L$ is 
\[
\ker(L) = \text{span} \{[1,0]\mu,[0,1]\mu,[1,1]v_1\mu,[1,1]v_2 \mu,[1,1] v_3 \mu,[1,1] |v|^2 \mu \}, 
\]
where $[\cdot,\cdot]$ is the vector in $\R^2$. 
Then we define the projection onto $\ker(L)$ by
\begin{align}\label{projection}\notag
	\PP f &= \left(\int_{\R^3}f_\pm dv\right) \mu + \sum_{i=1}^3\left(\int_{\R^3}v_i(f_++f_-) dv  \right)v_i\mu + \left(\int_{\R^3} \frac {|v|^2-3}{6}(f_++f_-)dv\right)(|v|^2 -3)\mu\\
 &=:\Big(a_\pm(t,x)+v\cdot b(t,x)+(|v|^2-3)c(t,x)\Big)\mu,
\end{align}
or equivalently by 
\[
\P f = \Big(a_+(t,x)[1,0]+a_-(t,x)[0,1]+v\cdot b(t,x)[1,1]+(|v|^2-3)c(t,x)[1,1]\Big)\mu. 
\]

\subsection{ Notations and function spaces} 

\noindent $\bullet$ Let the multi-indices $\alpha$ and $\beta$ be $\alpha=[\alpha_1,\alpha_2,\alpha_3]$, $\beta=[\beta_1,\beta_2,\beta_3]$ and $\partial^\alpha_{\beta}:=\partial^{\alpha_1}_{x_1}\partial^{\alpha_2}_{x_2}\partial^{\alpha_3}_{x_3}\partial^{\beta_1}_{v_1}\partial^{\beta_2}_{v_2}\partial^{\beta_3}_{v_3}$. If each component of $\theta$ is not greater than that of $\overline\theta$, we denote it by $\theta\le\overline\theta$; $\theta<\overline\theta$ means $\theta\le\overline\theta$, $|\theta|<|\overline\theta|$.\\
\noindent $\bullet$ We use $C$ to denote a universal constant that may change from line to line. The notation $a\lesssim b$ means $a\le Cb$, and $a\sim b$ means both $a\lesssim b$ and $b\lesssim a$ hold. We denote $C_{a_1,a_2,\cdots, a_n}$ by a constant depending on parameters $a_1,a_2,\cdots, a_n$. We also use the parameter $\epsilon$ to represent different positive numbers that are much smaller than 1 and determined in different cases.\\
\noindent $\bullet$ We define $\langle v\rangle :=(2+|v|^2)^{1/2}$, so $\log\langle v\rangle\ge \log2>0$. The lower bound of $\log\langle v\rangle$ will be used in the definition of the weight function.\\
\noindent $\bullet$ $(a(D)f)(v):= \frac{1} {(2\pi)^3}  \int_{\R^3}\int_{\R^3}e^{i(v-u) \xi }  a(\xi) f(u) du d\xi$ is the pseudo-differential operator with the symbol $a(\xi)$.\\
\noindent $\bullet$ For the linear operator $L$, we write $S_L$ as the semigroup generated by $L$.\\
\noindent $\bullet$  For $m, l\in\mathbb R$, we define the weighted Sobolev space $H_l^m$ by
$H^m_l:=\{f(v) {\|}f|_{ H^m_l}=|\langle \cdot\rangle^l\langle D \rangle^m  f|_{L^2}<+\infty\}$.
\noindent $\bullet$ The space $L\log L$ is defined as $L\log L:=\Big\{f(v)|\|f\|_{L\log L}=\int_{\R^3}|f|\log(1+|f|)dv\Big\}$.


\noindent $\bullet$ The norm $\|\cdot\|_{H^\alpha_xH^m_l}$ is defined as
$\|f\|_{H^\alpha_xH^m_l}:=\left(\int_{\T^3}\|\langle D_x\rangle ^\alpha f(x,\cdot)  \|^2_{H^m_l}dx\right)^{1/2}$,
and $H^0_xH^m_l=L^2_xH^m_l$. 

\noindent $\bullet$ For Landau equation, let $m:=\langle v\rangle^k$ for some $k>0$. We define the velocity space $H^1_*(m)$ with the norm
\begin{equation}
\label{landauhm}
{\|}f{\|}^2_{H^1_*(m)}:={\|}f{\|}^2_{L^2_v(m\langle v\rangle^{\gamma/2})}+{\|}P_v\nabla_vf{\|}^2_{L^2_v(m\langle v\rangle^{\gamma/2})}+{\|}(I-P_v)\nabla_vf{\|}^2_{L^2_v(m\langle v\rangle^{(\gamma+2)/2})}
\end{equation}
here $P_v$ is projection onto $v$, namely $P_v\xi=\left(\xi\cdot\frac{v}{|v|}\right)\frac{v}{|v|}$.

\noindent $\bullet$ 
To describe the polynomial perturbation with variable $v$, we define the weight  function $w(\alpha,\beta)$ as
\begin{equation}
\label{weightw}
w(|\alpha|, |\beta|)=\langle v\rangle ^{k-a|\alpha|-b|\beta|+c}, \quad b=\frac{\max\{-2\gamma, 0\}}{s}+5, \quad a=b+\min\{\gamma, 0\},\quad c=2b+6
\end{equation}

\noindent $\bullet$ Let $(f, g)$ be the inner product of $f, g$ in $v$ variable $(f, g)_{L^2_v}$ for short, And we will use $(f, g)_{L^2_{x, v}}$  if the integral is both in $x, v$. We use $(f, g)_{L^2_k}:=(f, g  \langle v \rangle^{2k})_{L^2_v}$, and 
$\|f(\theta)\|_{L^1_\theta} = \int_{\S^2}f(\theta) d\sigma = 2\pi\int_0^{\pi}   f(\theta)  \sin\theta  d\theta$.

\noindent $\bullet$ We define vector $\bm g$ as $\bm g:=(f_+, f_{-}, E,B)$.\\
\noindent $\bullet$ We define the instant energy functional $E_{K,k}$ and dissipation rate functional $D_{K,k}$ respectively by
\begin{align}
\label{functional1}
\begin{aligned}
    E_{K, k}(f,E,B)(t)&:=E^f_{K,k}(t)+E^{eb}_{K}(t),\,\,\text{where}\,\,E^f_{K,k}(t):=\sum_{|\alpha| +|\beta| \le K} \Vert C_{|\alpha|, |\beta|}\partial^\alpha_\beta f w(\alpha, \beta) \Vert_{L^2_{x, v}}^2 , \,\, E^{eb}_{K}(t):=\Vert [E, B] \Vert_{H^K_x}^2,\\
    D_{K, k}(f,E,B)(t)&: = D^f_{K,k}(t)+E^{eb}_{K-1}(t),\,\,\text{where}\,\,D^f_{K,k}(t):=\sum_{|\alpha| +|\beta| \le K} \Vert C_{|\alpha|, |\beta|} \partial^\alpha_\beta f w(\alpha, \beta) \Vert_{L^2_{x}H^s_{\gamma/2}}^2 
\end{aligned}
\end{align}
where $C_{|\alpha|, |\beta|}>0$ satisfy
$C_{|a|,|b| } \gg C_{|a|, |c|}$ if $|a| < |b|$, $C_{|a| +1,|b|-1 } \gg C_{|a|, |b|}$ for any $a, b\in \mathbb{N}^+$. 
Moreover, we write $F_K, G_K$ as the functionals without weight ($k=0$):  
\begin{align}\label{functional2}
    \begin{aligned}
        F^f_{K, k}(f)(t)&: =\sum_{|\alpha| +|\beta| \le K} \Vert C_{|\alpha|, |\beta|}\partial^\alpha_\beta f w(\alpha, \beta) \Vert_{L^2_{x, v}}^2,\\
 G_{K}(f,E,B)(t)&: = G^{f}_{K}(t)+ E^{eb}_{K}(t),\,\,\text{where}\,\, G^f_{K}(t):=\sum_{|\alpha| +|\beta| \le K} \Vert C_{|\alpha|, |\beta|} \partial^\alpha_\beta f w(\alpha, \beta) \Vert_{L^2_{x}H^s_{\gamma/2}}^2 
    \end{aligned}
\end{align}
We always omit $f, E, B$ of the functionals for simplicity, unless confused with other functions. Sometimes we will also use $D^{\bm g}, E^{\bm g}, F^{\bm g}, G^{\bm g}$ for the functionals.

\subsection{Main results} Our results can be stated as follows.
\begin{thm}
\label{maintheorem}
For the Cauchy problem of the Vlasov-Maxwell-Boltzmann/Landau system, with the forms in \eqref{VMB3} and \eqref{VMB4}. For any constant $k$ large enough, there exists a constant $\varepsilon_0>0$ small enough, such that if the initial data $f_0$ satisfies $F_0=\mu+f_0\ge 0$, the conservation laws hold and $E_{6,k}(0)\le \varepsilon_0, E_{7,k}(0)< \infty$, then \eqref{VMB3},\eqref{VMB4} has a unique solution $f$ that satisfies $F=\mu+f\ge 0$, and there exists a constant $\varepsilon>0$, such that $E_{6,k}(t)\le \varepsilon$ for any $t>0$. Moreover, we have the large-time asymptotic behavior as follows:
\begin{equation}
\label{maintheorem1}
E_{2,k}(t)\lesssim (1+t)^{-2}.
\end{equation}
\end{thm}
\subsection{Strategies of the proof}
In this subsection, we introduce the literature and strategies of our main proof.

We initiate our exploration by revisiting established findings within the context of the Landau and Boltzmann equations, focusing specifically on aspects pivotal to this paper—the global existence and long-term behavior of solutions to spatially inhomogeneous equations within the perturbation framework. To illuminate global solutions of the renormalized equation, particularly in the presence of substantial initial data, we draw attention to seminal works such as \cite{DL, DL2, L, V2, V3, DV, AV}.

For insights into the smoothing effects in the Boltzmann equation without cut-off, we turn to \cite{CH, CH2}.

These references serve as the foundation upon which we construct our exploration of spatially inhomogeneous equations and perturbation techniques to address the questions at hand.

\smallskip
In the realm of the non-cutoff Boltzmann equation, prior research efforts, exemplified by \cite{IMS, IMS2, IS, IS2, IS3, S}, have successfully established global regularity and elucidated long-term behavior under remarkably broad assumptions. These efforts demonstrate that global regularity and long-term behavior can be established by assuming uniform bounds in both time ($t$) and space ($x$), characterized by:

\begin{align*}
0<m_0 \le M(t, x) \le M_0, \quad E(t, x) \le E_0, \quad  H(t, x) \le H_0, 
\end{align*}
where $m_0$, $M_0$, $E_0$, and $H_0$ are positive constants. The functions $M(t, x)$, $E(t, x)$, and $H(t, x)$ are defined as:
\begin{align*}
M(t, x) =\int_{\R^3} f(t, x, v) dv, \, E(t, x) = \int_{\R^3} f(t, x, v) |v|^2 dv, \,  H(t, x) = \int_{\R^3} f(t, x, v) \ln f(t, x ,v) dv.
\end{align*}
further results regarding the non-cutoff Boltzmann equation can be found in \cite{GS, GS2, AMUXY, AMUXY2, AMUXY3, AMUXY4}. See also \cite{DLSS} for recent progress. Additionally, concerning global existence in bounded domains, we refer to \cite{Guo2009, Guo2016, Guo2020}. In particular, the case of the union of cubes is studied in \cite{Deng2021}.

Furthermore, for the Landau equation, local Hölder estimates are established in \cite{GIMV}, while higher regularity properties of solutions are explored in \cite{HS} through the application of kinetic variants of Schauder estimates.

An important facet of studying the Boltzmann/Landau equation is the perturbation framework near Maxwellian. Results pertaining to the cutoff Boltzmann equation can be found in \cite{G2, G3, SG, SG2}, and for non-cutoff Boltzmann equation, we refer to \cite{GS, GS2}. The results concerning the Landau equation can be found in \cite{G}. Importantly, all these works are grounded in the following decomposition:
\begin{align*}
\partial_t f  + v \cdot \nabla_x f =L_{\mu} f + \Gamma(f, f), \quad L_\mu f =\frac 1 {\sqrt{\mu}}  Q(\sqrt{\mu} f, \mu  ) + \frac 1 {\sqrt{\mu}}  Q(\mu, \sqrt{\mu}  f  ) , 
\end{align*}
where
\begin{align*}
\quad \Gamma(g, f) = \frac 1 2 \frac 1 {\sqrt{\mu}} Q(\sqrt{\mu} g, \sqrt{\mu} f) + \frac 1 2\frac 1 {\sqrt{\mu}}  Q(\sqrt{\mu} f, \sqrt{\mu} g), 
\end{align*}
signifying that the solution is constructed within a $\mu^{-1/2}$ weighted space.

Now we present key findings related to the Vlasov-Maxwell-Boltzmann/Landau (VMB/VML) equation. The groundbreaking work of Y. Guo, documented in \cite{G7}, established the global existence and regularity of solutions for the VMB equation on the torus $\mathbb{T}^3$ with spatial variable $x\in\mathbb T^3$. Moving beyond confined domains, R. Strain, in \cite{S2}, demonstrated the existence of global classical solutions over the entire space for the VMB equation. Furthermore, D. Duan and R. Strain, as detailed in \cite{DS}, provided insights into the optimal large-time behavior of these solutions.

Addressing specific cases, R. Duan et al. explored the cutoff VMB equation for $-3<\gamma<-1$ when $x$ spans the entirety of $\mathbb R^3$, as outlined in \cite{DLYZ}. Investigating scenarios without angular cutoff, \cite{DLYZ2} delved into the VMB equation with conditions $\gamma+2s<0$ and $s>1/2$.

Turning attention to the VML system, H. Yu's work in \cite{Y} presented significant results on global classical solutions near Maxwellian. Additional insights can be found in \cite{Wang}. Recent advancements, particularly those related to the Coulomb potential and strong background magnetic fields in the VML system, are documented in \cite{Fan}. These contributions collectively advance our understanding of the VMB and VML equations across diverse spatial configurations and highlight recent progress in addressing complex physical scenarios.

\smallskip
Furthermore, the Vlasov-Poisson-Boltzmann/Landau (VPB/VPL) equation emerges as a simplified instance of the VMB/VML equation when the magnetic field parameter $B$ is set to zero. Investigating perturbations around Maxwellians within the VPB equation, Y. Guo, in \cite{G6}, established the first results in the context of the cutoff hard sphere case. Diverse scenarios under the cutoff assumption have since been explored, with additional studies available in \cite{DYZ, DYZ2}.

For the VPB equation without angular cutoff, comprehensive results are documented in \cite{DL9, XXZ}. Notably, investigations into the regularizing effect are presented in \cite{Deng2021b}. The implications of bounded domains are discussed in works such as \cite{Cao2019, Dong2020, Deng2021e}.

Turning to the VPL equation, pioneering results were achieved in \cite{G8} for the torus case and further developments unfolded in \cite{SZ, W} for the scenario spanning the entire space.  It is essential to highlight that all the aforementioned results pertaining to the VPB/VPL equation are grounded in a common decomposition framework, which is expounded upon in the subsequent sections:
\begin{align*}
\partial_t f_\pm  + v \cdot \nabla_x f_\pm  \mp \nabla_x \phi \cdot \nabla_v f_\pm \pm \nabla_x \phi \cdot v \sqrt{\mu} \pm \frac 1 2 \nabla_x \phi \cdot v f =L_{\mu,\pm} f + \Gamma_\pm(f, f), 
\end{align*}
where
\begin{align}\label{L mu}
L_{\mu,_\pm} f = \frac 1 {\sqrt{\mu}}  Q(\sqrt{\mu} (f_\pm+f_\mp), \mu  ) + 2\frac 1 {\sqrt{\mu}} Q(\mu, \sqrt{\mu}f_\pm  ),  \,\, 
\end{align}
and
\[
\Gamma_\pm(f, f) = \frac 1 {\sqrt{\mu}}   Q(\sqrt{\mu} f_\pm, \sqrt{\mu} f_\pm) +\frac 1 {\sqrt{\mu}}  Q(\sqrt{\mu} f_\mp, \sqrt{\mu}f_\pm),
\]
demonstrating that these studies are conducted within a $\mu^{-1/2}$ weighted space.

\smallskip

\smallskip
An alternative avenue in perturbation studies near Maxwellian involves polynomial weight perturbation, with the results grounded in the application of the semigroup method. This method, originally introduced by Gualdani-Mischler-Mouhot in \cite{GMM}, initially focused on the investigation of the cutoff Boltzmann equation featuring a hard potential. The authors successfully established results concerning global existence and large-time behavior. Subsequently, this method found application in proving results for the Landau equation, as seen in \cite{CTW, CM}.

In the realm of non-cutoff Boltzmann equations, advancements in the context of a hard potential scenario are detailed in \cite{HTT, AMSY}, while the treatment of the soft potential case is addressed in \cite{CHJ}. Our current work builds upon this methodology, offering an extension, particularly concerning the macroscopic component $\mathcal{P}f$. The foundational principles of this extension were introduced in \cite{CDL}, where the primary idea is revisited for the reader's convenience.

To illustrate the application of this extended methodology, we consider the non-cutoff Boltzmann equation in the case where $\gamma = 0$ and within the functional space $H^2_xL^2_k$. Drawing inspiration from \cite{MS}, we establish a crucial inequality as follows:
\begin{align}\label{esLmu}
\sum_{\pm}(\L_\pm f, f_\pm)_{H^2_xL^2 (\mu^{-1/2})} \le -\lambda \Vert f \Vert_{H^2_x L^2 (\mu^{-1/2})}, \quad\text{ if } \P f =0, 
\end{align}
where $\lambda>0$ serves as a crucial constant. Together with the macroscopic estimates from \cite{Deng2021e, GS} for $\P f$, we can deduce that  $\Vert S_{\L}(t) f_0 \Vert_{H^2_x L^2 (\mu^{-1/2})} \le e^{-\lambda t} \Vert f_0 \Vert_{H^2_x L^2 (\mu^{-1/2})}^2$. Next, for some $M, R>0$, define 
\begin{align*}
	&A_\pm = -v\cdot\na_x+L_\pm  -M\chi_R,\quad
	K_1  = M\chi_R, \quad B_\pm = -v\cdot\na_x+L_\pm,\quad
	K_2 = \pm\mu v\cdot\na_x\phi,
\end{align*}
where $\chi \in D(\R)$ is the truncation function  satisfying $1_{[-1,1]} \le \chi \le 1_{[-2,2]}$ and we denote $\chi_R(\cdot) := \chi(\cdot/R)$ for $R > 0$.
By the results in \cite{CHJ} we have
\begin{align*}
\sum_{\pm}(L_\pm f, f_\pm)_{H^2_x L^2_k} \le -C \Vert f \Vert_{H^2_x H^s_k}^2 +  C_k \Vert f \Vert_{H^2_x L^2_v}^2,
\end{align*}
Take $M, R>0$ large, we have $(Af, f)_{H^2_x L^2_k} \le -C \Vert f \Vert_{H^2_x L^2_k}^2$, which implies 
$\Vert S_A(t) f \Vert_{H^2_x L^2_k} \le e^{-\lambda t} \Vert f_0 \Vert_{H^2_x L^2_k }^2$. By Duhamel's formula $S_B =S_A +S_B * K_1 S_A$ and $S_\L =S_B +S_\L * K_2 S_B$, roughly speaking, we have
\begin{multline*}
\Vert S_{B}(t) \Vert_{H^2_x L^2_{k} \to H^2_x L^2_k }   \le \Vert S_A(t) \Vert_{H^2_x L^2_k \to H^2_x L^2_k} + \int_0^t \Vert S_B(s)  \Vert_{H^2_x L^2 (\mu^{-1/2}) \to H^2_x L^2 (\mu^{-1/2})} \\\times\Vert K_1 \Vert_{H^2_x L^2_k \to H^2_x L^2(\mu^{-1/2} ) }\Vert S_A(t-s)\Vert_{H^2_x L^2_k \to H^2_x L^2_k}\,ds \le C e^{-\lambda t}. 
\end{multline*}
and hence
\begin{multline*}
\Vert S_{\L}(t) \Vert_{H^2_x L^2_{k} \to H^2_x  L^2_k }   \le \Vert S_B(t) \Vert_{H^2_x L^2_k \to H^2_x L^2_k} + \int_0^t \Vert S_L(s)  \Vert_{H^2_x L^2 (\mu^{-1/2}) \to H^2_x L^2 (\mu^{-1/2})} \\\times\Vert K_2 \Vert_{H^2_x L^2_k \to H^2_x L^2(\mu^{-1/2} ) }\Vert S_B(t-s)\Vert_{H^2_x L^2_k \to H^2_x L^2_k}\,ds \le C e^{-\lambda t}. 
\end{multline*}
Thus, the rate of convergence for the linear operator $\L$ is established. To estimate the nonlinear part, we need to define a scalar product by
\[
((f, g))_k := (f, g)_{L^2_k} + \eta\int_0^{+\infty} (S_\L(\tau )f, S_\L(\tau) g )_{L^2_v} d\tau.
\] 
Due to the fact that
\begin{align*}
\int_0^{+\infty} (S_{\L}(\tau) \L f, S_{\L}(\tau) f ) d\tau = \int_0^{+\infty} \frac d {d\tau} \Vert S_{\L}(\tau) f \Vert_{L^2}^2 d\tau 
= - \Vert f \Vert_{L^2}^2, 
\end{align*}
so after choosing  proper $\eta$, we deduce that
\[
((\L f, f))_k = ( \L f, f)_{L^2_k} + \eta\int_0^\infty (S_\L (\tau ) L f, S_\L(\tau) f)_{L^2_v} d\tau \sim \|f\|_{H^s_{k}}^2.
\]
Therefore, it is important to note that the linear operator $\L$ can exhibit non-negativity properties within an appropriate function space. The estimate of $\L$ within this specific function space empowers us to establish global well-posedness by combining the nonlinear estimates.

In the study presented in \cite{CDL}, the amalgamation of findings from the VPB/VPL system and the application of the semigroup method enabled us to establish the global well-posedness and elucidate the long-term behavior of the VPL/VPB system. Building upon this achievement, the objective of this paper is to extend our previous results to the VMB/VML system. However, a fundamental distinction arises between the Poisson equation and the Maxwell equation. Specifically, in the Maxwell equation, there is dissipation of the electric and magnetic fields, denoted as $E$ and $B$ respectively. Consequently, when computing the time derivative of the energy of the system $(f, E, B)$, the order of the $x$ derivative in $E$ and $B$ is consistently one less than that in $f$.

For instance, regarding the macroscopic estimate, the Vlasov-Poisson-Boltzmann equation exhibits an inequality of the form:
\[
\partial_t({\|}f{\|}_{H^2_xL^2_v}+{\|}\nabla_x\phi{\|}_{H^2_x})+\lambda{\|}\textbf Pf{\|}^2_{H^2_xL^2_V}+\lambda{\|}\nabla_x\phi{\|}^2_{H^2_x}\lesssim {\|}(\textbf I-\textbf P)f{\|}^2_{H^2_xL^2_{10}}+{\|}\nabla_x\phi{\|}^4_{L^2_x}
\]
In contrast, for the Vlasov-Maxwell-Boltzmann equation, the corresponding macroscopic estimate becomes:
\[
\partial_t({\|}f{\|}_{H^2_xL^2_v}+{\|}E,B{\|}_{H^2_x})+\lambda{\|}\textbf Pf{\|}^2_{H^2_x}++\lambda{\|}E,B{\|}^2_{H^1_x}\lesssim {\|}(\textbf I-\textbf P)f{\|}^2_{H^2_xL^2_{10}}+{\|}E,B{\|}^4_{L^2_x}
\]
thus, when exclusively examining the evolution behavior of the coordinate $f$, there is no loss of derivatives in the spatial variable $x$, and the outcome closely resembles that of the Poisson equation. However, when considering the collective evolution behavior of the entire system $(f, E, B)$, meticulous adjustment of the $x$ derivatives becomes imperative due to the dissipation inherent in $E$ and $B$. Indeed, our main theorem illustrates that achieving global existence necessitates a sacrifice in regularity of order 1. Additionally, to prove polynomial decay, further losses in $x$ regularity are inevitable. It is noteworthy to highlight another distinction from the VPB/VPL system: in contrast, the large-time asymptotic behavior of the system under consideration consistently exhibits polynomial decay, even in cases involving hard potential, as opposed to the exponential decay observed in the VPB/VPL system.

Another noteworthy aspect pertains to the choice of the weight function. It is essential to recall that the polynomial perturbation can be viewed as a change of variable, given by $F=\mu + \langle v\rangle^kf_\pm$, where $k>0$ is a constant. In the Vlasov-Poisson-Boltzmann (VPB) equation, when addressing the additional term $\pm k\nabla_x\phi\cdot\frac{v}{\langle v\rangle}f$, one can adopt the approach presented in \cite{G8}, \cite{CDL}, introducing the corresponding weight function $e^{\frac{\pm k\phi}{\langle v\rangle}}$ to absorb the extra term. However, in the In the Vlasov-Maxwell-Boltzmann (VMB)  equation, due to the dissipation of $E$ and $B$, this method is no longer effective. Consequently, we opt for the careful selection of the weight function $\langle v\rangle^{(k,\alpha,\beta)}$, and in the case of a soft potential, the weight function even needs to be time-dependent, a detail that will be elaborated upon later.

\subsubsection{About the choice of the weight function for soft potential case $\gamma+2s<0$}
When $\gamma+2s<0$, the challenge arises from the omission of the term equivalent to ${\|}f{\|}_{L^2_v}$ in the computation of $\partial_t E_{K,k}$. This omission is attributed to the inequality ${\|}f{\|}_{H^s_{\gamma/2}}\lesssim {\|}f{\|}_{L^2_v}$. To elaborate, the inherent property ${\|}f{\|}_{H^s_{\gamma/2}}\lesssim {\|}f{\|}_{L^2_v}$ leads to the absence of the crucial ${\|}f{\|}_{L^2_v}$ term in the calculation of $\partial_t E_{K,k}$. This omission introduces additional complexities into the proof. To overcome this challenge, we propose a solution by introducing a time-dependent weight function. Specifically, for any $k$ and $\delta>0$, we define $\mathsf w(v,t):=\langle v\rangle v^{k(1+(1+t)^{-\delta})}$. It is noteworthy that, due to our choice of $\langle v\rangle$, we ensure $\log \langle v\rangle \gtrsim \log 2>0$. Through direct calculation, we establish that \[\frac d{dt}\mathsf w^2=-\frac{2\delta k}{(1+t)^{1+\delta}}\log \langle v\rangle \mathsf w^2\lesssim -(1+t)^{-1-\delta}\mathsf w^2.\]
This modification introduces an extra term $(1+t)^{-1-\delta}F_{K,k}$ in the analysis of the convergence of the energy $E_{K,k}$, allowing us to demonstrate the polynomial decay of $E_{K,k}$. As we will illustrate later in the proof, the choice of $\delta=\frac{1}{2}$ proves to be adequate for this purpose.
\subsubsection{About the assumption $s\ge 1/2$}  Similar to the Poisson case, the assumption $s \geq 1/2$ proves to be essential. Let $\bm {g}=(f_{\pm},E,B)$. We need to study the following linear operator
\begin{equation}
    \label{L1}
    L_1 (\bm{g}) = (-v\cdot \nabla_x f_{\pm} \pm E \cdot v \mu +L_\pm f, \nabla_x \times B -\int_{\R^3} v (f_+-f_{-}) dv , -\nabla_x \times E ).
\end{equation}
 In a broad sense, when employing the semigroup method to analyze the Vlasov-Maxwell term $(E+v\times B) \cdot \nabla_v f$, it becomes imperative to manage the term
\[
\int_0^\infty (S_{L_1}(\tau)[(E+v\times B)\cdot\nabla_vf], S_{L_1}f) d\tau
\]
given the non-commutative nature of the operators $S_{L_1}(t)$ and $\nabla_v$, integration by parts with respect to $\nabla_v$ is not permissible. Therefore, we must rely on an upper bound to restrict this term. Specifically, the control of the term ${\|}E\cdot\nabla_vf{\|}_{H^{-s}_{12}}\sim{\|}E\cdot f{\|}_{H^{1-s}_{12}}$ by ${\|}E{\|}_{H^2_x}{\|}f{\|}_{H^s_{12}}$ necessitates the condition $1-s\leq s$, which translates to $s \geq 1/2$.

The exploration of the Boltzmann/Landau equation extends beyond the perturbation analysis based on the Maxwellian distribution. Another intriguing avenue of research revolves around the global existence and regularity of the solution in the vicinity of the vacuum state, denoted as $f_\text{vac}=0$. In contrast to the perturbation analysis centered on the Maxwellian distribution, investigations into perturbations near the vacuum state are notably scarce.

In the context of the Landau equation, J. Luk established results for moderately soft potentials within the range $-2<\gamma<0$ in \cite{L2}. S. Chaturvedi, in a separate work \cite{Cha1}, demonstrated the existence of solutions for hard potentials within the interval $0\leq\gamma<1$. Addressing the non-cutoff Boltzmann equation, S. Chaturvedi extended the analysis to moderately soft potentials with $0<\gamma+2s<2$ in \cite{C2}. Y. Guo, in \cite{G5}, successfully tackled the Vlasov-Poisson-Boltzmann equation with cutoff. Notwithstanding these achievements, the question of stability near vacuum for the Boltzmann/Landau equation, particularly with very soft potentials, remains an unresolved challenge.

However, by using weighted energy estimates and carefully studying null structure in the nonlinearity of the Boltzmann/Landau equation, it is reasonable to believe that the global stability and regularity of the non-cutoff Vlasov-Poisson-Boltzmann equation and the Vlasov-Poisson-Landau equation with moderately soft potentials can be established. These analytical tools will contribute to advancing our understanding of stability phenomena in kinetic equations. We leave it as a future work.
\subsection{Structure of the paper}
In Section \ref{sec2}, we establish estimates for the Boltzmann/Landau collision operator and the Vlasov-Maxwell term, incorporating polynomial weighting. Section \eqref{sec3} provides an analysis of the macroscopic dissipation in the context of the VMB/VML equation. Moving on to Section \ref{sec4}, we present estimates pertaining to the semigroup generated by the linearized VMB/VML operator defined in \eqref{L1}. Finally, in Section \ref{sec5}, we demonstrate the proof of the main theorem, highlighting global existence and convergence towards the Maxwellian with polynomial decay.

\section{Basic Estimates }
\label{sec2}

We first recall some basic results about the non-cutoff Boltzmann and Landau collision operators.

\subsection{Landau collision operator} 
First, we consider the basic estimates Landau collision operator, which can be found in \cite{CTW, CM}.

\begin{lem}[\cite{CTW, CM,CDL}]
\label{L51}
 For any $-3 \le \gamma \le 1$ and $k \ge 7$, denote $m = \langle v \rangle^k$, and remind the space $H^1_*(m)$ defined in \eqref{landauhm} Then for any smooth functions $f, g$ and $|\beta| \le 2$, we have 
\[
(Q(f, \partial_\beta \mu),  g \langle v \rangle^{2k} ) \le C_k \Vert f  \Vert_{L^2(\langle v \rangle^{5})} \Vert g \Vert_{L^2(\langle v \rangle^{5})},
\]
and 
\[
(Q(\mu , f) , f \langle v \rangle^{2k}) \le -\gamma_1 \Vert f \Vert_{H^1_*(m)}^2 + C_k \Vert f \Vert_{L^2}^2,
\]
for some constant $\gamma_1, C_k>0$.
For any smooth functions $f, g, h$, we have 
\[
(Q(f, g) , h \langle v \rangle^{2k} ) \le C \Vert f \Vert_{L^2_7} \Vert g \Vert_{H^1_{*, k+1} }\Vert h \Vert_{H^1_{*, k} }. 
\]
If $h=g$, then we have a better estimate
\[
(Q(f, g) , g \langle v \rangle^{2k} ) \le C \Vert f \Vert_{L^2_7} \Vert g \Vert_{H^1_{*, k} }^2.
\]
Moreover, to put all the derivatives on $h$, we have 
\[
(Q(f, g), h) \le \Vert f \Vert_{L^2_5}\Vert g \Vert_{L^2} \Vert h \Vert_{H^2_{5}}, \quad (Q(f, g), h) \le \Vert f \Vert_{L^2_5}\Vert g \Vert_{L^2_{10}} \Vert h \Vert_{H^2_{-5}}. 
\]
\end{lem}
\begin{proof}
The first two estimates follow from \cite[Lemma 2.3 and Lemma 2.5]{CM} and \cite[Lemma 2.7, Lemma 2.12]{CTW} with $\mu$ replacing by $\partial_\beta \mu$ if needed. 
The third and fourth statements follow from \cite[Lemma 3.5]{CTW} and \cite[Lemma 4.3]{CM}, but see \cite[Lemma 2.12]{CDL} for a modified proof. The last two estimates can be found in \cite[Lemma 2.13]{CDL}. 
\end{proof}

%
%

In \cite{CHJ}, the authors only need to compute the $x$ derivative term, but for the Vlasov-Maxwell-Boltzmann/Landau system, we also need estimates on the $v$ derivative term.
\begin{lem}\label{L29} Let $\gamma \in (-3, 1]$. For any $|\alpha| \ge 0,$ $ k \ge 10$, there exist constants $\gamma_3, C_k >0$ such that for any smooth functions $f, g$, we have
\begin{equation*}
\begin{aligned} 
|(\partial^\alpha Q (f, \mu),  \partial^\alpha g w^2(\alpha, \beta)  )_{L^2_{x, v} }   | \le C_k \Vert \partial^{\alpha} f\Vert_{L^2_x L^2_{5}}\Vert \partial^\alpha g \Vert_{L^2_x L^2_{5}} 
\end{aligned}
\end{equation*}
\end{lem}
\begin{proof}
The Landau case can be proved similarly by taking summation in Lemma \ref{L51}.
\end{proof}

We also recall some basic interpolation on $x$. 
\begin{lem}(\cite{CDL}, Lemma 2.3)
For any function $f$, $ k \in \R$, we have $\Vert f g \Vert_{H^2_x} \lesssim \min_{c + d = 2} \{ \Vert f \Vert_{H^c_x} \Vert g \Vert_{H^d_x}  \}$. More precisely, 
\begin{equation}
\begin{aligned}
\label{Linfty}
\Vert f \Vert_{L^\infty_x L^2_v} \lesssim \Vert f \langle v \rangle^{k} \Vert_{H^1_{x}L^2_v}+ \Vert f \langle v \rangle^{-k}  \Vert_{H^2_xL^2_v},
\end{aligned}
\end{equation}
\begin{equation}
\label{L3}
\Vert f  \Vert_{L^3_xL^2_v}  \lesssim\Vert f \langle v \rangle^{k} \Vert_{L^2_{x}L^2_v}+ \Vert f \langle v \rangle^{-k} \Vert_{H^1_xL^2_v},
\end{equation}
and for any constant $s \in (0, 1)$
\begin{equation}
\label{Ls}
\Vert f   \Vert_{H^1_v}
 \lesssim  \Vert f \langle v \rangle^{k}  \Vert_{H^s_v} + \Vert f \langle v \rangle^{ -k s /(1-s)} \Vert_{H^{1+s}_v}. 
\end{equation}
\end{lem}

\subsection{Boltzmann collision operator}
Here we recall the upper bound and the coercivity for the non-cutoff Boltzmann collision operator. 

\begin{lem}\label{L21}(\cite{H}, Lemma 1.1)
Let $w_1, w_2 \in \R$, $a, b \in[0, 2s]$ with $w_1+w_2 =\gamma+2s$  and $a+b =2s$. Then for any smooth functions $f, g, h$ we have \\
(1) if $\gamma + 2s >0$, then
\[
|(Q(g, h), f)_{L^2_v}| \lesssim (\Vert g \Vert_{L^1_{\gamma+2s +(-w_1)^+(-w_2)^+}}  +\Vert g \Vert_{L^2} ) \Vert h \Vert_{H^a_{w_1}}   \Vert f \Vert_{H^b_{w_2}}  ,
\]
(2) if $\gamma + 2s = 0$, then
\[
|(Q(g, h), f)_{L^2_v}|\lesssim (\Vert g \Vert_{L^1_{w_3}} + \Vert g \Vert_{L^2}) \Vert h \Vert_{H^a_{w_1}}   \Vert f \Vert_{H^b_{w_2}} , 
\]
where $w_3 = \max\{\delta,(-w_1)^+ +(-w_2)^+ \}$, with $\delta>0$ sufficiently small.\\
(3) if $-1< \gamma + 2s < 0$ we have
\[
|(Q(g, h), f)_{L^2_v}|\lesssim  (\Vert g \Vert_{L^1_{w_4}} + \Vert g \Vert_{L^2_{-(\gamma+2s)}}) \Vert h \Vert_{H^a_{w_1}}   \Vert f \Vert_{H^b_{w_2}}  
\]
where $w_4= \max\{-(\gamma+2s), \gamma+2s +(-w_1)^+ +(-w_2)^+ \}$.
\end{lem}
As a first application we have
\begin{cor}\label{C22}
Suppose that $\gamma > -3, \gamma +2s >-1$. For any multi-indices $|\beta | \le 2$, for any constant $k \ge 0$ and for any smooth functions $f, g$,  we have
\[
(Q(\partial_\beta \mu, f  ), g \langle  v \rangle^{2k} )\le C_k \Vert f \Vert_{H^s_{k + \gamma/2 + 2s}}  \Vert g \Vert_{H^s_{k+\gamma/2}}
\]
for some constant $C_k >0$.
\end{cor}

\begin{lem} \label{L23}(\cite{CHJ} Lemma 3.3)
Suppose that $-3 <\gamma \le 1$. For any $k \ge 14$, and smooth functions $g, h$, we have
\begin{equation*}
\begin{aligned}
|(Q (h, \mu), g \langle v \rangle^{2k}) | \le&  \gamma_3 \Vert b(\cos\theta) \sin^{k-\frac {3+\gamma} 2} \frac \theta 2 \Vert_{L^1_\theta}    \Vert h \Vert_{L^2_{k+\gamma/2}}\Vert g \Vert_{L^2_{k+\gamma/2}} + C_k \Vert h \Vert_{L^2_{k+\gamma/2-1/2}}\Vert g \Vert_{L^2_{k+\gamma/2-1/2}}
\\
\le& 2\gamma_3 \Vert b(\cos\theta) \sin^{k-\frac {3+\gamma} 2} \frac \theta 2 \Vert_{L^1_\theta}    \Vert h \Vert_{L^2_{k+\gamma/2}}\Vert g \Vert_{L^2_{k+\gamma/2}} + C_k \Vert h \Vert_{L^2}  \Vert g \Vert_{L^2}
\end{aligned}
\end{equation*}
for some constant $C_k>0$, where $\gamma_3$ is the constant such that 
\[
\int_{\R^3} \mu(v) |v-v_*|^\gamma dv \le \gamma_3 \langle v_* \rangle^{\gamma}
\]
holds. Moreover, for any $|\beta| \le 2$ we have
\[
|(Q (h, \partial_\beta\mu), g \langle v \rangle^{2k}) | \le  C_k \Vert h \Vert_{L^2_{k+\gamma/2}}\Vert g \Vert_{L^2_{k+\gamma/2}} 
\]
\end{lem}

\begin{rmk}
In \cite{CHJ} Lemma 3.3  the authors only prove the first statement of Lemma \ref{L23}, but the second statement can be proved the same way by just replacing $\mu$ by $\partial_\beta \mu$.
\end{rmk}
\begin{thm} \label{T25}(\cite{CHJ} Theorem 3.1)
Suppose that $-3<\gamma\le 1, \gamma+2s>-1$ and $G= \mu +g \ge 0$. Then if 
\[
G \ge 0,\quad  \Vert G \Vert_{L^1} \ge 1/2, \quad \Vert G \Vert_{L^1_2} +\Vert G \Vert_{L \log L} \le 4.
\]
we have
\begin{equation}
\label{ineq1}
\begin{aligned}
(Q(G, f), f \langle v \rangle^{2k} )\le&  - \frac {\gamma_2} {8} \Vert  b(\cos \theta) (1- \cos^{2k-3-\gamma} \frac \theta 2  )\Vert_{L^1_\theta}\Vert f \Vert_{L^2_{k+\gamma/2}}^2  - \gamma_1 \Vert f \Vert_{H^s_{k+\gamma/2}}^2 + C_k  \Vert f \Vert_{L^2_{k+\gamma/2-1/2}}^2
\\
&+C_k\Vert f \Vert_{L^2_{14}} \Vert g \Vert_{H^s_{ k+\gamma/2 }}\Vert f \Vert_{H^s_{ k + \gamma/2}} +C_k\Vert g \Vert_{L^2_{14} } \Vert f \Vert_{H^s_{ k + \gamma/2}}^2
\\
\le&  - \frac {\gamma_2} {12} \Vert  b(\cos \theta) (1- \cos^{2k-3-\gamma} \frac \theta 2  )\Vert_{L^1_\theta}\Vert f \Vert_{L^2_{k+\gamma/2}}^2  - \gamma_1 \Vert f \Vert_{H^s_{k+\gamma/2}}^2 + C_k  \Vert f \Vert_{L^2}^2
\\
&+C_k\Vert f \Vert_{L^2_{14}} \Vert g \Vert_{H^s_{ k+\gamma/2 }}\Vert f \Vert_{H^s_{ k + \gamma/2}} +C_k\Vert g \Vert_{L^2_{14} } \Vert f \Vert_{H^s_{ k + \gamma/2}}^2
\end{aligned}
\end{equation}
for some constants $\gamma_1, \gamma_2, C_k>0$, where $\gamma_2$ is the constant such that 
\[
\gamma_2 \langle v \rangle^\gamma \le \int_{\R^3} |v-v_*|^\gamma \mu_* dv_* .
\]
\end{thm}
\begin{lem} \label{L26}(\cite{CHJ} Lemma 3.4 + Lemma \ref{L21})
For any smooth functions $f, g, h$ and $k \ge 12$ , we have
\[
(Q(f, g), h \langle v \rangle^{2k} ) \lesssim \Vert f \Vert_{L^2_{14}} \Vert g \Vert_{H^s_{k+\gamma/2+2s }}  \Vert h \Vert_{H^s_{k+\gamma/2}} +  \Vert g \Vert_{L^2_{14}}  \Vert f \Vert_{H^s_{k+\gamma/2}}  \Vert h \Vert_{H^s_{k+\gamma/2}} .
\]
In particular by duality we have
\[
\Vert Q(f, g) \Vert_{H^{-s}_{k-\gamma/2}} \lesssim  \Vert f \Vert_{L^2_{14}} \Vert g \Vert_{H^s_{k+\gamma/2+2s }}  +  \Vert g \Vert_{L^2_{14}}  \Vert f \Vert_{H^s_{k+\gamma/2}}  
\]
\end{lem}

\subsection{Weight function and two collision operators}
We first talk on the choice of $w(\alpha, \beta)$, we have the following properties
\begin{lem}\label{L27}
Suppose $|\alpha|, |\beta|, |a|, |b|, |c|, |d|, k$ are non-negative integers, then $w(\alpha, \beta) $ satisfies the following properties:
\begin{equation}\label{equationw1}
w(|\alpha|, |\beta|) \langle v \rangle^{2s}\le  w(|\alpha|, |\beta_1|), \quad w(|\alpha|, |\beta|) \langle v \rangle^{5s}  \le  w(|\alpha_1|, |\beta|), \quad \forall |\alpha_1| < |\alpha|, |\beta_1 | < |\beta|
\end{equation}
and
\begin{equation}\label{equationw2}
w(|\alpha|, |\beta|) =\langle v \rangle^{\min \{ \gamma, 0 \} } w (|\alpha|+1, |\beta|-1 ), \quad \forall |\alpha| \ge 0, \quad \forall |\beta| \ge 1
\end{equation}
in particular, \eqref{equationw2} implies 
\begin{equation}\label{equationw3}
w(|\alpha|, |\beta|) \ge w(|\alpha| -k ,|\beta|  +  k ), \quad 0 \le k \le |\alpha|
\end{equation}
also we have
\begin{equation}\label{equationw4}
\langle v \rangle^{k+6} \le w(|\alpha|, |\beta|), \quad \forall |\alpha| + |\beta| \le 2 
\end{equation}
Finally we have
\begin{equation}\label{equationw5}
\langle v \rangle w(|\alpha|, |\beta|) \le  w(|\alpha|-1, |\beta|)^s w(|\alpha|-1,|\beta|+1)^{1-s} \langle v \rangle^{\gamma/2}, \quad \forall 1 \le  |\alpha| , \quad 0 \le  |\beta| 
\end{equation}
\end{lem}

\begin{proof}
\eqref{equationw1} is from the fact that
$a \ge \max \{-2\gamma, 0 \}+ 5s + \min \{\gamma, 0 \} \ge 5s  , \quad b \ge 5s$.\\
\eqref{equationw2} is just follow form the fact that
\[
k- a|\alpha|  - b|\beta| + c = k - a(|\alpha| + 1)  - b ( |\beta| - 1)  +  c  + \min \{ \gamma, 0\}  \quad \Longleftrightarrow  \quad  b -a +\min \{ \gamma, 0\} =0
\]
and we conclude from the definition of $a$. Then fourth statement just by
\[
\langle v \rangle^{k+6} \le w(0, 2) = \min_{|\alpha| +|\beta| \le 2} \{ w(|\alpha|, |\beta|) \}
\]
and the last statement is equivalent to 
\[
1+k- a|\alpha|  - b|\beta| + c \le (k - a(|\alpha| -  1)  - b  |\beta|   +  c )s  + (k - a(|\alpha| - 1)  - b ( |\beta| + 1)  +  c )(1-s) +\gamma/2
\]
which is equivalent to 
\[
1 \le a s + a (1-s) - b (1-s) +\gamma/2   \quad \Longleftrightarrow  \quad   1 \le b s + (a-b)  + \gamma/2  = b s  +  \min \{ \gamma, 0\}  + \gamma/2
\]
we easily conclude from the definition of $b$.
\end{proof}

\begin{lem}\label{L58}
For any $-3\le\gamma\le 1$, $Q$ as related Landau operator and $k \ge 10$ large. There exists a constant $C_k>0$, such that for any smooth function $f, g$:

(1) For any $|\alpha| + |\beta| \le 2$, 
\begin{equation*}
\begin{aligned} 
(\partial^\alpha Q ( \mu, f),  \partial^\alpha f w^2(\alpha, \beta) ) &\le  - \gamma_1 \Vert \partial^\alpha_\beta  f w(\alpha, \beta)\Vert_{L^2_x H^1_{*}}^2 + C_{k}  \Vert \partial^\alpha_\beta f \Vert_{L^2_x L^2_v}^2
\end{aligned}
\end{equation*}

(2)  For any $|\alpha| \ge 0, |\beta| \le 2, $
\[
| \partial^\alpha ( Q (g,  f), \partial^{\alpha}   f w^2 (\alpha, \beta))_{L^2_{x, v}} | \le C_k \Vert g \langle v \rangle^7 \Vert_{H^2_{x, v}} D^f_{2,k}.
\]
\end{lem}
\begin{proof}
Notice that
\[
(\partial^\alpha_\beta Q(\mu, f), \partial^\alpha_\beta f w^2(\alpha, \beta))_{L^2_{x, v}}= \sum_{\alpha_1 \le \alpha, \beta_1 \le \beta} ( Q(\partial_{\beta_1} \mu, \partial^{\alpha }_{\beta -\beta_1} f), \partial^{\alpha}_\beta f w^2( \alpha, \beta))_{L^2_{x, v}}.
\]
We split it into two cases $\beta_1 = 0$ and $|\beta_1|>0$. If $\beta_1  =0$ after integrating in $x$ in Lemma \ref{L51}, we have
\begin{equation}\label{l25extra1}
( Q( \mu, \partial^{\alpha }_\beta f), \partial^{\alpha}_\beta f w^2(\alpha, \beta))_{L^2_{x, v}}  \le  -\gamma_1 \Vert \partial^\alpha_\beta f w(\alpha, \beta)\Vert_{L^2_x H^1_{*}}^2 + C_{k}  \Vert \partial^\alpha_\beta f  \Vert_{L^2_x L^2_v }^2.
\end{equation}
For the case $|\beta_1| >0$, from the definition of $w$ we have $w( \alpha, \beta) \langle v \rangle \le w( \alpha, \beta - \beta_1)$. By Lemma  \ref{L51} we have
\begin{equation*}
\begin{aligned} 
|( Q(\partial_{\beta_1} \mu, \partial^\alpha_{\beta - \beta_1} f), \partial^{\alpha}_{\beta} f w^2 (\alpha, \beta))_{L^2_{x, v}} |\le& C_k \Vert \partial^\alpha_{\beta - \beta_1}  f w(\alpha, \beta) \Vert_{L^2_x H^1_{*, 1}}\Vert \partial^\alpha_\beta g w(\alpha, \beta) \Vert_{L^2_x H^1_{*}} 
\\
\le &  C_k \Vert \partial^\alpha_{\beta - \beta_1}  f w(\alpha, \beta_1 ) \Vert_{L^2_x H^1_{*}}  \Vert \partial^\alpha_\beta g w(\alpha, \beta) \Vert_{L^2_x H^1_{*}},
\end{aligned}
\end{equation*}
so (1) is thus finished by gathering the two terms together. Next, notice that
\[
( \partial^\alpha_{\beta } Q( g,  f), \partial^{\alpha}_{\beta} f w^2 ( \alpha, \beta))_{L^2_{x, v}}   =\sum_{\alpha_1 \le \alpha, \beta_1 \le \beta} ( Q(  \partial^{\alpha_1}_{\beta_1 }  g,  \partial^{\alpha -\alpha_1}_{\beta  -\beta_1}  f), \partial^{\alpha}_{\beta} f w^2 (\alpha, \beta))_{L^2_{x, v}}.
\]
We again split it into two cases $|\alpha_1| = |\beta_1| =0 $ and $|\alpha_1| +|\beta_1| >0$. For the case $|\alpha_1| = |\beta_1| =0 $, by Lemma \ref{L51} we have
\begin{equation*}
\begin{aligned}
|( Q( g, \partial^\alpha_{\beta } f), \partial^{\alpha}_{\beta} f w^2 (\alpha, \beta))_{L^2_{x, v}} |  
\le &  \int_{\T^3} C_k\Vert g \Vert_{L^2_7 } \Vert  \partial^\alpha_{\beta } f w(\alpha, \beta) \Vert_{H^1_{*}}^2 dx \\
\le  &C_k \Vert g \Vert_{H^2_x L^2_7 } \Vert  \partial^\alpha_{\beta } f w( \alpha, \beta) \Vert_{L^2_x H^1_{*}}^2 
\le C_k \Vert g \langle v \rangle^7 \Vert_{H^2_{x, v}} D^f_{2,k} ,
\end{aligned}
\end{equation*}
and  for all  $\alpha_1 \le \alpha, \beta_1 \le \beta, |\alpha_1| +|\beta_1| >0$, by Lemma \ref{L51} we have
\begin{equation*}
\begin{aligned} 
|( Q(\partial^{\alpha_1}_{\beta_1} g, \partial^{\alpha -\alpha_1}_{\beta - \beta_1} f), \partial^{\alpha}_{\beta} f w^2 (\alpha, \beta))_{L^2_{x, v}} |
\le& \int_{\T^3}  C_k\Vert \partial^{\alpha_1}_{\beta_1} g \Vert_{L^2_{7}} \Vert \partial^{\alpha -\alpha_1}_{\beta - \beta_1} f w( \alpha, \beta)\Vert_{H^1_{ *, 1 }}\Vert \partial^{\alpha}_{\beta}  f w(\alpha, \beta) \Vert_{H^1_{ * }} dx.
\end{aligned}
\end{equation*}
We split it into two cases,  $|\alpha_1| + |\beta_1| =1 $ or 2. For $|\alpha_1| + |\beta_1| =1 $,
by $\Vert fg\Vert_{L^2_x} \le \Vert  f \Vert_{L^6_x} \Vert g \Vert_{L^3_x}$, we have
\begin{equation*}
\int_{\T^3} \Vert \partial^{\alpha_1}_{\beta_1} g \Vert_{L^2_{7}} \Vert \partial^{\alpha - \alpha_1}_{\beta - \beta_1} f w(\alpha, \beta)\Vert_{H^1_{  *, 1}}\Vert \partial^{\alpha}_{\beta}  f w( \alpha, \beta) \Vert_{H^1_{ * }} dx
\lesssim \Vert \langle v \rangle^{7} g \Vert_{H^2_{x, v} } \Vert \partial^{\alpha - \alpha_1}_{\beta - \beta_1} f w( \alpha, \beta)\Vert_{L^3_x H^1_{ *, 1 }}  \sqrt{D^f_{2,k}}.
\end{equation*}
 We again split it into two parts  $|\alpha|+|\beta| = 1$ or 2. For the case  $|\alpha_1| + |\beta_1| = |\alpha|+|\beta| =1$ we have $|\alpha - \alpha_1| =|\beta - \beta_1| =0 $. By \eqref{L3}, $\Vert f w( \alpha, \beta)\Vert_{L^3_x H^1_{*, 1}}  \lesssim \Vert f w(0, 0)\Vert_{L^2_x H^1_{ * }}  + \Vert f w(1,0)\Vert_{H^1_x H^1_{  * }} \lesssim \sqrt{D^f_{2,k}}$.
For the case $|\alpha| + |\beta| =2$,  either $|\alpha - \alpha_1| =1, |\beta - \beta_1|=0$ or $|\alpha - \alpha_1| =0, |\beta - \beta_1| =1$. \\
For the first case, 
we have $\Vert \partial^{\alpha - \alpha_1} f w( \alpha, \beta)\Vert_{L^3_x H^1_{  *, 1 }}  \le  \Vert f w( 1, 0)\Vert_{H^1_x H^1_{ * }} + \Vert f w( 2 , 0)\Vert_{H^2_x H^1_{* }}  \lesssim \sqrt{D^f_{2,k}}$.\\
For the second case, $|\beta| \ge 1$. So
$
\Vert  \partial_{\beta - \beta_1}  f w(\alpha, \beta)\Vert_{L^3_x H^1_{ *, 1}} \le \Vert \nabla f w(0, 1)\Vert_{L^2_x H^{1}_{ *}} + \Vert \nabla f  w( 1 , 1)\Vert_{H^1_x H^{1}_{ * }} \lesssim \sqrt{D^f_{2,k}}
$ 
For the case $|\alpha_1|+|\beta_1| = 2$, we have $|\alpha| +|\beta|=2$, $|\alpha - \alpha_1| = |\beta - \beta_1|=0$. So
\begin{equation*}
\int_{\T^3} \Vert \partial^{\alpha_1}_{\beta_1} g \Vert_{L^2_{7}} \Vert \partial^{\alpha - \alpha_1}_{\beta - \beta_1} f w(\alpha, \beta)\Vert_{H^1_{ *, 1 }}\Vert \partial^{\alpha}_{\beta}  f w(\alpha, \beta) \Vert_{H^1_{ * }} dx
\lesssim \Vert \langle v \rangle^{7} g \Vert_{H^2_{x, v} } \Vert f w(\alpha, \beta)\Vert_{L^\infty_x H^1_{ *, 1 }}    \sqrt{D^f_{2,k}}, 
\end{equation*}
which is similar as \eqref{l25extra1}. We obtain again from \eqref{Linfty} that $\Vert  f w(\alpha, \beta)\Vert_{L^\infty_x H^1_{ *, 1}} \le \Vert f w(1, 0)\Vert_{H^1_x H^1_{ *}} + \Vert f w(2 , 0)\Vert_{H^2_x H^1_{ * }}\lesssim \sqrt{D^f_{2,k}}$ 
The proof is done after gathering the terms.
\end{proof}

\begin{lem}
\label{lembols}
(Lemma 2.10 of \cite{CDL}) Let $-3<\gamma\le 1, 0<s<1, \gamma+2s\ge -1$. For any functions $f,g$ and $k\ge 14$, we have
\begin{equation}
\label{equationbols1}
{\|}Q(f,g){\|}_{H^{-s}_{k-\gamma/2}}\lesssim {\|}f{\|}_{L^2_{14}}{\|}g{\|}_{H^s_{k+\gamma/2+2s}}+{\|}g{\|}_{L^2_{14}}{\|}f{\|}_{H^s_{k+\gamma/2}}
\end{equation}
moreover,  we have the following estimate of $\sum_{\pm}L_\pm$:
\begin{align}
    \label{esL}
    \sum_\pm(L_\pm f,\<v\>^{2k}f)_{L^2_v}\ge -\ga_1\|f\|_{H^s_{k+\ga/2}}^2+C_k\|\1_{\{|v|\le R\}}f\|_{L^2_{v}}^2. 
\end{align}

\end{lem}


\subsection{Electromanetic effect}
The next result is the estimate of the linear part of the Vlasov-Maxwell term. 
\begin{lem}\label{L211}
Let $f, E$ be any sufficiently smooth solution to equation \eqref{VMB3} and \eqref{VMB4}. There exists $C_k>0$, such that\\
\smallskip 
(1) (mixed derivatives) for any $|\alpha| + |\beta| \le K$, 
\begin{equation}
\label{L211eq1}
 |( \partial^\alpha_{\beta} (E\cdot v \mu), \partial^\alpha_{\beta}f_\pm w^2(\alpha, \beta) )_{L^2_{x, v}}| \le C_k \Vert E \Vert_{H^{K}_x }   \Vert f_\pm \Vert_{H^{K}_x L^2_{v}  }.
\end{equation}
(2) (spatial derivatives) for any $|\alpha| \le K $, we have
\begin{equation}
\label{L211eq2}
\sum_\pm ( \partial^\alpha (E\cdot v \mu), \partial^\alpha f_\pm w^2(\alpha, \beta) )_{L^2_{x, v}} - C_k \partial_t (\Vert \partial^\alpha E \Vert^2 +\Vert \partial^\alpha B \Vert^2)\le C_k \Vert E \Vert_{H^{K}_x }   \Vert f_\pm\Vert_{H^{K}_x L^2_{v}  }.
\end{equation}
\end{lem}
\begin{proof}
 From integration by parts, we have
\begin{align*}
 ( \partial^\alpha_{\beta} ( E  \cdot v \mu), \partial^\alpha_{\beta}f w^2(\alpha, \beta) )_{L^2_{x, v}}  &= ( \partial^\alpha E  \cdot \partial_\beta( v \mu) w^2(\alpha, \beta) , \partial^\alpha_{\beta}f )_{L^2_{x, v}}
\\&= ((-1)^{|\beta|}  \partial^\alpha E \cdot  \partial_\beta (\partial_\beta( v \mu) w^2(\alpha, \beta) ), \partial^\alpha f )_{L^2_{x, v}}
\end{align*}
by Fourier transform? and Cauchy-Schwartz inequality, it is easy to show that
\begin{equation*}
((-1)^{|\beta|} \partial^\alpha E \cdot  \partial_\beta (\partial_\beta( v \mu) w^2(\alpha, \beta) ), \partial^\alpha f_\pm )_{L^2_{x, v}} 
\le \Vert\partial^\alpha E \Vert_{L^{2}_x}  \left \Vert \left(\int_{\R^d}   \partial^\alpha f_\pm  \partial_\beta (\partial_\beta( v \mu) w^2(\alpha, \beta) ) dv \right) \right \Vert_{L^{2}_x}.
\end{equation*}
Since $\int_{\R^d}   \partial^\alpha f  \partial_\beta (\partial_\beta( v \mu) w^2(\alpha, \beta) ) dv \le C_k \Vert \partial^\alpha f \Vert_{L^2_v}$,
we have $\left \Vert \left(\int_{\R^d}   \partial^\alpha f  \partial_\beta (\partial_\beta( v \mu) w^2(\alpha, \beta) ) dv \right) \right \Vert_{L^{2}_x} \lesssim \Vert f \Vert_{H^{K}_x L^2_v}$ for any $|\alpha| \le K$, and the proof of \eqref{L211eq1} is finished. To prove \eqref{L211eq2}, we have
\[
 - ( \partial^\alpha (E\cdot v \mu), \partial^\alpha f w^2(\alpha, \beta) )_{L^2_{x, v}}  = -  ( \partial^\alpha (E\cdot v \mu), \partial^\alpha P f w^2(\alpha, \beta) )_{L^2_{x, v}} - ( \partial^\alpha (E\cdot v \mu), \partial^\alpha (I-P) f w^2(\alpha, \beta) )_{L^2_{x, v}} 
\]
recall the representation of $P$. Notice that after doing the scalar product with $v\cdot\mu$, only the term $b(t,x)v\cdot\mu$ in $P$ will not be 0. So
\[
\mp( \partial^\alpha (E\cdot v \mu), \partial^\alpha \bm P f_\pm w^2(\alpha, \beta) )_{L^2_{x, v}} = - \int_{\T^3}\partial^\alpha E \cdot  \left (\int_{\R^3} v \partial^\alpha (f_+-f_{-}) dv \right) dx \int_{\R^3} \frac {|v|^2} 3 w^2(\alpha, \beta) \mu^{2} dv
\]
and we recall that
\begin{align*}
\frac 1 2 \partial_t (\Vert \partial^\alpha E \Vert^2 +\Vert \partial^\alpha B \Vert^2) &= (\partial^\alpha E, \nabla_x \times \partial^\alpha B) - (\partial^\alpha E,  \int_{\R^3} v \partial^\alpha (f_+-f_{-}) dv  ) -(\partial^\alpha B, \nabla_x \times \partial^\alpha E ) \\
&= - (\partial^\alpha E,  \int_{\R^3} v \partial^\alpha (f_+-f_{-}) dv  )  =  - \int_{\T^3} \partial^\alpha E \cdot  \left (\int_{\R^3} v \partial^\alpha (f_+-f_{-}) dv \right) dx 
\end{align*}
thus there exists $C_k>0$, such that $ -  ( \partial^\alpha (E\cdot v \mu), \partial^\alpha P f w^2(\alpha, \beta) )_{L^2_{x, v}}  = C_k \partial_t (\Vert \partial^\alpha E \Vert^2 +\Vert \partial^\alpha B \Vert^2)$. So \eqref{L211eq2} is proved by noticing that 
$ |( \partial^\alpha (E\cdot v \mu), \partial^\alpha (\bm I-\bm P) f_\pm w^2(\alpha, \beta) )_{L^2_{x, v}} | \lesssim \Vert \partial^\alpha E \Vert_{L^2 } \Vert \partial^\alpha (\bm I-\bm P) f_\pm\Vert_{L^2_{x, v} }$.
\end{proof}
Then we come to the nonlinear term for the Vlasov-Maxwell term. 
\begin{lem}\label{L212}
Consider both Landau and Boltzmann case. for any given smooth functions $f, E$ and $|\alpha| + |\beta|   \le K$, we have\\
(1)For $K=2,3,4,5,6,$ we have\[
( \partial^\alpha_{\beta} (( E + v \times B)\cdot \nabla_v f), \partial^\alpha_{\beta}f w^2(\alpha, \beta) )_{L^2_{x, v}} \lesssim   \sqrt{E^{eb}_{2}  }  E^f_{K, k }   + \sqrt{E^{eb}_{6} }D_{K, k}
\]
(2)For $K=7,$ we have
\[
( \partial^\alpha_{\beta} (( E + v \times B)\cdot \nabla_v f), \partial^\alpha_{\beta}f w^2(\alpha, \beta) )_{L^2_{x, v}} \lesssim   \sqrt{E_{6, k} }  (\sqrt{E^{eb}_{7} } \sqrt{D_{7, k} }  +  D_{7, k}  )  +  \sqrt{E^{eb}_{2} } E^f_{7, k }  
\] 
\end{lem}
\begin{proof}
We have
\begin{multline*}( \partial^\alpha_{\beta} (( E + v \times B) \cdot \nabla_v f), \partial^\alpha_{\beta}f w^2(\alpha, \beta) )_{L^2_{x, v}}  = \underbrace{\sum_{\alpha_1 \le \alpha}   ( \partial^{\alpha_1} ( E + v \times B)  \cdot \nabla_v \partial_{\beta}^{\alpha-\alpha_1}f, \partial^\alpha_{\beta}f w^2(\alpha, \beta) )_{L^2_{x, v}}}_{\text{Term} 1}
\\
+\underbrace{ \sum_{\alpha_1 \le \alpha, \beta_1 \le \beta, |\beta_1|=1}   (  ( \partial_{\beta_1} v \times \partial^{\alpha_1}B)  \cdot \nabla_v \partial_{\beta-\beta_1}^{\alpha-\alpha_1}f, \partial^\alpha_{\beta}f w^2(\alpha, \beta) )_{L^2_{x, v}}}_{\text{Term} 2}.
\end{multline*}
We only give the estimate of Term 1 as Term 2 is more simple. First if $|\alpha_1| = 0$, then $\alpha-\alpha_1 = \alpha$. Notice that $\nabla_v \times v =0$, by integration by parts we have
\[
( (E + v \times B) \cdot \nabla_v \partial_{\beta}^{\alpha}f, \partial^\alpha_{\beta}f w^2(\alpha, \beta) )_{L^2_{x, v}}   =\frac 1 2 (   (E+ v \times B)\cdot \nabla_v (w^2(\alpha, \beta)) , (\partial^\alpha_{\beta}f)^2 )_{L^2_{x, v}} ,
\]
recall the form of $w(\alpha, \beta)$, there exists some $C_k>0$, such that $\nabla_v (w^2(\alpha, \beta))  = C_k  v {\langle v \rangle^{-2}} w^2(\alpha, \beta)$. Since $(v \times B )\cdot v =0$, we have
\begin{equation*}
\begin{aligned}
\frac 1 2 ( (E + v \times B)  \cdot \nabla_v (w^2(\alpha, \beta)) , (\partial^\alpha_{\beta}f)^2 )_{L^2_{x, v}}   =& \frac {C_k} 2 ( (   E + v \times B)  \cdot  \frac v {\langle v \rangle^2} w(\alpha, \beta) \partial^\alpha_{\beta}f , \partial^\alpha_{\beta}f w(\alpha, \beta) )_{L^2_{x, v}}  
\\
\lesssim&  \Vert E  \Vert_{H^2_x }\Vert \partial^\alpha_{\beta}f  w(\alpha, \beta) \Vert_{L^2_x L^2_v}^2\lesssim \sqrt{E^{eb}_{2}  }  E^f_{K, k }
\end{aligned}
\end{equation*}
For the case $|\alpha_1| >0$, we first consider  $K=2,3, 4, 5, 6$, suppose $|\alpha_1| \le  4$, we have
\begin{equation*}
\begin{aligned}
&( \partial^{\alpha_1} (E + v \times B)  \cdot \nabla_v \partial_{\beta}^{\alpha-\alpha_1}f, \partial^\alpha_{\beta}f w^2(\alpha, \beta) )_{L^2_{x, v}}\\& 
\lesssim \Vert \partial^{\alpha_1}[E, B]\Vert_{H^2_x} \Vert \langle v \rangle  \langle v \rangle^{-\gamma/2} \nabla_v \partial_{\beta}^{\alpha-\alpha_1}f w(\alpha, \beta )\Vert_{L^2_xL^2_v} \Vert \partial^\alpha_{\beta}  f   \langle v \rangle^{\gamma/2} w (\alpha, \beta)\Vert_{L^2_xL^2_v}
\\
&\lesssim  \Vert [E, B] \Vert_{H^6_x  }   \Vert \langle v \rangle  \langle v \rangle^{-\gamma/2} \nabla_v \partial_{\beta}^{\alpha-\alpha_1} f w(\alpha, \beta )\Vert_{L^2_xL^2_v} \sqrt{D_{K, k}} \lesssim   \sqrt{E^{eb}_{6} } D_{K, k } 
\end{aligned}
\end{equation*}
here we used \eqref{equationw5}. For the case $|\alpha_1|=5$, we have 
\begin{equation*}
\begin{aligned}
&( \partial^{\alpha_1} (E + v \times B)  \cdot \nabla_v \partial_{\beta}^{\alpha-\alpha_1}f, \partial^\alpha_{\beta}f w^2(\alpha, \beta) )_{L^2_{x, v}} 
\\
\lesssim &\Vert \partial^{\alpha_1}[E, B]\Vert_{H^1_x} \Vert \langle v \rangle  \langle v \rangle^{-\gamma/2} \nabla_v \partial_{\beta}^{\alpha-\alpha_1}f w(\alpha, \beta )\Vert_{H^1_xL^2_v} \Vert \partial^\alpha_{\beta}  f   \langle v \rangle^{\gamma/2} w (\alpha, \beta)\Vert_{L^2_xL^2_v}
\\
\lesssim & \Vert [E, B] \Vert_{H^6_x  }   \Vert \langle v \rangle  \langle v \rangle^{-\gamma/2} \nabla_v \partial_{\beta}^{\alpha-\alpha_1} f w(\alpha, \beta )\Vert_{H^1_xL^2_v} \sqrt{D_{K, k}} 
\lesssim  \sqrt{E^{eb}_{6} } D_{K, k } 
\end{aligned}
\end{equation*}
similarly for the case $|\alpha_1|=6$ we have
\begin{equation*}
\begin{aligned}
&( \partial^{\alpha_1} (E + v \times B)  \cdot \nabla_v \partial_{\beta}^{\alpha-\alpha_1}f, \partial^\alpha_{\beta}f w^2(\alpha, \beta) )_{L^2_{x, v}} 
\\
\lesssim &\Vert \partial^{\alpha_1}[E, B]\Vert_{L^2_x} \Vert \langle v \rangle  \langle v \rangle^{-\gamma/2} \nabla_v \partial_{\beta}^{\alpha-\alpha_1}f w(\alpha, \beta )\Vert_{H^2_xL^2_v} \Vert \partial^\alpha_{\beta}  f   \langle v \rangle^{\gamma/2} w (\alpha, \beta)\Vert_{L^2_xL^2_v}
\\
\lesssim & \Vert [E, B] \Vert_{H^6_x  }   \Vert \langle v \rangle  \langle v \rangle^{-\gamma/2} \nabla_v \partial_{\beta}^{\alpha-\alpha_1} f w(\alpha, \beta )\Vert_{H^2_xL^2_v} \sqrt{D_{K, k}} \lesssim  \sqrt{E^{eb}_{6} } D_{K, k } 
\end{aligned}
\end{equation*}
for the case $|\alpha_1| = 7, K=7$, this time we have $\alpha =\alpha_1, |\beta| =0$. We compute
\begin{equation}
\label{equationlem00}
\begin{aligned}
&( \partial^{\alpha_1} (E + v \times B)  \cdot \nabla_v \partial_{\beta}^{\alpha-\alpha_1}f, \partial^\alpha_{\beta}f w^2(\alpha, \beta) )_{L^2_xL^2_v}  
\\
\lesssim &\Vert \partial^{\alpha_1}[E, B]\Vert_{L^2_x} \Vert \langle v \rangle  \langle v \rangle^{-\gamma/2} f w(\alpha, \beta )\Vert_{H^2_xL^2_v} \Vert \partial^\alpha_{\beta}  f   \langle v \rangle^{\gamma/2} w (\alpha, \beta)\Vert_{L^2_xL^2_v}
\\
\lesssim & \Vert [E, B] \Vert_{H^7_x  }   \Vert \langle v \rangle  \langle v \rangle^{-\gamma/2} \nabla_v  f w(\alpha, \beta )\Vert_{H^2_xL^2_v} \sqrt{D_{7, k} }
\\ 
\lesssim  &  \Vert \langle v \rangle  \langle v \rangle^{-\gamma/2} \nabla_v  f w(\alpha, \beta )\Vert_{H^2_xL^2_v}    \sqrt{E^{eb}_{7} } \sqrt{D_{7, k}}
\\ 
\lesssim  & \sqrt{E_{6, k} } \sqrt{E^{eb}_{7}} \sqrt{D_{7, k} }
\end{aligned}
\end{equation}
the last inequality of \eqref{equationlem00} used the fact that
\[
 \langle v \rangle  \langle v \rangle^{-\gamma/2}w(7, 0) \le  w(2, 1)
\]
this is because of \eqref{equationw1} and \eqref{equationw2}
\[
 \langle v \rangle  \langle v \rangle^{-\gamma/2}w(7, 0)=\langle v\rangle ^{1-\gamma-20s} w(3,0)=\langle v\rangle ^{1-\gamma-20s-\min\{\gamma,0\}} w(2,1)\le w(2,1)\]
 so the proof is finished by gathering all the terms together.
\end{proof}

The next estimate is as follows.
\begin{lem}\label{L210}
Suppose that $-3\le\gamma\le 1, \gamma+2s >-1$ and $k\ge 20$. For any $G= \mu +g \ge 0$ satisfies
\[
\Vert G \Vert_{L^1} \ge \frac 1 2, \quad \Vert G\Vert_{L^1_2} + \Vert G\Vert_{L \log L} \le 4,
\]
there exist constants $\gamma_1, \gamma_2, C_k \ge 0$. such that for any $|\alpha| + |\beta| \le K$, if $K=2, 3, 4, 5,6$ we have
\begin{equation*}
\begin{aligned} 
&(\partial^\alpha_\beta Q ( G, f),  \partial^\alpha_\beta f w^2(\alpha, \beta) )_{L^2_{x, v}} 
\\
\le & - \frac {1} {12} \Vert  b(\cos \theta) \sin^2 \frac \theta 2  \Vert_{L^1_\theta}\Vert \partial^\alpha_\beta  f w(\alpha, \beta)    \Vert_{L^2_x L^2_{\gamma/2, * }}^2  +\gamma_1 \Vert  \partial^\alpha_\beta  f w(\alpha, \beta) \Vert_{L^2_xH^s_{\gamma/2}}^2
\\
&+C_k\sqrt{E^g_{K, k}}   D^f_{K, k} +C_k\sqrt{E^f_{K, k}}   \sqrt{D^f_{K, k}}  \sqrt{D^g_{K, k}} 
\\
& + C_{k}  \Vert \partial^\alpha_\beta f \Vert_{L^2_{x, v}}^2+ C_k \sum_{\beta_1 < \beta} \Vert \partial^\alpha_{\beta_1}  f w(\alpha, \beta_1) \Vert_{H^s_{\gamma/2}}\Vert \partial^\alpha_\beta f w(\alpha, \beta)\Vert_{H^s_{\gamma/2}} ,
\end{aligned}
\end{equation*}
if $K=7$ we have
\begin{equation*}
\begin{aligned} 
&(\partial^\alpha_\beta Q ( G, f),  \partial^\alpha_\beta f w^2(\alpha, \beta) )_{L^2_{x, v}} 
\\
\le & - \frac {1} {12} \Vert  b(\cos \theta) \sin^2 \frac \theta 2  \Vert_{L^1_\theta}\Vert \partial^\alpha_\beta  f w(\alpha, \beta)    \Vert_{L^2_x L^2_{\gamma/2, * }}^2  +\gamma_1 \Vert  \partial^\alpha_\beta  f w(\alpha, \beta) \Vert_{L^2_xH^s_{\gamma/2}}^2
\\
&+C_k\sqrt{E^g_{6, k}}   D^f_{7, k} +C_k\sqrt{E^f_{6, k}}   \sqrt{D^f_{7, k}}  \sqrt{D^g_{7, k}} 
\\
& + C_{k}  \Vert \partial^\alpha_\beta f \Vert_{L^2_{x, v}}^2+ C_k \sum_{\beta_1 < \beta} \Vert \partial^\alpha_{\beta_1}  f w(\alpha, \beta_1) \Vert_{H^s_{\gamma/2}}\Vert \partial^\alpha_\beta f w(\alpha, \beta)\Vert_{H^s_{\gamma/2}}.
\end{aligned}
\end{equation*}
\end{lem}
\begin{proof}
We focus on the cases $K=6,7$, since the case $K=3,4,5$ can be directly deduced from $K=3$, and we need to treat the case $K=2$ seperately. First we have
\[
(\partial^\alpha_\beta Q(G, f), \partial^\alpha f w^2(\alpha, \beta))_{L^2_{x, v}}= \sum_{\alpha_1 \le \alpha, \beta_1 \le \beta} ( Q(\partial^{\alpha_1}_{\beta_1} G, \partial^{\alpha - \alpha_1}_{\beta- \beta_1} f), \partial^{\alpha}_\beta f w^2(\alpha, \beta))_{L^2_{x, v}}.
\]
We split it into several cases. For the case $\alpha_1 =\beta_1 =0$, after integrating in $x$ in Lemma \ref{L58}, we have
\begin{equation*}
\begin{aligned} 
&( Q( G, \partial^{\alpha }_\beta f), \partial^{\alpha}_\beta f w^2(\alpha, \beta))_{L^2_{x, v}} 
\\
\le&- \frac {1} {12} \Vert  b(\cos \theta) \sin^2 \frac \theta 2  \Vert_{L^1_\theta}\Vert \partial^\alpha_\beta f w(\alpha, \beta)\Vert_{L^2_x L^2_{\gamma/2, *}}^2 - \gamma_1  \Vert \partial^\alpha_\beta f w(\alpha, \beta) \Vert_{L^2_x H^s_{\gamma/2}}^2+ C_{k}  \Vert \partial^\alpha_\beta f \Vert_{L^2_{x, v}}^2
\\
&+\int_{\T^3} C_k\Vert \partial^\alpha_\beta f   \Vert_{L^2_{14}} \Vert g w(\alpha, \beta) \Vert_{ H^s_{ \gamma/2 }}\Vert  \partial^\alpha_\beta f w(\alpha, \beta) \Vert_{ H^s_{  \gamma/2}}+ C_k\Vert   g \Vert_{   L^2_{14}} \Vert \partial^\alpha_\beta f w(\alpha, \beta) \Vert_{  H^s_{  \gamma/2}}\Vert \partial^\alpha_\beta f w(\alpha, \beta) \Vert_{  H^s_{ \gamma/2}} dx.
\end{aligned}
\end{equation*}
For the case $|\beta_1| > 0, \alpha_1=0$, using $\partial_{\beta_1} G = \partial_{\beta_1} \mu + \partial_{\beta_1} g$ we split it into two parts. Remind that because of \eqref{equationw1}, we have if $|\beta_1| > 0$
\[
w(\alpha, \beta) \langle v \rangle^{2s} \le w(\alpha, \beta - \beta_1),
\]
by Corollary  \ref{C22} we have
\begin{equation*}
\begin{aligned} 
|( Q(\partial_{\beta_1} \mu, \partial^\alpha_{\beta - \beta_1} f), \partial^{\alpha}_{\beta} f w^2 (\alpha, \beta))_{L^2_{x, v}} |\le& C_k \Vert \partial^\alpha_{\beta - \beta_1}  f w(\alpha, \beta) \Vert_{L^2_x H^s_{  \gamma/2+2s }}\Vert \partial^\alpha_\beta f w(\alpha, \beta) \Vert_{L^2_x H^s_{\gamma/2}} 
\\
\le &  C_k \Vert \partial^\alpha_{\beta - \beta_1}  f w(\alpha, \beta-  \beta_1 ) \Vert_{L^2_x H^s_{  \gamma/2 }}\Vert \partial^\alpha_\beta f w(\alpha, \beta) \Vert_{L^2_x H^s_{\gamma/2}} .
\end{aligned}
\end{equation*}
For the $\partial_{\beta_1} g$ term, by Lemma \ref{L26} we have
\begin{equation*}
\begin{aligned} 
|( Q(\partial_{\beta_1} g, \partial^\alpha_{\beta - \beta_1} f), \partial^{\alpha}_{\beta} f w^2 (\alpha, \beta))_{L^2_{x, v}} |
\le& \int_{\T^3}  C_k\Vert \partial^\alpha_{\beta - \beta_1} f \Vert_{L^2_{14}} \Vert \partial_{\beta_1} g w(\alpha, \beta)\Vert_{H^s_{ \gamma/2 
 }}\Vert \partial^{\alpha}_{\beta}  f  w(\alpha, \beta) \Vert_{H^s_{  \gamma/2}} 
\\
&+C_k\Vert \partial_{\beta_1} g \Vert_{L^2_{14}} \Vert \partial^\alpha_{\beta - \beta_1} f w(\alpha, \beta)\Vert_{H^s_{  \gamma/2 +2s }}\Vert \partial^{\alpha}_{\beta}  f w(\alpha, \beta) \Vert_{H^s_{ \gamma/2}} dx,
\end{aligned}
\end{equation*}
For $|\alpha_1|>0$, since $\partial^{\alpha_1} \mu =0$, we have
\[
( Q(\partial^{\alpha_1}_{\beta_1} G, \partial^{\alpha - \alpha_1}_{\beta -\beta_1} f), \partial^{\alpha}_\beta f \langle v \rangle^{2k})_{L^2_{x, v}} = ( Q(\partial^{\alpha_1}_{\beta_1} g, \partial^{\alpha - \alpha_1}_{\beta -\beta_1} f), \partial^{\alpha}_\beta f \langle v \rangle^{2k})_{L^2_{x, v}}, 
\]
by Lemma \ref{L26} we have
\begin{equation*}
\begin{aligned} 
|( Q(\partial^{\alpha_1}_{\beta_1} g, \partial^{\alpha - \alpha_1}_{\beta -\beta_1} f), \partial^{\alpha}_\beta f \langle v \rangle^{2k})_{L^2_{x, v}}|
\le& \int_{\T^3}  C_k\Vert \partial^{\alpha - \alpha_1}_{\beta - \beta_1} f \Vert_{L^2_{14}} \Vert \partial^{\alpha_1}_{\beta_1} g w(\alpha, \beta)\Vert_{H^b_{ \gamma/2 }}\Vert \partial^{\alpha}_{\beta}  f  w(\alpha, \beta) \Vert_{L^2_{  \gamma/2}} dx
\\
&+C_k\int_{\T^3} \Vert \partial^{\alpha_1}_{\beta_1} g \Vert_{L^2_{14}} \Vert \partial^{\alpha - \alpha_1}_{\beta - \beta_1} f w(\alpha, \beta)\Vert_{H^s_{ \gamma/2+2s }}\Vert \partial^{\alpha}_{\beta}  f w(\alpha, \beta) \Vert_{H^s_{ \gamma/2}} dx. 
\end{aligned}
\end{equation*}
Gathering  all the terms, we have
\begin{equation*}
\begin{aligned} 
&(\partial^\alpha_\beta Q ( \mu+g, f),  \partial^\alpha_\beta f w^2(\alpha, \beta) )_{L^2_{x, v}}
\\
\le & - \frac {1} {12} \Vert  b(\cos \theta) \sin^2 \frac \theta 2    \Vert_{L^1_\theta}\Vert \partial^\alpha_\beta f w(\alpha, \beta) \Vert_{L^2_xL^2_{\gamma/2, *}}^2 + C_{k}  \Vert \partial^\alpha_\beta f \Vert_{L^2_xL^2_{x, v} }^2 - \gamma_1\Vert \partial^\alpha_\beta f w(\alpha, \beta) \Vert_{L^2_xH^s_{\gamma/2}}^2 
\\
& + C_k \sum_{|\beta_1| < |\beta|} \Vert \partial^\alpha_{\beta_1}  f w(\alpha, \beta_1 ) \Vert_{L^2_x H^s_{\gamma/2}}\Vert \partial^\alpha_\beta f w(\alpha, \beta  ) \Vert_{L^2_xH^s_{\gamma/2}} 
\\
&+C_k \int_{\T^3}  \sum_{\alpha_1 \le \alpha, \beta_1 \le \beta}\Vert \partial^{\alpha - \alpha_1}_{\beta - \beta_1} f \Vert_{L^2_{14}} \Vert \partial^{\alpha_1}_{\beta_1} g w(\alpha, \beta)\Vert_{H^s_{ \gamma/2 }}\Vert \partial^{\alpha}_{\beta}  f  w(\alpha, \beta) \Vert_{H^s_{  \gamma/2}} dx
\\
&+C_k\int_{\T^3} \sum_{\alpha_1 \le \alpha, \beta_1 \le \beta, |\alpha_1| +|\beta_1| >0} \Vert \partial^{\alpha_1}_{\beta_1} g \Vert_{L^2_{14}} \Vert \partial^{\alpha - \alpha_1}_{\beta - \beta_1} f w(\alpha, \beta)\Vert_{H^s_{  \gamma/2+2s }}\Vert \partial^{\alpha}_{\beta}  f w(\alpha, \beta) \Vert_{H^s_{ \gamma/2}} dx
\\
&+C_k\int_{\T^3} \Vert  g \Vert_{L^2_{14}} \Vert \partial^{\alpha }_{\beta} f w(\alpha, \beta)\Vert_{H^s_{  \gamma/2 }}\Vert \partial^{\alpha}_{\beta}  f w(\alpha, \beta) \Vert_{H^s_{ \gamma/2}} dx.
\end{aligned}
\end{equation*}
 we only need to prove that 
\[
\int_{\T^3} \Vert  g \Vert_{L^2_{14}} \Vert \partial^{\alpha }_{\beta} f w(\alpha, \beta)\Vert_{H^s_{  \gamma/2 }}\Vert \partial^{\alpha}_{\beta}  f w(\alpha, \beta) \Vert_{H^s_{ \gamma/2}} dx \lesssim \sqrt{E^g_{K, k}}   D^f_{K, k}, 
\] 
and  for all  $\alpha_1 \le \alpha, \beta_1 \le \beta, |\alpha_1| +|\beta_1| >0$
\[
\int_{\T^3} \Vert \partial^{\alpha_1}_{\beta_1} g \Vert_{L^2_{14}} \Vert \partial^{\alpha - \alpha_1}_{\beta - \beta_1} f w(\alpha, \beta)\Vert_{H^s_{  \gamma/2+2s }}\Vert \partial^{\alpha}_{\beta}  f w(\alpha, \beta) \Vert_{H^s_{ \gamma/2}} dx  \lesssim  \sqrt{E^g_{K, k}}   D^f_{K, k}, 
\]
since the other nonlinear term follows by directly changing $f$ and $g$ and even more simple since $s>0$.  First , for the case $|\alpha_1| = |\beta_1| =0 $,  we have
\begin{equation*}
\begin{aligned}
\int_{\T^3} \Vert  g \Vert_{L^2_{14}} \Vert \partial^{\alpha}_{\beta } f w(\alpha, \beta)\Vert_{H^s_{  \gamma/2 }}\Vert \partial^{\alpha}_{\beta}  f w(\alpha, \beta) \Vert_{H^s_{ \gamma/2}} dx 
&\lesssim \Vert  g \Vert_{L^\infty_xL^2_{14}} \Vert \partial^{\alpha  }_{\beta } f w(\alpha, \beta)\Vert_{L^2_x H^s_{  \gamma/2 }}\Vert \partial^{\alpha}_{\beta}  f w(\alpha, \beta) \Vert_{L^2_x H^s_{ \gamma/2}} 
\\
\lesssim& \Vert  g \Vert_{H^2_x L^2_{14}} \Vert \partial^{\alpha  }_{\beta } f w(\alpha, \beta)\Vert_{L^2_x H^s_{  \gamma/2 }}\Vert \partial^{\alpha}_{\beta}  f w(\alpha, \beta) \Vert_{L^2_x H^s_{ \gamma/2}} 
\\
\lesssim&  \sqrt{E^g_{K, k}}   D^f_{K, k}.
\end{aligned}
\end{equation*}
For the term
\[
\int_{\T^3} \sum_{\alpha_1 \le \alpha, \beta_1 \le \beta, |\alpha_1| +|\beta_1| >0} \Vert \partial^{\alpha_1}_{\beta_1} g \Vert_{L^2_{14}} \Vert \partial^{\alpha - \alpha_1}_{\beta - \beta_1} f w(\alpha, \beta)\Vert_{H^s_{  \gamma/2+2s }}\Vert \partial^{\alpha}_{\beta}  f w(\alpha, \beta) \Vert_{H^s_{ \gamma/2}} dx.:=(*)
\]
\underline{If $K=6,7$}. For the case $1 \le |\alpha_1| + |\beta_1|  \le 4 $, we have
\begin{align*}
(*)&\lesssim  \Vert \partial^{\alpha_1}_{\beta_1} g \Vert_{H^2_x L^2_{14}} \Vert \partial^{\alpha - \alpha_1}_{\beta - \beta_1} f w(\alpha, \beta)\Vert_{L^2_x H^s_{  \gamma/2+2s }}\Vert \partial^{\alpha}_{\beta}  f w(\alpha, \beta) \Vert_{L^2_x H^s_{ \gamma/2}} 
\\
&\lesssim \Vert \langle v \rangle^{14} g \Vert_{H^6_{x, v} } \Vert \partial^{\alpha - \alpha_1}_{\beta - \beta_1} f w(\alpha, \beta)\Vert_{L^2_x H^s_{  \gamma/2+2s }}   \sqrt{D^f_{K, k}} \lesssim \sqrt{E^g_{6, k}} D^f_{K, k}
\end{align*}
for the case $ |\alpha_1| +|\beta_1|  =5$, we have
\begin{align*}
(*)\lesssim&  \Vert \partial^{\alpha_1}_{\beta_1} g \Vert_{H^1_x L^2_{14}} \Vert \partial^{\alpha - \alpha_1}_{\beta - \beta_1} f w(\alpha, \beta)\Vert_{H^1_x H^s_{  \gamma/2+2s }}\Vert \partial^{\alpha}_{\beta}  f w(\alpha, \beta) \Vert_{L^2_x H^s_{ \gamma/2}} 
\\
\lesssim& \Vert \langle v \rangle^{14} g \Vert_{H^6_{x, v} } \Vert \partial^{\alpha - \alpha_1}_{\beta - \beta_1} f w(\alpha, \beta)\Vert_{H^1_x H^s_{  \gamma/2+2s }}    \sqrt{D^f_{K, k}} \lesssim \sqrt{E^g_{6, k}} D^f_{K, k}
\end{align*}
for the case $ |\alpha_1| +|\beta_1|  =6$, we have
\begin{align*}
(*)\lesssim&  \Vert \partial^{\alpha_1}_{\beta_1} g \Vert_{L^2_x L^2_{14}} \Vert \partial^{\alpha - \alpha_1}_{\beta - \beta_1} f w(\alpha, \beta)\Vert_{H^2_x H^s_{  \gamma/2+2s }}\Vert \partial^{\alpha}_{\beta}  f w(\alpha, \beta) \Vert_{L^2_x H^s_{ \gamma/2}} 
\\
\lesssim& \Vert \langle v \rangle^{14} g \Vert_{H^6_{x, v} } \Vert \partial^{\alpha - \alpha_1}_{\beta - \beta_1} f w(\alpha, \beta)\Vert_{H^2_x H^s_{  \gamma/2+2s }}    \sqrt{D^f_{K, k}} \lesssim \sqrt{E^g_{6, k}} D^f_{K, k}
\end{align*}
the estimates above used the fact that 
\begin{equation}
\label{equationw6}
w(\alpha, \beta) \langle v \rangle^{2s} \le w(\alpha -\alpha_1, \beta-\beta_1) , \quad |\alpha_1| +|\beta_1|  \ge 1
\end{equation}
because of \eqref{equationw1}.
For the case  $K=7$, $|\alpha_1| +|\beta_1| =7$, this time $\alpha = \alpha_1 , \beta =\beta_1$. Similarly as  \eqref{equationw6}, we have
$w(\alpha, \beta) \langle v \rangle^{\gamma/2+s} \le w(2, 1)$.
So
\begin{align*}
(*)\lesssim&  \Vert \partial^{\alpha}_{\beta} g \Vert_{L^2_x L^2_{14}} \Vert  f w(\alpha, \beta)\Vert_{H^2_x H^s_{  \gamma/2+2s }}\Vert \partial^{\alpha}_{\beta}  f w(\alpha, \beta) \Vert_{L^2_x H^s_{ \gamma/2}} 
\lesssim\Vert \langle v \rangle^{14} g \Vert_{H^7_{x, v} } \Vert  f w(\alpha, \beta)\Vert_{H^2_x H^s_{  \gamma/2+2s }}   \sqrt{D^f_{7, k}} 
\\
\lesssim &   \Vert  f w(\alpha, \beta)\Vert_{H^2_x H^s_{  \gamma/2+2s }}   \sqrt{D^f_{7, k}}  \sqrt{D^g_{7, k}} 
\lesssim   \sqrt{E^f_{6, k}}   \sqrt{D^f_{7, k}}  \sqrt{D^g_{7, k}} 
\end{align*}
and the proof is thus finished. \\
\underline{If $K=2$.}
we split it into two cases,  $|\alpha_1| + |\beta_1| =1 $ and  $|\alpha_1| + |\beta_1| = 2 $, for the case $|\alpha_1| + |\beta_1| =1 $
by $\Vert fg\Vert_{L^2_x} \le \Vert  f \Vert_{L^6_x} \Vert g \Vert_{L^3_x}$,
we have
\begin{equation*}
\begin{aligned}
(*)
\lesssim&  \Vert \partial^{\alpha_1}_{\beta_1} g \Vert_{L^6_x L^2_{14}} \Vert \partial^{\alpha - \alpha_1}_{\beta - \beta_1} f w(\alpha, \beta)\Vert_{L^3_x H^s_{  \gamma/2+2s }}\Vert \partial^{\alpha}_{\beta}  f w(\alpha, \beta) \Vert_{L^2_x H^s_{ \gamma/2}} 
\\
\lesssim& \Vert \partial^{\alpha_1}_{\beta_1} g \Vert_{H^1_xL^2_{14}} \Vert \partial^{\alpha - \alpha_1}_{\beta - \beta_1} f w(\alpha, \beta)\Vert_{L^3_x H^s_{  \gamma/2+2s }}\Vert \partial^{\alpha}_{\beta}  f w(\alpha, \beta) \Vert_{L^2_x H^s_{ \gamma/2}} 
\\
\lesssim& \Vert \langle v \rangle^{14} g \Vert_{H^2_{x, v} } \Vert \partial^{\alpha - \alpha_1}_{\beta - \beta_1} f w(\alpha, \beta)\Vert_{L^3_x H^s_{  \gamma/2+2s }}    \sqrt{D^f_{2,k}}, 
\end{aligned}
\end{equation*}
 we again split it into two parts  $|\alpha|+|\beta| = 1$ and  $|\alpha|+|\beta| = 2$, for the case  $|\alpha_1| + |\beta_1| = |\alpha|+|\beta| =1$ we have $|\alpha - \alpha_1| =|\beta - \beta_1| =0 $, by
$\Vert f  \Vert_{L^3_xL^2_v}  \lesssim \Vert f  \Vert_{H^{1/2}_x L^2_v } \lesssim\Vert f \langle v \rangle^{k} \Vert_{L^2_{x}}+ \Vert f \langle v \rangle^{-k} \Vert_{H^1_x}$,we have
\[
\Vert f w(\alpha, \beta)\Vert_{L^3_x H^s_{  \gamma/2+2s }}  \lesssim \Vert f w(0, 0)\Vert_{L^2_x H^s_{  \gamma/2 }}  + \Vert f w(1,0)\Vert_{H^1_x H^s_{  \gamma/2 }} \lesssim  \sqrt{D^f_{2,k}}
\]
because of \eqref{equationw1},\eqref{equationw2}, we have
\[
\max_{|\alpha|+|\beta| =1}w^2 (\alpha, \beta) \langle v \rangle^{4s} \le w^2(1, 0) \langle v \rangle^{4s} \le  w(0, 0) w(1, 0) , 
\]
then we consider the case $|\alpha_1| +|\beta_1| =1, |\alpha| + |\beta| =2$,  this time we have $|\alpha - \alpha_1| =1, |\beta - \beta_1|=0$ or $|\alpha - \alpha_1| =0, |\beta - \beta_1| =1$, for the first case we have
\[
\Vert \partial^{\alpha - \alpha_1} f w(\alpha, \beta)\Vert_{L^3_x H^s_{  \gamma/2+2s }}  \le  \Vert f w(1, 0)\Vert_{H^1_x H^s_{ \gamma/2 }} + \Vert f w( 2 , 0)\Vert_{H^2_x H^s_{ \gamma/2 }}  \lesssim  \sqrt{D^f_{2,k}}
\]
because of \eqref{equationw1},\eqref{equationw2}, we have
\[
\max_{|\alpha|+|\beta| =2}w^2 (\alpha, \beta) \langle v \rangle^{4s} \le w^2(2, 0) \langle v \rangle^{4s}  \le w(1, 0) w(2, 0) , 
\]
for the second case we have $|\beta| \ge 1$ and 
\[
\Vert  \partial_{\beta - \beta_1}  f w(\alpha, \beta)\Vert_{L^3_x H^s_{ \gamma/2+2s }} \le \Vert f w(0, 1)\Vert_{L^2_x H^{1+s}_{ \gamma/2 }} + \Vert f  w( 1 , 1)\Vert_{H^1_x H^{1+s}_{ \gamma/2 }} \lesssim  \sqrt{D^f_{2,k}}
\]
because of \eqref{equationw1},\eqref{equationw2}, we have
\[
\max_{|\alpha|+|\beta| =2, |\beta| \ge 1}w^2 (\alpha, \beta) \langle v \rangle^{4s}\le w^2(1, 1) \langle v \rangle^{4s}  \le  w(0, 1) w(1, 1), 
\]
for the case $|\alpha_1|+|\beta_1| = 2$, it is easily seen that this time $|\alpha| +|\beta|=2$ and $|\alpha - \alpha_1| = |\beta - \beta_1|=0$, so we have
\begin{equation*}
\begin{aligned}
&\int_{\T^3} \Vert \partial^{\alpha_1}_{\beta_1} g \Vert_{L^2_{14}} \Vert \partial^{\alpha - \alpha_1}_{\beta - \beta_1} f w(\alpha, \beta)\Vert_{H^s_{  \gamma/2+2s }}\Vert \partial^{\alpha}_{\beta}  f w(\alpha, \beta) \Vert_{H^s_{ \gamma/2}} dx
\\
\lesssim&  \Vert \partial^{\alpha}_{\beta} g \Vert_{L^2_x L^2_{14}} \Vert f w(\alpha, \beta)\Vert_{L^\infty_x H^s_{  \gamma/2+2s }}\Vert \partial^{\alpha}_{\beta}  f w(\alpha, \beta) \Vert_{L^2_x H^s_{ \gamma/2}} 
\\
\lesssim& \Vert \langle v \rangle^{14} g \Vert_{H^2_{x, v} } \Vert f w(\alpha, \beta)\Vert_{L^\infty_x H^s_{  \gamma/2+2s }}     \sqrt{D^f_{2,k}}, 
\end{aligned}
\end{equation*}
by 
\[
\Vert f \Vert_{L^\infty_xL^2_v} \lesssim \Vert f \Vert_{H^{8/5}_x L^2_v} \le \Vert f \langle v \rangle^{\frac 3 2 k}\Vert_{H^1_{x}L^2_v}^{3/5} \Vert f\langle v \rangle^{-k} \Vert_{H^2_xL^2_x }^{2/5} \lesssim \Vert f \langle v \rangle^{\frac 3 2 k} \Vert_{H^1_{x}L^2_v}+ \Vert f \langle v \rangle^{-k}  \Vert_{H^2_xL^2_v}
\]
we have
\[
\Vert  f w(\alpha, \beta)\Vert_{L^\infty_x H^s_{ \gamma/2  +2s}} \le \Vert f w(1, 0)\Vert_{H^1_x H^s_{ \gamma/2 }} + \Vert f w(2 , 0)\Vert_{H^2_x H^s_{ \gamma/2 }}\lesssim  \sqrt{D^f_{2,k}}
\]
because of \eqref{equationw1},\eqref{equationw2}, we have
\[
\max_{|\alpha|+|\beta| =2}w^2 (\alpha, \beta) \langle v \rangle^{4s}\le  w^2(2, 0) \langle v \rangle^{4s} \le  w(1, 0)^{4/5} w(2, 0)^{6/5}, 
\]
so we always have $(*)\lesssim \sqrt{E^g_{2,k}}D^f_{2,k}$. And the proof is done after gathering the terms. 
\end{proof}

\subsection{Transport Part}
At the end of this section, we give the estimate of the transport part term.

\begin{lem}\label{L28}
Suppose $|\beta|>0$. For any smooth function $f$, for any $\eta>0$ small we have
\[
\sum_{|\alpha| +|\beta| \le K} C_{|\alpha|, |\beta|}^2  ( \partial^\alpha_{\beta} (v \cdot \nabla_x f)  ,\partial^\alpha_{\beta}f  w^2(\alpha, \beta)  )_{L^2_{x, v}} \le \eta \sum_{|\alpha| +|\beta| \le K} C_{|\alpha|, |\beta|}^2  \Vert \partial^\alpha_\beta  f w(\alpha, \beta)    \Vert_{L^2_x L^2_{\gamma/2 }}^2  
\]

\end{lem}
\begin{proof}

Notice that $\partial_{v_i} (v \cdot \nabla_x f) = \partial_{x_i} f + v \cdot \nabla_x \partial_{v_i} f$ for $i =1, 2, 3$
and $w(\alpha, \beta) \le \langle v \rangle^{\gamma} w (|\alpha|+1, |\beta|-1 )$ for $ |\alpha| \ge 0, |\beta| \ge 1$. So for $|\beta|>0$, we have
\begin{equation}
\label{eqL28}
\Vert \partial^\alpha_{\beta} (v \cdot \nabla_x f)  ,\partial^\alpha_{\beta}f  w^2(\alpha, \beta)  \Vert_{L^2_{x, v}} \le C_k\Big(\sum_{|\beta_2|=|\beta|-1,  |\alpha_2| =|\alpha|+1} \Vert \partial^{\alpha_2}_{\beta_2}  f w(\alpha_2, \beta_2)\Vert_{L^2_xL^2_{\gamma/2}}\Big) \Vert \partial^\alpha_{\beta} f w(\alpha, \beta) \Vert_{L^2_{x}L^2_{\gamma/2} }.
\end{equation}
next, notice that for any  $\eta>0$,
\begin{equation*}
\begin{aligned} 
& C_{|\alpha|, |\beta|}^2 C_k\sum_{|\beta_2|=|\beta|-1,  |\alpha_2| =|\alpha|+1} \Vert \partial^{\alpha_2}_{\beta_2}  f w(\alpha_2, \beta_2)\Vert_{L^2_x L^2_{\gamma/2}}\cdot\Vert \partial^\alpha_{\beta} f w(\alpha, \beta) \Vert_{L^2_{x}  L^2_{\gamma/2} } 
\\
\le&  \frac {C_{|\alpha|, |\beta|} C_k} \eta   \sum_{|\beta_2|=|\beta|-1,  |\alpha_2| =|\alpha|+1} \Vert \partial^{\alpha_2}_{\beta_2}  f w(\alpha_2, \beta_2)\Vert_{L^2_x L^2_{\gamma/2}}\cdot \eta C_{|\alpha|, |\beta|} \Vert \partial^\alpha_{\beta} f w(\alpha, \beta) \Vert_{L^2_{x}  L^2_{\gamma/2} } 
\end{aligned}
\end{equation*}
sum up on $|\alpha| +|\beta| \le K$ on left side of \eqref{eqL28} and recall the definition of $C_{|\alpha|,|\beta|}$, we chould choose small $\eta$ that satisfies
$\frac {C_{|\alpha|, |\beta|} C_k} \eta  \ll C_{|\alpha|+1, |\beta|-1}$. The lemma is thus proved.

\end{proof}
After combining the lemmas above, we have the following local existence.
\begin{thm}
\label{thmlocal}
Let $-3<\gamma\le 1, 0<s<1, \gamma+2s\ge -1$ for Botzmann case, $-3\le \gamma\le 1$ for Landau case. for any $K\ge 3$ an d $k\ge 14$, there exist small constants $\varepsilon_0,\tau_0, T_0>0$, such that if $(f_0, E_0, B_0)$ satisfy $E_{K,k}(0)\le \varepsilon_0$. Then the Cauchy problem
\begin{equation}
	\label{Cauchy1}
	\left\{
	\begin{aligned}
&\partial_t f_\pm+v \cdot \nabla_x f_\pm \mp (E + v \times B) \cdot \nabla_v f_\pm \pm ( E  \cdot v)   \mu =Q(f_\pm+f_\mp,\mu)+Q(2\mu+f_\pm+f_\mp, f_\pm)\\
		&\partial_t E - \nabla_x \times B = -\int_{\R^3} v (f_+-f_{-}) dv:=G, \quad
\nabla_x \cdot E = \int_{\R^3} (f_+-f_{-}) dv   ,\\
&\partial_t B +  \nabla_x \cdot E =0,\quad \nabla_x \cdot B =0, \quad f(0) = f_0,\quad E(0) = E_0,\quad B(0)=B_0
		\end{aligned}\right.
\end{equation}
admits a unique weak solution $(f,E,B)\in L^{\infty}_{T_0}X_k$ satisfying $\mu+f  \ge 0$ and $E_{K,k}(t)(f)\le \tau_0$ for any $0\le t\le T_0$.
\end{thm}

\section{Macroscopic Estimate}
\label{sec3}
In this section, we will derive the macroscopic dissipation for the Vlasov-Maxwell-Boltzmann/Landau equation. Remind $\bm{P}, a, b, c$ defined in \eqref{projection}. 
we also rewrite the equation \eqref{VMB1} as
the equation becomes
\begin{equation}
\label{equationmacro01}
\partial_t f +v \cdot \nabla_x f  \mp E  \cdot v   \mu =L_\pm f+N_\pm(f)
\end{equation}
here the linear part $L$ and nonlinear part $N$ are defined by
\begin{equation}\label{nlp}
L_\pm f:=Q(f_\pm+f_\mp, \mu)+2Q(\mu, f_\pm), \quad N_\pm(f):=\pm( E + v \times B) \cdot \nabla_v f +Q(f_\pm , f_\pm)+Q(f_\mp , f_\pm)
\end{equation}
 Remind that the electromagnetic fields $[E,B]$ satisfy the Maxwell equation
\begin{equation}
\label{maxwell1}
\partial_tE-\nabla_x\times B=-G,\quad\partial_tB+\nabla_x\times E=0,\quad
\nabla_x\cdot E=a_{+}-a_{-}, \quad\nabla_x\cdot B=0
\end{equation}
where 
\begin{equation}
\label{equationG}
G=((\bm I-\bm P)f\cdot [1,-1], v)_{L^2_v}
\end{equation}
we also consider linearized equation of \eqref{equationmacro01}, which is
\begin{equation}
\label{equationmacro02}
\partial_t \mathsf f =-v \cdot \nabla_x\mathsf f  \pm E  \cdot v   \mu +L_\pm\mathsf  f
\end{equation}
our next goal is to estimate $a(t, x), b(t, x),c(t, x)$ in terms of $(I-P) f$.

Unlike \cite{G4}, our $P$ is not symmetric we can not just compare the $v_i, v_j$  terms on both sides of \eqref{equationmacro01}. We have

\begin{lem}\label{L420}
Let $K=2,3,4,5,6$. Let $(f,E,B)$ be the solution to \eqref{VMB3} and \eqref{VMB4}. Then there exists an instant energy functional $E_{K,\text{int}}(t)$ satisfying
\begin{equation}
\label{macro1}
E_{K, \text{int}}(t)\lesssim \|f\|_{H^K_xL^2_v} + \|[E,B]\|_{H^{K-1}_x},
\end{equation}
such that 
\begin{equation}
\label{macro2}
\partial _tE_{K, \text{int}}(t) + \lambda \|[a, b, c]\|^2_{H^K_x} + \lambda \|[E,B]\|^2_{H^{K-1}_x}
\lesssim \|(\bm I-\bm P) f\|^2_{H^K_xL^2_{10}}+\|  N_{{\|}}  \|^2_{H^K_x}+ \|E\|^4_{L^2_x} + \|B\|^4_{L^2_x},  
\end{equation}
for some constant $\lambda>0$,  
where $\partial^{\alpha}N_{{\|}}$ is the $L^2_v$ projection of $\partial^{\alpha}N(f)(t, x, v)$ onto the subspace generated by the  basis
\[
\{[1,0]\mu,\,\,[0,1]\mu,\,\,[1,1]v_i\mu,\,\,[1,1]v_iv_j\mu,\,\,[1,1]v_i|v|^2\mu\},\quad 1\le i, j\le 3.?
\]
as a direct result, for $f$ as the solution of $\partial_t \mathsf f +v \cdot \nabla_x \mathsf f  \mp E  \cdot v   \mu =L_\pm \mathsf f$, $(E,B)$ satisfy \eqref{VMB4},  there exists a functional $E_{K, \text{int}}\lesssim {\|}f{\|}_{H^K_xL^2_v}$, such that
\begin{equation}
\label{macro3}
\partial_tE_{K, \text{int}}(t)+\lambda{\|}[a,b,c]{\|}^2_{H^K_x}\lesssim {\|}(I-P)f_\pm{\|}^2_{H^K_xL^2_{10}}.
\end{equation}
\end{lem}

\begin{proof}
Here we only prove the cases $K=2$ and $K=3$, since the case $K\ge 4$ case is similar as $K=3$. From $f=Pf+(I-P)f$, we directly have
\begin{equation*}
\begin{aligned}
&(\partial_t a_\pm + \partial_t b \cdot v + \partial_t c (|v|^2-3))\mu+v\cdot(\nabla_x a_\pm +\nabla_x (b \cdot v)+ \nabla_x c (|v|^2-3))\mu\pm E\cdot v\mu
\\
=&-(\partial_t+v\cdot\nabla_x)(\bm I_\pm-\bm P_\pm)f- L_\pm(\bm I_\pm-\bm P_\pm)f+N_\pm(f),
\end{aligned}
\end{equation*}
take the  inner product with $1, v_iv_j, v_i^2, |v|^2, i, j =1, 2, 3, i \neq j$, and integrate in $v\in\mathbb R^3$, then we have
\begin{equation}
\label{uni55}
\begin{aligned}
&\partial_t a_\pm+\nabla_x \cdot  b=(-(\partial_t+v\cdot\nabla_x)(\bm I_\pm-\bm P_\pm)f- L_\pm(\bm I_\pm-\bm P_\pm)f+N_\pm(f), 1)_{L^2_v}
\\
& \partial_{x_j} b_i  +  \partial_{x_i} b_j=(-(\partial_t+v\cdot\nabla_x)(\bm I_\pm-\bm P_\pm)f- L_\pm(\bm I_\pm-\bm P_\pm)f, v_iv_j)_{L^2_v}+(N_\pm(f), v_iv_j)_{L^2_v}:=\gamma_{1ij}+\gamma_{2ij}, \quad i\ne j
\\
&\partial_t a _\pm+2 \partial_t c + 2\partial_{x_i} b_i  +\nabla_x \cdot b=(-(\partial_t+v\cdot\nabla_x)(\bm I_\pm-\bm P_\pm)f- L_\pm(\bm I_\pm-\bm P_\pm)f+N_\pm(f), v_i^2)_{L^2_v}
\\
&3\partial_t a_\pm + 6\partial_t c + 5\nabla_x \cdot  b=(-(\partial_t+v\cdot\nabla_x)(\bm I_\pm-\bm P_\pm)f- L_\pm(\bm I_\pm-\bm P_\pm)f+N_\pm(f), |v|^2)_{L^2_v},
\end{aligned}
\end{equation}
and
\begin{equation}
\label{equationG2}
\partial_tG+\nabla_x(a_+-a_{-})-2E+\nabla_x\cdot\Theta((\bm I-\bm P)f\cdot [1,-1])=((N+Lf)\cdot [1,-1], v)_{L^2_v}
\end{equation}
here $\Theta_{jm}(f_\pm)=(f_{\pm}, (v_jv_m-1))_{L^2_v}$. Minus $\eqref{uni55}_3$ by $\eqref{uni55}_1$, we have
\begin{equation}
\label{abc0}
\partial_t  c+  \partial_{x_i} b_i=(-(\partial_t+v\cdot\nabla_x)(I-P)f- L_\pm(\bm I_\pm-\bm P_\pm)f, \frac{ |v_i|^2-1}  {  2  })_{L^2_v}+(N_\pm(f), \frac{ |v_i|^2-1}  {  2  })_{L^2_v}:=\gamma_{1i}+\gamma_{2i},
\end{equation}
take the vector inner product with $v, v|v|^2$ and integrate in $v\in\mathbb R^3$, then we have
\begin{equation}
\label{uni5new}
\begin{aligned}
&\partial_t b+ \nabla_x a_\pm \mp E +2\nabla_x c=(-(\partial_t+v\cdot\nabla_x)(\bm I_\pm-\bm P_\pm)f- L_\pm(\bm I_\pm-\bm P_\pm)f+N_\pm(f), v)_{L^2_v}
\\
&5\partial_t b+ 5\nabla_x a_\pm \mp  5E +20\nabla_x c= (-(\partial_t+v\cdot\nabla_x)(\bm I_\pm-\bm P_\pm)f- L_\pm(\bm I_\pm-\bm P_\pm)f+N_\pm(f), v|v|^2)_{L^2_v},
\end{aligned}
\end{equation}
so we deduce that
\begin{equation}\label{abc1}
\nabla_x c=  (-(\partial_t+v\cdot\nabla_x)(\bm I_\pm-\bm P_\pm)f- L_\pm(\bm I_\pm-\bm P_\pm)f+N_\pm(f), \frac{v(|v|^2-5)}{10})_{L^2_v},
\end{equation}
\begin{equation}\label{abc2}
\nabla_x a_\pm=- \partial_t b\pm  E  + (-(\partial_t+v\cdot\nabla_x)(\bm I_\pm-\bm P_\pm)f- L_\pm(\bm I_\pm-\bm P_\pm)f+N_\pm(f), \frac{v(10-|v|^2)}{5})_{L^2_v}.
\end{equation}
here we use the fact that
$\int_{\mathbb R^3}{|v|^2\mu dv}=3, \int_{\mathbb R^3}{|v|^4\mu dv}=15, \int_{\mathbb R^3}{|v|^6\mu dv}=105$. 
so
\begin{equation}
\label{uni8}
\begin{aligned}
&\partial_t a_\pm+\nabla_x  \cdot b=(N_\pm(f),1)_{L^2_v}
\\
&\partial_t b + \nabla_x a_\pm+2\nabla_x c\mp E=-(v\cdot\nabla_x(\bm I_\pm-\bm P_\pm)f, v)_{L^2_v}+(N_\pm(f), v)_{L^2_v}
\\
&\partial_t c+\frac{1}{3}\nabla_x \cdot  b=-\frac{1}{6}(v\cdot\nabla_x(\bm I_\pm-\bm P_\pm)f,|v|^2)_{L^2_v}+\left( N_\pm(f),\frac{|v|^2-3}{6} \right)_{L^2_v},
\end{aligned}
\end{equation}
we define 
\[
\xi_{a}=\frac{v_i(10-|v|^2)}{5}, \quad  \xi_{bi}=\frac{2v_i^2-5}{2}, \quad \xi_{c}= \frac{v(|v|^2-5)}{10},
\]
for convenience. It is easily seen that
\[
|(v \cdot \nabla_x (I-P)f,  \xi_{a})| +|(v \cdot \nabla_x (I-P)f,  \xi_{bi})| +|(v \cdot \nabla_x (I-P)f,  v_iv_j)| +|(v \cdot \nabla_x (I-P)f,  \xi_{c})| \le \Vert  \nabla_x (I-P) f\Vert_{L^2_6},
\]
and since $L = Q(\mu, f) +Q(f, \mu) $. For both Boltzmann case and Landau case,  by Lemma \ref{L55} we have
\[
|(Q(\mu, f), \xi_{a})| \le  \Vert \mu \Vert_{L^2_{10}} \Vert f \Vert_{L^2_{10}} \Vert \xi_{a}\Vert_{H^{\kappa}_{\iota}} \le C\Vert f \Vert_{L^2_{10}} , \quad  |(Q(f, \mu), \xi_{a})| \le  \Vert f \Vert_{L^2_{10}} \Vert \mu \Vert_{L^2_{10}} \Vert \xi_{a}\Vert_{H^{\kappa}_{\iota}} \le C\Vert f \Vert_{L^2_{10}} ,
\]
here $\kappa=2s, \iota=-7$ Boltzmann case, and $\kappa=2, \iota=-5$ for Landau case. Similarly we have
\begin{equation}
\label{ineqi}
|( Lf,\xi_{a})_{L^2_v}|+ |( Lf,\xi_{bi})_{L^2_v}| + |( Lf, v_iv_j)_{L^2_v}| + |( Lf,\xi_{c})_{L^2_v}| \lesssim \Vert f \Vert_{L^2_{10}}.
\end{equation}
Let $|\alpha|\le 2$. Take $\partial^\alpha $ derivative on both side of \eqref{abc1}, multiply  by $\nabla_x\partial^\alpha c$ and integrate in $x$, we have 
\begin{equation*}
\begin{aligned}
{\|}\nabla_x\partial^\alpha c{\|}^2&=-\int_{\mathbb T^3}{( (\partial_t+v\cdot\nabla_x)(\bm I_\pm-\bm P_\pm)\partial^\alpha f, \xi_{c})_{L^2_v}\cdot\nabla_x\partial^\alpha cdx}
\\
&-\int_{\mathbb T^3}{(L_\pm(\bm I_\pm-\bm P_\pm)\partial^\alpha f,\xi_{c})_{L^2_v}\cdot \nabla_x\partial^\alpha cdx}+\int_{\mathbb T^3}{( \partial^\alpha N_\pm(f)  ,\xi_{c})_{L^2_v}\cdot \nabla_x\partial^\alpha c  dx}
\\
&\le  - \int_{\mathbb T^3}{(\partial_t(\bm I_\pm-\bm P_\pm)\partial^\alpha f, \xi_{c})_{L^2_v}\cdot\nabla_x\partial^\alpha cdx}
\\
& + C  {\|}\nabla_x\partial^\alpha c{\|}\cdot({\|}\nabla_x\partial^{\alpha}(I-P)f{\|}_{L^2_xL^2_{10}}+{\|}\partial^{\alpha}(I-P)f{\|}_{L^2_xL^2_{10}}+{\|}\partial^{\alpha}N_{{\|}}{\|}_{L^2_x } )
\\
&= - \frac{d}{dt}\int_{\mathbb T^3}{((\bm I_\pm-\bm P_\pm)\partial^\alpha f, \xi_{c})_{L^2_v}\cdot\nabla_x\partial^\alpha cdx}-\int_{\mathbb T^3}{( (\bm I_\pm-\bm P_\pm)\nabla_x\partial^\alpha f, \xi_{c})_{L^2_v}\cdot\partial_t\partial^\alpha cdx}
\\
&+C{\|}\nabla_x\partial^\alpha c{\|}\cdot({\|}\nabla_x\partial^{\alpha}(\bm I_\pm-\bm P_\pm)f{\|}_{L^2_xL^2_{10}}+{\|}\partial^{\alpha}(\bm I_\pm-\bm P_\pm)f{\|}_{L^2_xL^2_{10}}+{\|}\partial^{\alpha}N_{{\|}}{\|}_{L^2_x} ),
\end{aligned}
\end{equation*}
from \eqref{uni8},
\[
\partial_t\partial^\alpha c=\frac{1}{3}\nabla_x\partial^{\alpha}b-\frac{1}{6}( v\cdot\nabla_x(\bm I_\pm-\bm P_\pm)\partial^{\alpha}f,|v|^2)_{L^2_v}+\left( \partial^{\alpha}N_\pm(f),\frac{|v|^2-3}{6}\right)_{L^2_v},
\]
so from Cauchy-Schwartz inequality, for any $\epsilon>0$ small we have
\[
-\int_{\mathbb T^3}{( (\bm I_\pm-\bm P_\pm)\nabla_x\partial^\alpha f, \xi_{ci})_{L^2_v}\cdot\partial_t\partial^\alpha cdx}\le \frac{\varepsilon}{2}{\|}\nabla_x\partial^\alpha b{\|}^2+\frac{1}{\varepsilon}{\|}\nabla_x\partial^\alpha(\bm I_\pm-\bm P_\pm)f{\|}^2_{L^2_xL^2_{10}}+C\varepsilon{\|}\partial^\alpha N_{{\|}}{\|}_{L^2_{x}}^2,
\]
and we deduce that 
\begin{equation}\label{uni11}
\begin{aligned}
{\|}\nabla_x\partial^\alpha c{\|}^2&\le -\frac{d}{dt}\int_{\mathbb T^3}{((\bm I_\pm-\bm P_\pm)\partial^\alpha f, \xi_{ci})_{L^2_v}\cdot\nabla_x\partial^\alpha cdx}+ \varepsilon{\|}\nabla_x\partial^\alpha b{\|}^2+ \varepsilon{\|}\nabla_x\partial^\alpha c{\|}^2\\
&+\frac{C}{\varepsilon}({\|}\nabla_x\partial^{\alpha}(\bm I_\pm-\bm P_\pm)f{\|}_{L^2_xL^2_{10 }} ^2+{\|}\partial^{\alpha}(\bm I_\pm-\bm P_\pm)f{\|}^2_{L^2_xL^2_{10}}+{\|}\partial^{\alpha}N_{{\|}}{\|}_{L^2_x} ^2)
\end{aligned}
\end{equation}
we next compute $\nabla_x\partial^{\alpha}b$. For fixed $i$, we use \eqref{uni55},\eqref{abc0} to compute
\begin{equation}
\label{uni14}
\begin{aligned}
\Delta_x  b_i&=\sum_{j\ne i}\partial_{x_j x_j}  b_i+\partial_{x_ix_i} b_i
\\
&=\sum_{j\ne i}\left(-\partial_{x_ix_j} b_j+\partial_{x_j}( \gamma_{1ij}+ \gamma_{2ij})\right)+\partial_{x_i} (\gamma_{1i}+\gamma_{2i})-\partial_t\partial_{x_i} c
\\
&=\sum_{j\ne i}\left(\partial_t\partial_{x_i} c-\partial_{x_i} ( \gamma_{1j}+ \gamma_{2j})\right)-\partial_t\partial_{x_i} c+\sum_{j\ne i}\partial_{x_j}  ( \gamma_{1j}+ \gamma_{2j})+\partial_{x_i} (\gamma_{1i}+\gamma_{2i})
\\
&=\partial_t\partial_{x_i}  c+\sum_{j\ne i}\left(\partial_{x_j}   ( \gamma_{1j}+ \gamma_{2j})-\partial_{x_i}(\gamma_{1j}+\gamma_{2j})\right)+\partial_{x_i}(\gamma_{1i}+\gamma_{2i})
\\
&=\sum_{j\ne i}\left(\partial_{x_j} ( \gamma_{1j}+ \gamma_{2j})-\partial_{x_i} (\gamma_{1j}+\gamma_{2j})  \right)-\partial_{x_ix_i}   b_i+2\partial_{x_i}(\gamma_{1i}+\gamma_{2i}),
\end{aligned}
\end{equation}
we can rewrite the linear terms as $\gamma_{1i}, \gamma_{1j}, \gamma_{1ij}$ as
\[
\sum_{j\ne i}\left(\partial_{x_j}\gamma_{1ij}-\partial_{x_i} \gamma_{1j}  \right)+2\partial_{x_i}\gamma_{1i}=\sum\partial_j\left(-((\partial_t+v\cdot\nabla_x)(\bm I_\pm-\bm P_\pm)f,\xi_{ij})_{L^2_v}-(L_\pm(\bm I_\pm-\bm P_\pm)f,\xi_{ij})_{L^2_v}\right),
\]
here $\xi_{ij}$ are linear combinations of the basis $\{1, v, |v|^2\}$. Similarly as \eqref{ineqi}, we have
$|(Lf, \xi_{ij})_{L^2_v}|\lesssim {\|}f{\|}_{L^2_{10}}$. Take $\partial^\alpha$ on both side of \eqref{uni14}, multiply  by $\partial^\alpha b_i$ and integrate in $x$, we deduce that
\begin{equation*}
\begin{aligned}
{\|}\nabla_x\partial^\alpha b{\|}^2&=\sum_{i, j}    \left(-\int_{\mathbb T^3}{\sum_{j\ne i}\left(\partial_{x_j}\partial^\alpha(\gamma_{1ij}+\gamma_{2ij})-\partial_{x_i}\partial^\alpha(\gamma_{1j}+\gamma_{2j})\right)\cdot\partial^\alpha b_idx}\right)\\
&- \sum_{i, j}\left(\int_{\mathbb T^3}{|\partial_{x_i}\partial^\alpha b_{i}|^2dx}-2\int_{\mathbb T^3}{\partial_{x_i}\partial^\alpha( \gamma_{1i}+ \gamma_{2i})\cdot\partial^\alpha b_idx}\right)
\\
&\le \sum_{i, j} \int_{\mathbb T^3}{(\partial_t(\bm I_\pm-\bm P_\pm)\partial_{x_j}\partial^\alpha f, v_iv_j)_{L^2_v}\cdot\partial^\alpha b_i dx}
\\
&+C{\|}\nabla_x\partial^\alpha b{\|}\cdot({\|}\nabla_x\partial^{\alpha}(\bm I_\pm-\bm P_\pm)f{\|}_{L^2_xL^2_{10}}+{\|}\partial^{\alpha}(I-P)f{\|}_{L^2_xL^2_{10}}+{\|}\partial^{\alpha}N_{{\|}}{\|}_{L^2_x})
\\
&=\sum_{i.j}\frac{d}{dt}\int_{\mathbb T^3}{( (\bm I_\pm-\bm P_\pm)\partial_{x_j}\partial^\alpha f, v_iv_j)_{L^2_v}\cdot\partial^\alpha b_idx}-\int_{\mathbb T^3}{( (\bm I_\pm-\bm P_\pm)\partial_{x_j}\partial^\alpha f, v_iv_j)_{L^2_v}\cdot\partial_t\partial^\alpha b_idx}
\\
&+C{\|}\nabla_x\partial^\alpha b{\|}\cdot({\|}\nabla_x\partial^{\alpha}(\bm I_\pm-\bm P_\pm)f{\|}_{L^2_xL^2_{10 }}+{\|}\partial^{\alpha}(\bm I_\pm-\bm P_\pm)f{\|}_{L^2_xL^2_{10 }}+{\|}\partial^{\alpha}N_{{\|}}{\|}_{L^2_x}),
\end{aligned}
\end{equation*}
from \eqref{uni8},
\[
\partial_t\partial^\alpha b_i=-\nabla_{x_i}\partial^{\alpha}a-2\nabla_{x_i}\partial^{\alpha}c-\nabla_{x_i}\partial^{\alpha}\phi-(v\cdot\nabla_x(\bm I_\pm-\bm P_\pm)\partial^{\alpha}f ,v_i)_{L^2_v}+(\partial^{\alpha}N_\pm(f),v_i)_{L^2_v},
\]
so from Cauchy-Schwartz inequality,
\begin{equation*}
\begin{aligned}
-\int_{\mathbb T^3}{( (\bm I_\pm-\bm P_\pm)\partial_{x_j}\nabla_x\partial^\alpha f, v_iv_j)_{L^2_v}\cdot\partial_t\partial^\alpha b_idx}&\le \frac{\varepsilon}{2}{\|}\nabla_x\partial^\alpha a+2\nabla_x\partial^\alpha c{\|}^2+\frac{\varepsilon}{2}{\|}\nabla_x\partial^{\alpha}\phi{\|}^2
\\
&+\frac{1}{\varepsilon}{\|}\nabla_x\partial^\alpha(\bm I_\pm-\bm P_\pm)f{\|}^2_{L^2_xL^2_{10}}+\frac{\varepsilon}{2}{\|}\partial^\alpha N_{{\|}}{\|}_{L^2_x}^2,
\end{aligned}
\end{equation*}
and we deduce that
\begin{equation}\label{uni15}
\begin{aligned}
{\|}\nabla_x\partial^\alpha b{\|}^2&\le- \sum_{i,j}\frac{d}{dt}\int_{\mathbb T^3}{((\bm I_\pm-\bm P_\pm)  \partial_{x_j}  \partial^\alpha f, v_iv_j)_{L^2_v}\cdot\partial_{x_j}\partial^\alpha b_idx}+ \varepsilon{\|}\nabla_x\partial^\alpha b{\|}^2+\varepsilon{\|}\nabla_x \partial^{\alpha}\phi{\|}^2\\
&+ \varepsilon{\|}\nabla_x\partial^\alpha a+2\nabla_x\partial^\alpha c{\|}^2+\frac{C}{\varepsilon}({\|}\nabla_x\partial^{\alpha}(\bm I_\pm-\bm P_\pm)f{\|}_{L^2_xL^2_{10}}^2+{\|}\partial^{\alpha}(\bm I_\pm-\bm P_\pm)f{\|}^2_{L^2_xL^2_{10}}+{\|}\partial^{\alpha}N_{{\|}}{\|}_{L^2_x}^2),
\end{aligned}
\end{equation}
the final step is to compute $\nabla_x\partial^{\alpha}a$. Take the addition and difference of \eqref{abc2} over $\pm$,  we have 
\begin{equation}
\label{aplus1}
\nabla_xa_++\nabla_xa_{-}=-2\partial_tb+\sum_\pm (-(\partial_t+v\cdot\nabla_x)(\bm I_\pm-\bm P_\pm)f- L_\pm(\bm I_\pm-\bm P_\pm)f+N_\pm(f), \xi_a)_{L^2_v}
\end{equation}
and
\begin{equation}
\label{aplus2}
\nabla_xa_+-\nabla_xa_{-}=2E+\sum_\pm\pm (-(\partial_t+v\cdot\nabla_x)(\bm I_\pm-\bm P_\pm)f- L_\pm(\bm I_\pm-\bm P_\pm)f+N_\pm(f), \xi_a)_{L^2_v}
\end{equation}
apply  $\partial^\alpha$ derivative on both side of \eqref{aplus1} and take inner product with $\nabla_x\partial^\alpha(a_++a_{-})$, we deduce that
\begin{equation}
\label{aplus3}
\begin{aligned}
&{\|}\nabla_x\partial^\alpha(a_++a_{-}){\|}^2=-\int{\partial_t(\partial^\alpha b+((\bm I_\pm-\bm P_\pm)\partial^\alpha f,\xi_a)_{L^2_v})\cdot\nabla_x\partial^\alpha(a_++a_{-})dx}\\
&-\int{(v\cdot\nabla_x(\bm I_\pm-\bm P_\pm)\partial^\alpha f- L_\pm(\bm I_\pm-\bm P_\pm)\partial^\alpha f+\partial^\alpha N_\pm(f), \xi_a)_{L^2_v}\cdot\nabla_x\partial^\alpha(a_++a_{-})dx}\\
\end{aligned}
\end{equation}
after integration by parts about $\partial_t$ and $\nabla_x$, and applying $\eqref{uni55}_1$ and Cauchy-Schwartz inequality, we have
\begin{equation}
\label{aplus4}
\begin{aligned}
&{\|}\nabla_x\partial^\alpha(a_++a_{-}){\|}^2+\frac d{dt}\sum_\pm\int{(\partial^\alpha b+((\bm I_\pm-\bm P_\pm)\partial^\alpha f,\xi_a)_{L^2_v})\cdot\nabla_x\partial^\alpha(a_++a_{-})dx}\\
&\lesssim\frac C\varepsilon({\|}\partial^\alpha(\bm I_\pm-\bm P_\pm)f{\|}_{H^1_xL^2_{10}}+{\|}\partial^\alpha N_{{\|}}{\|}^2_{L^2_x})+\varepsilon{\|}\nabla_x\partial^\alpha b{\|}^2_{L^2_x}
\end{aligned}
\end{equation}
next, from integration by parts,
\[
(\partial^\alpha E, \nabla_x\partial^\alpha(a_+-a_{-}))_{L^2_x}=-(\nabla_x\partial^\alpha E, \partial^\alpha(a_+-a_{-}))_{L^2_x}={\|}\partial^\alpha(a_+-a_{-}){\|}^2_{L^2_x}
\]
apply $\partial^\alpha$ derivative on both side of \eqref{aplus2} and take inner product with $\nabla_x\partial^\alpha(a_+-a_{-})$, we deduce that
\begin{equation}
\label{aplus5}
\begin{aligned}
&{\|}\nabla_x\partial^\alpha(a_++a_{-}){\|}^2+2{\|}\partial^\alpha(a_+-a_{-}){\|}^2_{L^2_x}+\frac d{dt}\sum_\mp\int{(((\bm I_\pm-\bm P_\pm)\partial^\alpha f,\xi_a)_{L^2_v})\cdot\nabla_x\partial^\alpha(a_++a_{-})dx}\\
&\lesssim\frac C\varepsilon({\|}\partial^\alpha(\bm I_\pm-\bm P_\pm)f{\|}_{H^1_xL^2_{10}}+{\|}\partial^\alpha N_{{\|}}{\|}^2_{L^2_x})+\varepsilon{\|}\nabla_x\partial^\alpha b{\|}^2_{L^2_x}
\end{aligned}
\end{equation}

We next calculate $E$. Remind \eqref{equationG2} and $Lf=L(\bm I-\bm P)f$, we have
\begin{equation*}
\begin{aligned}
2{\|}E{\|}^2_{L^2_x}&=(E, \partial_tG+\nabla_x(a_+-a_{-})+\nabla_x\cdot\Theta((\bm I-\bm P)f\cdot [1,-1]))_{L^2_x}-(E, ((N+Lf)\cdot[1,-1], v)_{L^2_v})_{L^2_x}
\\
&=\partial_t(E,G)_{L^2_x}-(\nabla_x\times B, G)_{L^2_x}+{\|}G{\|}^2_{L^2_x}+(E, \nabla_x(a_+-a_{-}))_{L^2_x}\\&+(E,\nabla_x\cdot\Theta((\bm I-\bm P)f\cdot [1,-1]))_{L^2_x}-(E, ((N+L(\bm I-\bm P)f)\cdot[1,-1], v)_{L^2_v})_{L^2_x}
\end{aligned}
\end{equation*}
so 
\begin{equation}
\label{equationE1}
\begin{aligned}
-\partial_t(E,G)_{L^2_x}+\lambda{\|}E{\|}^2_{L^2_x}&\lesssim \varepsilon{\|}\nabla_x\times B{\|}^2_{L^2_x}+\frac C\varepsilon{\|}(\bm I-\bm P)f{\|}^2_{H^1_xL^2_D}+{\|}\nabla_x(a_+-a_{-}){\|}^2_{L^2_x}+{\|}(N,\xi)_{L^2_v}{\|}^2_{L^2_x}
\end{aligned}
\end{equation}
for the first derivative of $E$, from integration by parts in $x$, we have
\[
(\nabla_x(\nabla_x\times B), \nabla_xG)_{L^2_x}=(\nabla_x\times B, \nabla_x\cdot\nabla_xG)_{L^2_x}\lesssim {\|}\nabla_x\times B{\|}_{L^2_x}{\|}(\bm I-\bm P)f{\|}_{H^2_xL^2_D}
\]
so
\begin{equation}
\label{equationE2}
\begin{aligned}
2{\|}E{\|}^2_{L^2_x}&=(\nabla_xE, \partial_t\nabla_xG+\nabla^2_x(a_{+}-a_{-})+\nabla_x\nabla_x\cdot\Theta((\bm I-\bm P)f\cdot[1,-1])-( (\nabla_x N+L\nabla_xf)\cdot [1,-1],v)_{L^2_v})_{L^2_x}\\
&=\partial_t(\nabla_xE,\nabla_xG)_{L^2_x}-(\nabla_x(\nabla_x\times B), \nabla_xG)_{L^2_x}+{\|}\nabla_x G{\|}^2_{L^2_x}+(\nabla_xE, \nabla_x^2(a_+-a_{-}))_{L^2_x}\\
&+(\nabla_xE, \nabla_x^2\cdot\Theta((\bm I-\bm P)f\cdot[1,-1]))_{L^2_x}-(\nabla_xE,  (\nabla_x N+L\nabla_xf)\cdot [1,-1],v)_{L^2_v})_{L^2_x}
\end{aligned}
\end{equation}
and we obtain that for any $\eta>0$, there exists $C_\eta>0$ that satisfies
\begin{equation}
\label{equationE3}
-\partial_t(\nabla_xE, \nabla_xG)+\lambda{\|}\nabla_xE{\|}^2_{L^2_x}\lesssim \varepsilon{\|}\nabla_x\times B{\|}^2_{L^2_x}+\frac C\varepsilon{\|}(\bm I-\bm P)f{\|}^2_{H^2_xL^2_D}+{\|}a_{+}-a_{-}{\|}^2_{H^2_x}+{\|}(N,\xi)_{L^2_v}{\|}^2_{H^1_x}
\end{equation}
for the second derivatives of $E$, we have
\begin{equation}
\label{equationE4}
\begin{aligned}
2{\|}\nabla_x^2E{\|}^2_{L^2_x}&=(\nabla^2_xE, \partial_t\nabla^2_xG+\nabla^3_x(a_+-a_{-})+\nabla^2_x\nabla_x\cdot\Theta((\bm I-\bm P)f\cdot[1,-1])
\\
&-((\nabla^2_xN+L\nabla_xf)\cdot [1,-1], v)_{L^2_v})_{L^2_x}
\end{aligned}
\end{equation}
from integration by parts, we have
\[
|(\nabla^2_x(\nabla_x\times B),\nabla^2_x G)_{L^2_v}|\lesssim {\|}\nabla^2_x B{\|}_{L^2_x}{\|}\nabla^3_x(\bm I-\bm P)f{\|}_{L^2_xL^2_D}
\]
so
\begin{equation}
\label{equationE5}
\begin{aligned}
-\partial_t(\nabla^2_xE,\nabla^2_xG)_{L^2_x}+\lambda{\|}\nabla^2_xE{\|}^2_{L^2_x}\lesssim \varepsilon{\|}\nabla^2_x B{\|}^2_{L^2_x}+\frac C
\varepsilon{\|}(\bm I-\bm P)f{\|}^2_{H^3_xL^2_D}+{\|}a_{+}-a_{-}{\|}^2_{H^3_x}+{\|}(N,\xi)_{L^2_v}{\|}^2_{H^3_x}
\end{aligned}
\end{equation}

Now let us focus on $B$. Notice that ${\|}\nabla_x\times B{\|}={\|}\nabla_x B{\|}$ because $B$ is divergence free. From integrating by parts, we have
\begin{equation}
\label{equationeb}
(\nabla_x\times B, E)=(B, \nabla_x\times E),\quad (\partial_t (\nabla_x\times B), E)=(\partial_t B, \nabla_x\times E)
\end{equation}
so we have
\begin{equation*} 
\begin{aligned}
{\|}\nabla_x\times B{\|}^2&=(\nabla_x\times B, \partial_tE+G)=\partial_t (\nabla_x\times B, E)-(\partial_t B,\nabla_x\times E)+(\nabla_x\times B, G)\\&=\partial_t (\nabla_x\times B, E)+{\|}\nabla_x\times E{\|}^2+(\nabla_x\times B, G)\\
\end{aligned}
\end{equation*}
this implies
\begin{equation}
\label{EB1}
-\partial_t (\nabla_x\times B, E)+\lambda {\|}\nabla_x\times B{\|}^2\lesssim {\|}\nabla_x\times E{\|}^2+{\|}(\bm I-\bm P)f{\|}^2_{L^2_xL^2_D}
\end{equation}
next, for the second deriavative on $B$, we have
\begin{equation*}
\begin{aligned}
\Vert \nabla_x^2  B\Vert^2 &= \Vert \nabla_x \times (\nabla_x \times B)\Vert^2 = (\nabla_x \times( \nabla_x \times B), \nabla_x \times(\partial_t E +G)) 
\\
 &= \partial_t(\nabla_x \times (\nabla_x \times B),\nabla_x \times E) - (\nabla_x\times\partial_t B), \nabla_x\times(\nabla_x \times E))  +(\nabla_x \times (\nabla_x \times B), \nabla_x \times G)
\\&
=\partial_t(\nabla_x \times (\nabla_x \times B), \nabla_x \times E)  +{\|}\nabla_x\times(\nabla_x \times E){\|}^2+(\nabla_x \times (\nabla_x \times B), \nabla_x \times G)
\end{aligned}
\end{equation*}
and similarly as \eqref{EB1}, we have
\begin{equation}
\label{EB2}
-\partial_t(\nabla_x \times (\nabla_x \times B), \nabla_x \times E)+\lambda \Vert \nabla_x^2  B\Vert^2 \lesssim {\|}\nabla^2_x E{\|}^2+ {\|}(\bm I-\bm P)f{\|}^2_{H^1_xL^2_D}
\end{equation}
moreover,  we obtain from Poincar\'e inequality that
\begin{equation}
\label{equationE6}
{\|}B{\|}_{L^2_x}\lesssim {\|}\nabla_xB{\|}_{L^2_x}\lesssim{\|}\nabla_x\times B{\|}_{L^2_x}
\end{equation}
so after taking $\eta, \kappa>0$ small enough, we take linear combination $ \eqref{equationE1}+\eqref{equationE3}+\kappa\times\eqref{EB1}$ and use \eqref{equationE6} to obtain
\begin{equation}
\label{equationE7}
\begin{aligned}
&-\partial_t(E,G)_{L^2_x}-\partial_t(\nabla_xE,\nabla_xG)_{L^2_x}-\kappa\partial_t(\nabla_x\times B, E)_{L^2_x}+\lambda {\|}(E,B){\|}^2_{H^1_x}\\&
\lesssim{\|}(\bm I-\bm P)f{\|}^2_{H^2_xL^2_D}+{\|}a_+-a_{-}{\|}^2_{H^2_x}+{\|}(N, \xi)_{L^2_v}{\|}^2_{H^1_x}
\end{aligned}
\end{equation}
moreover, take linear $\eqref{equationE1}+ \eqref{equationE3}+\eqref{equationE5}+\kappa\times\eqref{EB1}+\kappa\times\eqref{EB2}$ and apply \eqref{equationE6}, we have a refined esimate with the second derivative terms of $(E, B)$ as follows
\begin{equation}
\label{equationE8}
\begin{aligned}
&-\partial_t(E,G)_{L^2_x}-\partial_t(\nabla_xE,\nabla_xG)_{L^2_x}-\partial_t(\nabla^2_xE,\nabla^2_xG)_{L^2_x}\\
&-\kappa\partial_t(\nabla_x\times B, E)_{L^2_x}-\partial_t(\nabla_x\times(\nabla_x\times B), \nabla_x\times E)_{L^2_x}+\lambda {\|}(E,B){\|}^2_{H^2_x}\\&
\lesssim{\|}(\bm I-\bm P)f{\|}^2_{H^3_xL^2_D}+{\|}a_+-a_{-}{\|}^2_{H^3_x}+{\|}(N, \xi)_{L^2_v}{\|}^2_{H^2_x}
\end{aligned}
\end{equation}

Now we come to prove \eqref{macro2}. Note that $|a_++a_{-}|^2+|a_+-a_{-}|^2=2|a_+|^2+2|a_{-}|^2$. Recall the Poincar\'e inequality we have
\begin{equation}
\label{apoincare}
\Vert  a\Vert \lesssim \Vert  \nabla_x a \Vert,\quad \Vert  b\Vert \lesssim \Vert  \nabla_x b \Vert, \quad \Vert  c\Vert \lesssim \Vert  \nabla_x c \Vert + \int_{\T^3} c dx   \lesssim  \Vert  \nabla_x c \Vert +  \Vert  E\Vert^2+ \Vert  B\Vert^2,\quad 
\Vert  B\Vert  \lesssim \Vert  \nabla_x  B\Vert
\end{equation}
for $K=2$, we define the energy $E_{2,int}$ as
\begin{equation*}
\begin{aligned}
&\sum_{\pm}\sum_{|\alpha|\le 2}(\int {((\bm I_\pm-\bm P_\pm)\partial^\alpha f,\xi_c)_{L^2_v}\cdot \nabla_x\partial^\alpha cdx}+\sum_{i,j}\int {((\bm I_\pm-\bm P_\pm)\partial^\alpha f,\xi_{ij})_{L^2_v}\cdot \partial_{x_j}\partial^\alpha b_idx}\\
&+\lambda\int{ (\partial^\alpha b+((I-P)\partial^\alpha f,\xi_a)_{L^2_v})\cdot\nabla_x\partial^\alpha(a_++a_{-})dx})\mp \lambda\int{ (((I-P)\partial^\alpha f,\xi_a)_{L^2_v})\cdot\nabla_x\partial^\alpha(a_+-a_{-})dx})\\
&-\kappa (E,G)_{L^2_x}-\kappa (\nabla_xE,\nabla_xG)_{L^2_x}-\kappa^2 (\nabla_x\times B, E)_{L^2_x}
\end{aligned}
\end{equation*}
for $K=3$, we define the energy $E_{3,int}$ as
\begin{equation*}
\begin{aligned}
&\sum_{\pm}\sum_{|\alpha|\le 3}(\int {((\bm I_\pm-\bm P_\pm)\partial^\alpha f,\xi_c)_{L^2_v}\cdot \nabla_x\partial^\alpha cdx}+\sum_{i,j}\int {((\bm I_\pm-\bm P_\pm)\partial^\alpha f,\xi_{ij})_{L^2_v}\cdot \partial_{x_j}\partial^\alpha b_idx}\\
&+\lambda\int{ (\partial^\alpha b+((I-P)\partial^\alpha f,\xi_a)_{L^2_v})\cdot\nabla_x\partial^\alpha(a_++a_{-})dx})\mp \lambda\int{ (((I-P)\partial^\alpha f,\xi_a)_{L^2_v})\cdot\nabla_x\partial^\alpha(a_+-a_{-})dx})\\
&-\kappa (E,G)_{L^2_x}-\kappa (\nabla_xE,\nabla_xG)_{L^2_x}-\kappa (\nabla_x^2E,\nabla_x^2G)_{L^2_x}-\kappa^2 (\nabla_x\times B, E)_{L^2_x}-\kappa^2(\nabla_x\times (\nabla_x\times B), \nabla_x\times E)_{L^2_x}
\end{aligned}
\end{equation*}
so for $K=2$, we choose $\varepsilon>0$ small enough, take linear combination $\eqref{uni11}+\eqref{uni15}+\eqref{aplus4}+\eqref{aplus5}+\kappa\times\eqref{equationE7}$ and \eqref{apoincare} to obtain
\begin{equation*}
\partial_tE_{2,int}(t)+\lambda{\|}a_{\pm},b,c{\|}^2_{H^2_x}+\lambda {\|}E, B{\|}^2_{H^1_x}\lesssim {\|}(\bm I-\bm P)f{\|}^2_{H^2_xL^2_D}+{\|}(g,\xi)_{L^2_v}{\|}^2_{H^2_x}+{\|}E, B{\|}^4_{L^2_x}
\end{equation*}
for $K=3$, we choose $\varepsilon>0$ small enough, take linear combination $\eqref{uni11}+\eqref{uni15}+\eqref{aplus4}+\eqref{aplus5}+\kappa\times\eqref{equationE8}$ and \eqref{apoincare} to obtain
\begin{equation*}
\partial_tE_{3,int}(t)+\lambda{\|}a_{\pm},b,c{\|}^2_{H^3_x}+\lambda {\|}E, B{\|}^2_{H^2_x}\lesssim {\|}(\bm I-\bm P)f{\|}^2_{H^3_xL^2_D}+{\|}(g,\xi)_{L^2_v}{\|}^2_{H^3_x}+{\|}E, B{\|}^4_{L^2_x}
\end{equation*}
and the lemma is proved.
\end{proof}
\section{Semigroup of the linearized operator}
\label{sec4}
In this Section, we will derive the estimate on the semigroup generated by the linearized operator $L_1$ given by \eqref{L1}. 
To take into account both the particle distribution and the electromagnetic field, we consider the equation 
\begin{equation}
\label{equationg0}
\partial_t\bm{g} =  L_1\bm{g}, 
\end{equation}
where we have denoted 
\begin{align}
\label{equationg11}
\begin{aligned}
\bm{g}&= (f_+,f_-, E, B),\\
L_1 (\bm{g}) &= \Big(-v\cdot \nabla_x f_{+} + E \cdot v \mu +L_+ f,\,-v\cdot \nabla_x f_{-} - E \cdot v \mu +L_- f,\\&\qquad\quad\nabla_x \times B -\int_{\R^3} v (f_+-f_{-}) dv ,\, -\nabla_x \times E\Big), 
\end{aligned}
\end{align}
and denote by $\g=S_{L_1}(t)\bm{ g}_0$ the solution of the equation
\begin{equation}
\label{equationg2}
\partial_t \bm{ g} =L_1 \bm{g}, \quad \bm{g}(0)=\bm{ g}_0, \quad \nabla_x \cdot \g_E = \int_{\R^3} (f_+-f_{-})(v) dv   , \quad  \nabla_x \cdot \g_B =0.
\end{equation}
In the following, we'll use $\bm{ g}_{f}, \bm{ g}_E, \bm{ g}_B$ to denote the variables of $\bm g$, and simply write $[f, E, B]=[\bm{ g}_{f}, \bm{ g}_E, \bm{ g}_B]$ in the same line accordingly when there's no confusion. Recall also that we use the norm $\|\g\|^2_{H^k_xH^m_k}=\|f\|^2_{H^k_xH^m_k}+\|[E,B]\|^2_{H^k_x}$. 
For the estimate of semigroup $S_{L_1}(t)$ generated by $L_1$, we prove the following lemma.

\begin{lem}
\label{lemsemigroup}
Let $k\ge k_0$ and $\al\ge 0$. Then for any suitable function $\g$, we have large-time behavior:
\begin{align}\label{semidecay}
    \Vert \pa^\al S_{L_1}(t)\bm g\Vert_{H^2_xL^2_k}   &\lesssim (1+t)^{-1/2} \Vert\pa^\al\bm g \Vert_{H^3_xL^2_{4}},  
\end{align}
and we have a stronger result about integrability and regularizing behavior:
\begin{equation}
\label{spectralgapl1}
\left ( \int_0^\infty \Vert\pa^\al S_{L_1}(t)\bm g  \Vert_{H^2_xL^2_k}^2 dt \right)^{\frac 1 2}   
\lesssim  \min\Big\{\|\pa^\al\g\|_{H^3_xL^2_{k+10}},\,\Vert\pa^\al \bm g \Vert_{H^3_xH^{-s}_{k-\gamma/2}}\Big\}.
\end{equation}
moreover, for the $f$ coordinate, we have a better estimate:
\begin{equation}\label{equationl10}
\left (  \int_0^\infty \Vert\pa^\al( S_{L_1} (t)\bm g)_f  \Vert_{H^2_xL^2_k}^2 dt \right)^{\frac 1 2}  \lesssim\Vert \pa^\al\bm g \Vert_{H^2_xH^{-s}_{k-\gamma/2}}. 
\end{equation}
\end{lem}

\begin{proof}
Note that the operators $\partial_x$ and $L_1$ are commutative, so we only consider the case $\alpha=0$. The proof below also works if $H^2_x$ is replaced by $H^{k+2}_x$ and $H^3_x$ is replaced by $H^{k+3}_x$ for any $k\in\mathbb N$. 
For the function $\bm g$ defined in \eqref{equationg11}, an arbitrary weight function $m$ and the projection $\bm P$ given in \eqref{projection}, we denote $m(\bm I-\bm P) g:= (m(\bm I-\bm P) f_\pm, E, B )$, in other words, the weight function $m$ and the projection $\bm P$ act only on the variable $f$. 
Note also that for polynomial weight, we choose $k\ge k_0$ while for exponential weight, one can choose arbitrary $k\ge 0$ below.

\smallskip 
First, we define linear operators $\mathcal A, \mathcal B$, which split $L_1$. Let $M>0$ be a large constant to be chosen, $\chi\in C^\infty_c(\mathbb R)$ be the smooth truncation function satisfying $1_{[-1,1]} \le  \chi \le 1_{[-2,2]}$ and we denote $\chi_R(\cdot) :=\chi(\cdot/R)$ for $R>0$. Then we define
\begin{align*}
\mathcal A \bm g:&= \big(-v\cdot \nabla_x f_+ +L_+ f  -M \chi_R f_+, 
\,-v\cdot \nabla_x f_- +L_- f  -M \chi_R f_-,\, \nabla_x \times B -E ,\, -\nabla_x \times E -B\big),\\
 K_1 \bm g :&= \big(M\chi_R f_+,M\chi_R f_-, 0, 0\big), 
\end{align*}
and 
\begin{align*}
\mathcal B \bm g:&= \big(-v\cdot \nabla_x f_+ +L_+ f ,\,-v\cdot \nabla_x f_- +L_- f ,\, \nabla_x \times B -E ,\, -\nabla_x \times E -B\big),\\
K_2 \bm g :&=\Big( E \cdot v \mu,\, -E \cdot v \mu,\, -E - \int_{\R^3}v (f_+-f_{-}) dv,\, B\Big), 
\end{align*}
which give $\mathcal A+K_1 =\mathcal B$ and $\mathcal B+K_2= L_1$. Next, we denote by $S_{\mathcal A}(t)g_0$ the solution to the equation
\begin{equation}
\label{SAg0}
\partial_t \bm g =\mathcal A \bm g, \quad\bm g(0)=\bm g_0, \quad \nabla_x \cdot E = \int_{\R^3} (f_+-f_{-})(v) dv   , \quad  \nabla_x \cdot B =0,
\end{equation}
and denote by $S_{\mathcal B}(t)g_0$ the solution to the equation
\begin{equation}
\label{SBg0}
\partial_t\bm g =\mathcal B \bm g, \quad\bm g(0)=\bm g_0, \quad \nabla_x \cdot E = \int_{\R^3} (f_+-f_{-})(v) dv   , \quad  \nabla_x \cdot B =0.
\end{equation}
Thus, from the definition of the semigroup, we have
\begin{equation}
\label{defsemigroup}
\partial_tS_{\mathcal A}(t)\bm g-\mathcal AS_{\mathcal A}(t)\bm g=0,\quad \partial_tS_{\mathcal B}(t)\bm g-\mathcal BS_{\mathcal B}(t)\bm g=0,\quad \partial_tS_{L_1}(t)\bm g-\mathcal L_1S_{L_1}(t)\bm g=0,  
\end{equation}
and the Duhamel's principle implies 
\begin{align}
    \label{Duha}
    S_{\mathcal B}(t)=S_{\mathcal A}(t)+\int^t_0S_{\mathcal B}(t-s)K_1S_{\mathcal A}(s)\,ds,\quad S_{L_1}(t)=S_{\mathcal B}(t)+\int^t_0S_{L_1}(t-s)K_1S_{\mathcal B}(s)\,ds.
\end{align}

\smallskip\noindent{\bf Step 1: Estimate of $S_{\mathcal A}(t)$.}
First, after direct computation by choosing $M>0$ sufficiently large and using \eqref{esL}, there exists $\lambda_0>0$ such that
\begin{align}\label{48}\notag
(\mathcal A\bm g, \bm g)_{L^2_xL^2_k} &= \sum_\pm(L_\pm f-M\chi_R f_\pm, f_\pm)_{L^2_xL^2_k}   + (\nabla_x \times B,E)_{L^2_x} -\Vert E \Vert^2_{L^2_x} -(\nabla_x \times E,B )_{L^2_x}  -\Vert B \Vert^2 _{L^2_x} 
\\
&\le - \lambda \Vert f \Vert_{L^2_xH^s_{k+\gamma/2}}^2 -\Vert E \Vert^2_{L^2_x} -\Vert B \Vert^2_{L^2_x}  \le -\lambda \Vert \bm g \Vert_{L^2_xL^2_{k+\gamma/2}}^2. 
\end{align}
Thus, if $\gamma \ge 0$,  from the $L^2_xL^2_k$ energy estimate of \eqref{defsemigroup}$_1$, 
\begin{equation}
\label{equationa1}
\Vert S_\mathcal A (t)\bm g_0\Vert_{L^2_xL^2_k} \lesssim  e^{-\lambda t}\Vert  \bm g_0\Vert_{L^2_xL^2_k}; 
\end{equation}
if $-3\le\gamma<0$, 
from interpolation ${\|}f{\|}_{L^2_v}\le {\|}f\langle v\rangle^{\gamma/2}{\|}^\frac 23_{L^2_v}{\|}f\langle v\rangle^6{\|}^\frac 13_{L^2_v}$, we have $$(\mathcal A\bm g, \bm g)_{L^2_xL^2_k}\le -\lambda \Vert \bm g\Vert_{L^2_xL^2_{k+6}}^{-1}\Vert \bm g \Vert_{L^2_xL^2_{k}}^3.$$
After solving the ODE of $L^2_xL^2_k$ energy estimate of \eqref{defsemigroup}$_1$, i.e. $$\pa_t\|\g\|_{L^2_xL^2_k}^2\le -2\lam\big(\sup_{0\le s\le t}\|\g\|_{L^2_xL^2_{k+6}}\big)^{-1}\|\g\|_{L^2_xL^2_k}^3,$$ we have
\begin{equation}
\label{equationa2}
\Vert S_\mathcal A (t) \bm g_0\Vert_{L^2_xL^2_k} \lesssim (1+t)^{-2}\sup_{0\le s\le t}\|S_\mathcal A (s)\g_0\|_{L^2_xL^2_{k+6}} \lesssim  (1+t)^{-2}\Vert \bm g_0\Vert_{L^2_xL^2_{k+6}}. 
\end{equation}
Therefore, the integrability can be obtained from \eqref{equationa1} and \eqref{equationa2}: for any $\al\ge 0$, 
\begin{align}\label{inteA}
    \int^t_0\|\pa^\al S_\mathcal A (t) \bm g(s)\|^2_{L^2_xL^2_k}\,ds
    \lesssim \|\pa^\al \bm g(s)\|^2_{L^2_xL^2_{k+6}}. 
\end{align}

\smallskip 
Next, for the regularizing effect, we perform the duality argument of $S_{\mathcal A}(t)$. Consider the dual $\mathcal A^*_{l}$ of $\mathcal A_l=\<v\>^l\mathcal A(\<v\>^{-l}f)$ for any $l\in\R$,
and let $\h =\<v\>^{2k-l}\g$; note that velocity weights here and below are applied to $\g_f$ only. Then $\big(\mathcal A^*_{l}\g,\<v\>^{2k}\g\big)_{L^2_xL^2_v} = 
	\big(\<v\>^l\g,\mathcal A(\<v\>^{2k-l}\g)\big)_{L^2_xL^2_v}=
	\big(\<v\>^{2l-2k}\h,\mathcal A \h\big)_{L^2_xL^2_v}.$
Take the $L^2_xL^2_v$ inner product of $\pa_tS_{\mathcal A^*_{l}}(t)\g-\mathcal A^*_lS_{\mathcal A^*_{l}}(t)\g=0$ with $\<v\>^{2k}S_{\mathcal A^*_{l}}(t)\g$, we have 
\[
\frac{1}{2}\pa_t\|\<v\>^{k}S_{\mathcal A^*_{l}}(t)\g\|^2_{L^2_xL^2_v} + \big(\mathcal A^*_lS_{\mathcal A^*_{l}}(t)\g,\<v\>^{2k}S_{\mathcal A^*_{l}}(t)\g\big)_{L^2_xL^2_v} = 0.
\]
Hence for $l-k\ge k_0$, we have $\frac{1}{2}\pa_t\|\<v\>^{l-k}\h\|^2_{L^2_xL^2_v} + \big(\<v\>^{2l-2k}\h, \mathcal A\h\big)_{L^2_xL^2_v} = 0,$ where $\h=\<v\>^{2k-l}S_{\mathcal A^*_{l}}(t)\g$.
Integrating over $t$ and using \eqref{48}, we have 
\begin{align}\label{intSA}
\lam \int^\infty_0\|\<v\>^{l-k}\h\|_{L^2_xH^s_{\gamma/2}}^2\,dt\le \|\<v\>^{l-k}\h|_{t=0}\|^2_{L^2_xL^2_v}\,\,\Rightarrow\,\,\lam \int^\infty_0\|\<v\>^{k}S_{\mathcal A^*_{l}}(t)\g\|_{L^2_xH^s_{\gamma/2}}^2\,dt\le \|\<v\>^{k}\g\|^2_{L^2_xL^2_v}.
\end{align}
Observe that if  $\pa_t\g=\mathcal A\g$, then $\pa_t \<v\>^l\g=\mathcal A_l\<v\>^l\g$.
Thus, $\<v\>^lS_{\mathcal A}(t)\g=S_{\mathcal A_l}(\<v\>^l\g)$. 
Moreover, by duality, we have $(S_{\mathcal A_l}\g_1,\g_2)_{L^2_xL^2_v} = (\g_1,S_{\mathcal A_l^*}\g_2)_{L^2_xL^2_{v}}$ for any vector-valued functions $\g_1,\g_2$ given in \eqref{equationg11}. Therefore, 
there exists a vector-valued function (given in \eqref{equationg11}) sequence $\{\varphi_n\}$ in Schwartz space such that $\|\varphi_n\|_{L^2_xL^2_{k}}\le 1$, we have from \eqref{intSA} that 
\begin{align}\label{intSA1}\notag
	\int^\infty_0\|\<v\>^{k} S_\mathcal A(t)\partial^\alpha \g\|_{L^2_xL^2_v}^2\,dt
	&= \int^\infty_0\lim_{n\to\infty}\big|\big(\<v\>^{2k}  S_{\mathcal A}(t) \partial^\alpha \g,\varphi_n\big)_{L^2_{x,v}}\big|^2\,dt\notag\\
 &\notag= \liminf_{n\to\infty}\int^\infty_0\big|\big(\<v\>^{2k}\partial^\alpha \g, S_{\mathcal A^*_{2k}}(t)\varphi_n\big)_{L^2_{x,v}}\big|^2\,dt\\
	&\notag\le\liminf_{n\to\infty}\int^\infty_0\|\<v\>^{k}\partial^\alpha \g\|^2_{L^2_xH^{-s}_{-\gamma/2}}\|\<v\>^kS_{\mathcal A^*_{2k}}(t)\varphi_n\|_{L^2_xH^s_{\gamma/2}}^2\,dt\\
 &\lesssim \|\<v\>^{k-\gamma/2}\partial^\alpha \g\|^2_{L^2_xH^{-s}_v}, 
\end{align}
for any $k\ge k_0$. So, noting that $\pa^\al$ commutes with $S_\mathcal A (t)$, we have proved that for any $\al\ge0$, 
\begin{equation}\label{duality}
\int_{0}^\infty \Vert \pa^\al S_\mathcal A (t) \bm g\Vert_{L^2_xL^2_k}^2 dt \lesssim \Vert\pa^\al \bm g \Vert_{L^2_xH^{-s}_{k-\gamma/2}}^2.
\end{equation}

\smallskip\noindent{\bf Step 2: Estimate of $S_{\mathcal B}(t)$.}
Next, we compute $\mathcal B$ and consider linearized equation $\pa_t\g=\mathcal B\g$. The estimates of $\mathcal B$ in the norm $L^2(\langle v\rangle ^k)$ and $L^2(\mu^{-1/2})$ are respectively as follows:
\begin{multline}
\label{equationbf1}
(\mathcal B\bm g, \bm g)_{H^2_xL^2_k}=\sum_\pm(L_\pm f, f_\pm)_{H^2_xL^2_k}  + (\nabla_x \times B,E)_{H^2_x} -\Vert E \Vert^2_{H^2_x} -(\nabla_x \times E,B)_{H^2_x} -\Vert B \Vert^2_{H^2_x} 
\\
\le - \lambda \Vert f \Vert_{H^2_xH^s_{k+\gamma/2}}^2 -\Vert E \Vert^2_{H^2_x}-\Vert B \Vert^2_{H^2_x}  + C \Vert f \Vert_{H^2_x L^2_v}^2 \le -\lambda \Vert \bm g \Vert_{H^2_xH^s_{k+\gamma/2}}^2+C \Vert f \Vert_{H^2_x L^2_v}^2, 
\end{multline}

and
\begin{align}
\label{equationbf2}\notag
(\mathcal B\bm g, \bm g)_{H^2_xL^2_v(\mu^{-1/2})} &= \sum_\pm(L_\pm f, f_\pm)_{H^2_xL^2_v(\mu^{-1/2})}  + (\nabla_x \times B,E)_{H^2_x} -\Vert E \Vert^2_{H^2_x} -(\nabla_x \times E,B)_{H^2_x}-\Vert B \Vert^2_{H^2_x} 
\\
&
\le -\lambda \Vert \mu^{-\frac{1}{2}}  (\bm I-\bm P \mu^{\frac{1}{2}})\bm g \Vert_{H^2_xH^s_{\gamma/2}}^2,  
\end{align}
where we used \eqref{esL} and \eqref{esLmu}. 
For the linear equation 
$\pa_t\g+\mathcal{B}\g=0$, the conservation law is 
\[
\int_{\T^3} \int_{ \R^3}  f_+ ( v) dv dx =0, \,\,\int_{\T^3} \int_{ \R^3}  f_{-} ( v) dv dx  =0,\]\[\int_{\T^3} \int_{ \R^3} v (f_++f_{-})(v) dv dx = 0,\,\, \int_{\T^3} \int_{ \R^3} |v|^2 (f_++f_{-})(v) dv dx=0. 
\]
Here the fourth equation does not contain $E$ and $B$, which is different from \eqref{conse2}.  Using the same arguments that lead to \eqref{macro3}, for $f$ as the solution of $\partial_t\g=\mathcal B\g$, there exists a functional $E_{K, int}(t)\lesssim \|f\|_{H^K_xL^2_v}$ such that 
\[
\partial _tE_{K, int}(t) + \lambda \|[a, b, c]\|^2_{H^K_x} 
\lesssim \|(\bm I-\bm P) f\|^2_{H^K_xL^2_{10}}, 
\]
where $[a,b,c]$ are given by \eqref{projection}. 
Together with 
\eqref{equationbf2}, we deduce that there exists a functional $E_{int, \mathcal B} (t)  \sim\Vert \mu^{-1/2} S_\mathcal B(t)\bm g\Vert_{H^2_xL^2_v}^2$ such that
\[
\partial_t E_{int, \mathcal B} (t)   + \lambda \Vert   \mu^{-1/2} S_\mathcal B(t)\bm g \Vert_{H^2_xH^s_{\gamma/2}}^2 \le 0. 
\]
Similar to the analysis of $\mathcal A$ in \eqref{equationa1} and \eqref{equationa2}, we obtain that
\begin{equation}
\label{equationa3}
\Vert  \mu^{-1/2}  S_\mathcal B(t) \bm g \Vert_{H^2_xL^2_v}\lesssim  e^{-\lambda t}\Vert  \mu^{-1/2}   \bm g \Vert_{H^2_xL^2_v}\,\,\text{for}\,\, \gamma\ge 0,
\end{equation}
\begin{equation}
\label{equationa4}
\Vert  \mu^{-1/2}  S_\mathcal B(t) \bm g \Vert_{H^2_xL^2_v}\lesssim  (1+t)^{-2}\Vert  \mu^{-1/2}   \bm g \Vert_{H^2_xL^2_6}\,\,\text{for}\,\,-3\le\gamma<0.
\end{equation} 
For operator $K_1$, we easily see that for any $\alpha\in\mathbb N$ ($M>0$ is fixed now),
\begin{equation}\label{equationk1}
\Vert \mu^{-1} K_1 f  \Vert_{H^\alpha} \lesssim \Vert  f  \Vert_{H^\alpha}.
\end{equation}
Thus, from Duhamel's principle \eqref{Duha}, we could compute that
\begin{align}
\label{420}\notag
\Vert S_\mathcal B(t) \bm g \Vert_{H^2_xL^2_k}   &\lesssim \Vert S_\mathcal A(t) \bm g \Vert_{H^2_xL^2_k}    + \int_{0}^t \Vert  S_\mathcal B(t-s) K_1 S_\mathcal A(s) \bm g \Vert_{H^2_xL^2_k}\,ds
\\
&\notag\lesssim  (1+t)^{-2} \Vert\bm g \Vert_{H^2_xL^2_{k+6} }    + \int_{0}^t \Vert  \mu^{-1/2} S_\mathcal B(t-s) K_1 S_\mathcal A(s)\bm  g \Vert_{H^2_xL^2_v}\,ds
\\
&\notag\lesssim  (1+t)^{-2} \Vert\bm g \Vert_{H^2_xL^2_{k+6} }    + \int_{0}^t (1+(t-s))^{-2}  \Vert \mu^{-1/2}  K_1 S_\mathcal A(s) \bm g \Vert_{H^2_xL^2_{6}}\,ds
\\
&\notag\lesssim  (1+t)^{-2} \Vert\bm g \Vert_{H^2_xL^2_{k+6} }    + \int_{0}^t (1+(t-s))^{-2}  \Vert S_\mathcal A(s)\bm  g \Vert_{H^2_xL^2_v}\,ds
\\
&\lesssim  (1+t)^{-2} \Vert\bm g \Vert_{H^2_xL^2_{k+6} }.    
\end{align}
Next, we give a more refined estimate of $S_\mathcal B(t)\bm g$. Notice from \eqref{Duha} that
\begin{align*}
 \left (\int_0^\infty \Vert S_ \mathcal B(t)\g  \Vert_{H^2_xL^2_k}^2 dt  \right)^{\frac 1 2} \lesssim \left ( \int_0^\infty \| S_\mathcal A(t)\g  \Vert_{H^2_xL^2_k}^2 dt \right)^{\frac 1 2} +  \left ( \int_0^\infty \Big\Vert \int_0^t   S_\mathcal B(s)   K_1 S_\mathcal A(t-s) \g ds \Big\Vert_{H^2_xL^2_k}^2 dt\right)^{\frac 1 2} . 
\end{align*}
For the second term, using Minkowski's inequality twice we easily compute that 
\begin{align*}
 \left (\int_0^\infty \Big\Vert \int_0^t   S_\mathcal B(s)   K_1 S_\mathcal A(t-s) \bm g\,ds \Big\Vert_{H^2_xL^2_k}^2 dt\right)^{\frac 1 2} 
\lesssim & \left ( \int_0^\infty  \left|  \int_0^t  \Vert  S_\mathcal B(s)   K_1 S_\mathcal A(t-s)\bm  g \Vert_{H^2_xL^2_k}\,ds \right|^2 dt \right)^{\frac 1 2} 
\\
\lesssim & \int_0^\infty   \left ( \int_s^\infty  \Vert S_\mathcal B(s)   K_1 S_ \mathcal A(t-s) \bm g \Vert_{H^2_xL^2_k}^2    dt \right)^{\frac 1 2}\,ds
\\
(\text{because of \eqref{equationa3}, \eqref{equationa4}})\lesssim & \int_0^\infty   \left ( \int_0^\infty  (1+s)^{-4} \Vert  K_1 S_\mathcal A(t) \bm g \Vert_{H^2_xL^2_{6}}^2  dt \right)^{\frac 1 2}\,ds
\\
(\text{because of \eqref{equationk1}})\lesssim &   \left ( \int_0^\infty  \Vert  S_\mathcal A(t) \bm g \Vert_{H^2_xL^2_v}^2   dt \right)^{\frac 1 2} ,
\end{align*}
Thus, using \eqref{inteA} or \eqref{duality} to control $S_{\mathcal A}$, we have
\begin{align}\label{420a}
\left ( \int_0^\infty \Vert S_{\mathcal B}(t)\bm g  \Vert_{H^2_xL^2_k}^2 dt \right)^{\frac 1 2}  \lesssim \left ( \int_0^\infty \Vert S_\mathcal A(t)\bm g  \Vert_{H^2_xL^2_k}^2 dt \right)^{\frac 1 2}  \lesssim  \min\Big\{\|\bm g\|_{H^2_xL^2_{k+6}},\,\Vert\bm g \Vert_{H^2_xH^{-s}_{k-\gamma/2}}\Big\}. 
\end{align}
Moreover, we need a regularizing estimate of $(S_{\mathcal B}(t)\textbf g)_f$. Notice that $f=(S_{\mathcal B}(t)\textbf g)_f$ is the solution of the equation  $\partial_t f_\pm =-v\cdot \nabla_x f_\pm +L_\pm f$, while the terms $E,B$ do not appear. Therefore there is no dissipation in $x$ variable, $f$ depends only on its initial data, and hence,  
by the duality argument, we have
\begin{align}\label{420b}
\Vert (S_{\mathcal B}(t)\bm g)_{f}\Vert_{H^2_xL^2_5}
\lesssim  t^{-1/2 }(1+t)^{-2} \Vert \bm g_{f}\Vert_{H^2_x H^{-s}_{10}}.
\end{align}

\smallskip\noindent{\bf Step 3: Estimate of $S_{L_1}(t)$.}
Finally, we come back to $L_1$ and consider the linearized equation $\partial_t \bm g =L_1 \bm g$ by letting $\g=S_{L_1}(t)\g_0$ with $\g=(f_+,f_-,E,B)$. Noticing that
\begin{equation}
\label{equationsl1}
\frac 12\partial_t {\|}[E,B]{\|}^2_{L^2_x}=-\int_{\mathbb T^3} \int_{\mathbb R^3}E\cdot v(f_+-f_{-})dv dx,
\end{equation}
and using the energy estimate in \cite{GS}, there exists $\lambda_0>0$, such that
\begin{equation}
\label{equationsl2}
\frac 12\partial_t\sum_\pm{\|}f_\pm{\|}^2_{L^2(\mu^{-1/2})}\le -\lambda _0{\|}\mu^{-1/2}(\textbf I-\textbf P\mu^{1/2})f_\pm{\|}^2_{H^s_{\gamma/2}}+\int_{\mathbb T^3} \int_{\mathbb R^3}E\cdot v(f_+-f_{-})dv dx.
\end{equation}
Thus, taking $H^2_xL^2_v$ inner product of $\partial_t \bm g =L_1 \bm g$ with $\mu^{-1}S_{L_1}(t)\g$ over $\mathbb T^3\times \mathbb R^3$ and recalling \eqref{equationsl1}, \eqref{equationsl2}, we have
\begin{equation}
\label{equationsl3}
\partial_t{\|}\mu^{-\frac 12}f{\|}^2_{H^2_xL^2_v}+\partial_t{\|}[E,B]{\|}^2_{H^2_x}+\lambda_0{\|}\mu^{-\frac 12}(\textbf I-\textbf P\mu^\frac 12)f{\|}^2_{H^2_xH^s_{\gamma/2}}\le 0. 
\end{equation}
For the weighted version of \eqref{equationsl3}, taking $H^2_xL^2_v$ inner product with $\langle v\rangle^k\mu^{-1}S_{L_1}(t)\g$ over $\mathbb T^3\times \mathbb R^3$, we have
\begin{equation}
\label{equationsl4}
\partial_t{\|}\mu^{-\frac 12}f{\|}^2_{H^2_xL^2_k}+\partial_t{\|}[E,B]{\|}^2_{H^2_x}+\lambda_0{\|}\mu^{-\frac 12}f{\|}^2_{H^2_xH^s_{k+\gamma/2}}\le C{\|}\mu^{-\frac 12}f{\|}^2_{H^2_xL^2_{\gamma/2}}.
\end{equation}
Recalling the macroscopic estimate \eqref{macro2}, for equation $\partial_t \bm g =L_1 \bm g$, there exists functional $E_{int, L_1}\lesssim {\|}\mu^{-\frac 12}f{\|}^2_{H^2_xL^2_v}$ such that
\begin{equation}
\label{equationsl5}
\begin{aligned}
&\partial_tE_{int,L_1}(t)+\partial_t{\|}[E,B]{\|}^2_{H^2_x}+\lambda_0{\|}\mu^{-\frac 12}\textbf P\mu^\frac 12f{\|}^2_{H^2_xH^s_{\gamma/2}}+\lambda_0{\|}[E,B]{\|}^2_{H^1_x}\\&\quad\le C{\|}\mu^{-\frac 12}(\textbf I-\textbf P\mu^\frac 12)S_{L_1}(f){\|}^2_{H^2_xL^2_{\gamma/2}}+{\|}[E,B]{\|}^4_{L^2_x}.
\end{aligned}
\end{equation}
Now take linear combinations \eqref{equationsl3}+$\kappa^2\times$\eqref{equationsl4}+$\kappa\times$\eqref{equationsl5} and assume the a priori assumption \[{\|}\mu^{-\frac 12}f{\|}^2_{H^2_xL^2_{k+6}}\le\varepsilon, \quad {\|}[E,B]{\|}^2_{L^2_x}\le\varepsilon.\]
Then we obtain that there exists a functional $\mathcal E_{L_1}(t)\sim {\|}\mu^{-\frac 12}f{\|}^2_{H^2_xL^2_k}+{\|}[E,B]{\|}^2_{H^2_x}$, such that
\begin{equation}
\label{equationslf1}
\partial_t\mathcal E_{L_1}(t)+\lambda_0{\|}\langle v\rangle^k\mu^{-\frac 12}f{\|}^2_{H^2_xH^s_{\gamma/2}}+\lambda_0{\|}[E,B]{\|}^2_{H^1_x}\le 0, 
\end{equation}
which implies 
\begin{align}\label{430}
    \sup_{0\le s\le t}\mathcal E_{L_1}(s)
    +\lambda_0\int^t_0{\|}\langle v\rangle^k\mu^{-\frac 12}f(s){\|}^2_{H^2_xH^s_{\gamma/2}}\,ds+\lambda_0\int^t_0{\|}[E,B](s){\|}^2_{H^1_x}\,ds\le 2\mathcal E_{L_1}(0). 
\end{align}
To obtain the large-time and regularizing behavior, we first examine the energy estimate \eqref{equationslf1} in a manner analogous to $S_{\mathcal A}(t)$, but only with respect to the $f$ coordinate. 
Regardless of the sign of $\gamma$, integrating \eqref{equationslf1} over $[0,t]$, we have 
\begin{align*}
    \|\mu^{-\frac 12}f(t)\|^2_{H^2_xL^2_k}
    \le \|\g_0\|_{H^2_xL^2_k}^2 -\lambda_0\big(\sup_{0\le s\le t}\|f\|_{L^2_xL^2_{k+6}}\big)^{-1}\int^t_0\|\mu^{-\frac 12}f\|^3_{H^2_xL^2_{k}}\,ds. 
\end{align*}

For operator $K_2$ we can calculate that for any $\alpha\in\mathbb N$,
\begin{align}\label{estimatek2}\notag
\Vert \mu^{-1/2} K_2 \bm g \Vert_{H^\alpha_xL^2_{10}}^2&= \Big\Vert \Big(E \cdot v \mu^{1/2},\, -E \cdot v \mu^{1/2},\,  E - \int_{\R^3}v (f_+-f_{-}) dv, B\Big) \Big\Vert_{H^\alpha_xL^2_{10}}^2 \\
&\lesssim \Vert f \Vert_{H^\alpha_xL^2_4}^2   + \Vert B  \Vert^2_{H^\alpha} +  \Vert E  \Vert^2_{H^\alpha}  = \Vert  \bm g \Vert_{H^\alpha_xL^2_4}^2. 
\end{align}

Thus, by Duhamel's principle, 
\begin{multline}
\label{equationl8}
\left (  \int_0^\infty \Vert( S_{L_1} (t)\bm g)_f  \Vert_{H^2_xL^2_k}^2 dt \right)^{\frac 1 2}  \lesssim \left (  \int_0^\infty \Vert (S_{\mathcal B}(t)\bm g )_f \Vert_{H^2_xL^2_k}^2 dt\right)^{\frac 1 2} \\  + \left (  \int_0^\infty \Big\Vert \int_0^t   (S_{L_1} (s)   K_2 S_{\mathcal B}(t-s)  \bm g)_f ds \Big\Vert_{H^2_xL^2_k}^2 dt \right)^{\frac 1 2}. 
\end{multline}
For the second term, by Minkowski's inequality and \eqref{430}, we compute 
\begin{align}
\label{equationl9bb}\notag
\left ( \int_0^\infty \Big\Vert \int_0^t   (S_{L_1}(t-s)   K_2 S_{\mathcal B}(s) \bm g )_f ds \Big\Vert_{H^2_xL^2_k}^2 dt \right)^{\frac 1 2} 
\lesssim &\notag \int_0^\infty  \left (  \int_s^\infty   \Vert ( S_{L_1}(t-s)  K_2 S_{\mathcal B}(s)\bm  g)_f \Vert_{H^2_xL^2_k}^2   dt \right)^{\frac 1 2}\,ds
\\
\lesssim &\notag\int_0^\infty   \left (  \int_0^\infty  \Vert ( S_{L_1}(t)  K_2S_{\mathcal B}(s)\bm  g)_f \Vert_{H^2_xL^2_k}^2  dt \right)^{\frac 1 2}\,ds  
\\
\lesssim &\notag  \int_0^\infty  \Vert  K_2 S_{\mathcal B}(t)\bm  g \Vert_{H^2_xL^2_v (\mu^{-1/2})}\,dt 
\\
\lesssim &    \int_0^\infty  \Vert  S_{\mathcal B}(t)\bm  g \Vert_{H^2_xL^2_4}\,dt. 
\end{align}
Gathering \eqref{420b} and \eqref{equationl9bb}, we have
\begin{equation*}
\left (  \int_0^\infty \Vert( S_{L_1} (t)\bm g)_f  \Vert_{H^2_xL^2_k}^2 dt \right)^{\frac 1 2}  \lesssim \left (  \int_0^\infty \Vert (S_\mathcal B(t)\bm g )_f \Vert_{H^2_xL^2_k}^2 dt\right)^{\frac 1 2}   +\int_0^\infty  \Vert  S_B(t) \bm g \Vert_{H^2_xL^2_4}\,dt  \lesssim\Vert \bm g \Vert_{H^2_xH^{-s}_{k-\gamma/2}}. 
\end{equation*}
This implies \eqref{equationl10}. Next we come to estimate $S_{L_1}\bm g$ including  coordinates $(E,B)$.

Due to the dissipations of $E$ and $B$ in the $x$ variable,
the computations of $[E,B]=(S_{L_1}(g)\bm g)_{E.B}$ differ. Initially, for any $K\in\mathbb N^+$, employing the Gagliardo-Nirenberg interpolation inequality,
\[
{\|}[E,B]{\|}_{H^{K}_x}\lesssim {\|}[E,B]{\|}^\frac 12_{H^{K-1}_x}{\|}[E,B]{\|}^\frac 12_{H^{K+1}_x}, 
\]
which implies that ${\|}[E,B]{\|}_{H^1_x}\gtrsim {\|}[E,B]{\|}^{-1}_{H^3_x} {\|}[E,B]{\|}^2_{H^2_x}$. Next, for $f=(S_{L_1}(g)\bm g)_f$, if $\gamma\ge 0$, we directly have \[{\|}\mu^{-\frac 12}f{\|}_{H^2_xH^s_{k+\gamma/2}} \gtrsim {\|}\mu^{-\frac 12}f{\|}_{H^2_xL^2_k};\] 
while for $-3\le\gamma<0$, we obtain from interpolation that 
\[
{\|}f{\|}_{L^2_v}\lesssim {\|}f\langle v \rangle^{\gamma/2}{\|}_{L^2_v}^{\frac 12} {\|}f\langle v \rangle^{4}{\|}^{\frac 12}_{L^2_v},
\]
and thus ${\|}\mu^{-\frac 12}f{\|}_{H^2_xH^s_{k+\gamma/2}} \gtrsim {\|}\mu^{-\frac 12}f{\|}^{-1}_{H^2_xL^2_{k+4}}{\|}\mu^{-\frac 12}f{\|}^2_{H^2_xL^2_k}$. So there exists $\lambda_1>0$ such that
\[
\partial_t\mathcal E_{L_1}(t)+\lambda_0({\|}\mu^{-\frac 12}f{\|}^2_{H^2_xL^2_{k+4}}+{\|}[E,B]{\|}^2_{H^3_x})^{-1}\mathcal E_{L_1}(t)^2\le 0.
\]
Solving this ODE and using \eqref{430} to control the sup-norm, we obtain that 
\begin{align}
\label{equationsl6}\notag
{\|}\mu^{-\frac 12}f(t){\|}^2_{H^2_xL^2_k}+{\|}{[E,B]}{\|}^2_{H^2_x}&\lesssim (1+t)^{-1}\sup_{0\le s\le t}({\|}\mu^{-\frac 12}f(s){\|}^2_{H^2_xL^2_{k+4}}+{\|}[E,B](s){\|}^2_{H^3_x})\\
&\lesssim (1+t)^{-1}({\|}\mu^{-\frac 12}f(0){\|}^2_{H^2_xL^2_{k+4}}+{\|}[E,B](0){\|}^2_{H^3_x}). 
\end{align}
Finally, combining \eqref{420}, \eqref{estimatek2} and \eqref{equationsl6}, 
we have large-time behavior:
\begin{align}\label{438}\notag
\Vert S_{L_1}(t)\bm g\Vert_{H^2_xL^2_k}   &\lesssim \Vert S_{\mathcal B}(t)\bm g \Vert_{H^2_xL^2_k}    + \int_{0}^t \Vert  S_{L_1}(t-s) K_2 S_{\mathcal B}(s)\bm  g_{E,B} \Vert_{H^2_xL^2_k}\,ds
\\
&\notag\lesssim  (1+t)^{-2} \Vert \bm g\Vert_{H^2_xL^2_k}   + \int_{0}^t  (1+(t-s))^{-1/2}\Vert  \mu^{-1/2} K_2 S_{\mathcal B}(s) \bm g \Vert_{H^3_xL^2_{4}}\,ds
\\
&\notag\lesssim  (1+t)^{-2} \Vert \bm g\Vert_{H^2_xL^2_k}    + \int_{0}^t  (1+(t-s))^{-1/2} \Vert S_{\mathcal B}(s)\bm  g \Vert_{H^3_xL^2_{4}}\,ds
\\
&\lesssim  (1+t)^{-1/2} \Vert\bm g \Vert_{H^3_xL^2_{4} }. 
\end{align}
Next, for integrability and regularizing effect, we write 
\begin{multline}
\label{equationl8a}
\left (  \int_0^\infty \Vert S_{L_1} (t)\bm g  \Vert_{H^2_xL^2_k}^2 dt \right)^{\frac 1 2}  \lesssim \left (  \int_0^\infty \Vert S_{\mathcal B}(t)\bm g \Vert_{H^2_xL^2_k}^2 dt\right)^{\frac 1 2} \\  + \left (  \int_0^\infty \Big\Vert \int_0^t S_{L_1} (s)   K_2 S_{\mathcal B}(t-s)  \bm g\, ds \Big\Vert_{H^2_xL^2_k}^2\,dt \right)^{\frac 1 2}. 
\end{multline}
For the first term, we can use \eqref{420a}. 
For the second term, by Minkowski's inequality, \eqref{420}, \eqref{430} and \eqref{estimatek2}, we compute 
\begin{align}
\label{equationl9}\notag
\left ( \int_0^\infty \Big\Vert \int_0^t  S_{L_1}(t-s)   K_2 S_{\mathcal B}(s) \bm g ds \Big\Vert_{H^2_xL^2_k}^2 dt \right)^{\frac 1 2} 
\lesssim &\notag \int_0^\infty  \left (  \int_s^\infty   \Vert S_{L_1}(t-s)  K_2 S_{\mathcal B}(s)\bm  g\Vert_{H^2_xL^2_k}^2   dt \right)^{\frac 1 2}\,ds
\\
\lesssim &\notag\int_0^\infty   \left (  \int_0^\infty  \Vert S_{L_1}(t)  K_2S_{\mathcal B}(s)\bm  g \Vert_{H^3_xL^2_k}^2  dt \right)^{\frac 1 2}\,ds  
\\
\lesssim &\notag  \int_0^\infty  \Vert  K_2 S_{\mathcal B}(t)\bm  g \Vert_{H^3_xL^2_v (\mu^{-1/2})}\,dt 
\\
\lesssim & \int_0^\infty  \Vert  S_{\mathcal B}(t)\bm  g \Vert_{H^3_xL^2_4}\,dt. 
\end{align}
Therefore, using \eqref{420} or \eqref{420b}, we continue \eqref{equationl8a} to deduce 
\begin{align*}
\left ( \int_0^\infty \Vert S_{L_1}(t)\bm g  \Vert_{H^2_xL^2_k}^2 dt \right)^{\frac 1 2}   \lesssim  \min\Big\{\|\g\|_{H^3_xL^2_{k+10}},\,\Vert \bm g \Vert_{H^3_xH^{-s}_{k-\gamma/2}}\Big\}.
\end{align*}
This completes the proof of Lemma \ref{lemsemigroup}. 
\end{proof}

\section{Proof of the main theorem}
\label{sec5}
In this Section, we come to the proof of the main stability theorem.
We first define the time-decreasing weight function
\begin{equation}
\label{weight decreasing}
w(\alpha, \beta)  = \langle v \rangle^{k (1+(1+t)^{-\frac 12} )- a|\alpha|  - b |\beta|}. 
\end{equation}
Then we have
\begin{equation}
\label{maintheorem0}
\begin{aligned}
\frac d {dt} w^2(t, \alpha, \beta)  &= \frac d {dt}  e^{2 (k (1+(1+t)^{-\frac 12} )-  a|\alpha|  - b|\beta| + c)} \log \langle v \rangle \\& = - \frac {k} {(1+t)^\frac 32} \log \langle v \rangle w^2(t, \alpha, \beta)\le -k\log 2(1+t)^{-\frac 32}w^2(t, \alpha, \beta)
\end{aligned}
\end{equation}

Now we are ready to prove our main theorem. 
\begin{proof}[Proof of Theorem \ref{maintheorem}]
We define the lifespan time $T^*$ by
\begin{equation}
\label{lifespan}
T^*:=\sup_{t\ge 0}\big\{E_{6, k} (t)\le  \epsilon, E_{2, k}(t) \le \epsilon (1+t)^{-3},\, \forall t \in [0, T^*)\big\}. 
\end{equation}
The fact that $T^*>0$ follows directly from local existence Theorem \ref{thmlocal}. We next prove that $T^*=\infty$. Notice from \eqref{VMB3} that
\begin{equation}
\label{maintheorem1aa}
\partial_t \partial^\alpha_\beta f_\pm + \partial^\alpha_\beta(v \cdot \nabla_x f_\pm)  \pm \partial^\alpha_\beta ((E + v \times B) \cdot \nabla_v f _\pm)\mp \partial^\alpha_\beta ( E  \cdot v)   \mu =\partial^\alpha_\beta L_\pm f+\partial^\alpha_\beta \Gamma_\pm(f , f). 
\end{equation}
we multiply \eqref{maintheorem1aa} by $w(t,\alpha,\beta)$ and do the standard $L^2$ energy estimate. Remind \eqref{maintheorem0}, we have 
\begin{multline*}
\frac 1 2 \partial_t \Vert \partial^\alpha_\beta f_\pm w(\alpha, \beta) \Vert_{L^2_{x, v}}^2 +k\log 2\Vert \partial^\alpha_\beta f_\pm w(\alpha, \beta) \Vert_{L^2_{x, v}}^2   - ( \partial^\alpha_\beta f_\pm , \partial^\alpha_\beta f_\pm  \partial_t w^2(\alpha, \beta) )_{L^2_{x, v}} +( \partial^\alpha_\beta(v \cdot \nabla_x f_\pm) , f_\pm w^2(\alpha, \beta) )_{L^2_{x, v}} \\ \pm (\partial^\alpha_\beta ((E + v \times B) \cdot \nabla_v f_\pm ), f_\pm w^2(\alpha, \beta) )_{L^2_{x, v}}
\mp( \partial^\alpha_\beta ( E  \cdot v)\mu   , f_\pm w^2(\alpha, \beta) )_{L^2_{x, v}} \\
  = 
  + ( L_\pm\partial^\alpha_\beta f_\pm,  f_\pm w^2(\alpha, \beta) )_{L^2_{x, v}} +(\partial^\alpha_\beta \Gamma_\pm(f , f),  f_\pm w^2(\alpha, \beta) )_{L^2_{x, v}}.  
\end{multline*}
we make $\partial^\alpha $ for \eqref{VMB2} and sum up over $|\alpha|\le 7$, and after summing $\alpha, \beta$ with $|\al|+|\beta|\le 7$, collecting the Lemmas from \ref{L58} to \ref{L28}, there exist constants $\lambda_k, C_k, C> 0$, such that 
\begin{equation}
\label{eqe4}
\begin{aligned}
\partial_t E_{7, k }(t)  + k\log 2 (1+t)^{-\frac 32}  E^f_{7, k} (t)+   \lambda_k D_{7, k}(t)   \le  &\,C_k \Vert [E, B]\Vert_{H^7_x} \Vert f \Vert_{H^7_x}  + \sqrt{E_{2, k}(t) }  E^f_{7, k}(t)
\\
&+\sqrt{E_{6, k} (t) }  \sqrt{E_{7, k} (t)} \sqrt{D_{7, k} (t)} + C \Vert f \Vert_{H^7_{x, v}}^2 , 
\end{aligned}
\end{equation}
where the energy functionals are given in \eqref{functional1} and \eqref{functional2}. It is obvious that 
\begin{equation}
\label{maintheorem2}
||f,E,B||^2_{H^7_{x,v}}\le E_{7,k}(t)
\end{equation}notice that on $[0, T^*]$, we have \begin{equation}
\label{maintheorem3}\sqrt{ E_{2,k}(t)}\le\epsilon (1+t)^{-\frac 32} \le k\log 2 (1+t)^{-\frac 32}\end{equation} and \begin{equation}
\label{maintheorem4} \sqrt{E_{6, k} (t) }  \sqrt{E_{7, k} (t)} \sqrt{D_{7, k} (t)}\le  \lambda_k D_{7, k}(t)+\frac 1{4\lambda_k}E_{6,k}(t)E_{7,k}(t)\le \lambda_k D_{7, k}(t)+  E_{7,k}(t)
\end{equation}
after combining \eqref{eqe4} to \eqref{maintheorem4} together, there exists a constant $C>0$, such that $\partial_t E_{7, k }(f)  \le C E_{7, k }(f) $ for any $t\in [0,T^*]$.
By the \emph{a priori} assumption \eqref{lifespan}, 
for any $t\in [0, T^*)$, we have $E_{7, k }(f)(t)  \le e^{Ct} E_{7, k}(f_0)$. It implies that, if $\g(t)=(f_+(t),f_-(t), E(t), B(t))$ is the solution to \eqref{VMB3} and \eqref{VMB4}, then for any $|\alpha| \le 6$, we have
\begin{align}\notag
\label{star1}
&\int_0^\infty ( S_{L_1}(\tau)L_1\partial^\alpha \g(t), S_{L_1} (\tau)  \partial^\alpha \g(t) )_{L^2_x L^2_v} d\tau 
\\
=&\notag\frac 12\int_0^\infty \frac {d} {d\tau} \Vert S_{L_1}(\tau)  \partial^\alpha  \g  (t)\Vert_{L^2_x L^2_v}^2 d\tau 
\\
=&\notag \lim_{\tau \to \infty}\frac 12\Vert S_{L_1}(\tau)  \partial^\alpha \g(t)\Vert_{L^2_xL^2_v}^2  - \frac 12\Vert  \partial^\alpha \g (t)\Vert_{L^2_xL^2_v}^2 
\\
\le& \notag\lim_{\tau \to \infty}  C(1+\tau)^{-1}  \Vert \g(t ) \Vert_{H^7_xL^2_{4}}^2 -\frac 12\Vert \partial^\alpha \g (t)\Vert_{L^2_xL^2_v}^2.
\\
\le & \lim_{\tau \to \infty}  C(1+\tau)^{-1} e^{Ct}  \Vert   \g(0 ) \Vert_{H^7_xL^2_{4}}^2 -\frac 12\Vert \partial^\alpha \g (t)\Vert_{L^2_xL^2_v}^2 = -\frac 12\Vert   \partial^\alpha \g(t) \Vert_{L^2_xL^2_v}^2, 
\end{align}
where we used the large-time behavior \eqref{semidecay}. 
For $2\le K\le 6$, under the definition of $T^*$, we can deduce similarly as \eqref{eqe4} that 
\begin{equation}
\label{eqe3}
\begin{aligned}
\partial_t E_{K, k }(t)  + k\log 2 (1+t)^{-\frac 32}  E^f_{K, k} (t)+   \lambda D_{K, k}(t)   \le  &\,C_k \Vert [E, B]\Vert_{H^K_x} \Vert f \Vert_{H^K_x}  + \sqrt{E_{2, k}(t) }  E^f_{K, k}(t)
\\
&+\sqrt{E_{K, k} (t) }  {D_{K, k} (t)} + C \Vert f \Vert_{H^K_{x, v}}^2 , 
\end{aligned}
\end{equation}
remind the \emph{a priori} assumption \eqref{lifespan}. Similarly as the argument for $K=7$, we obtain that for any $K=2,3,4,5,6$,
\begin{equation}
\label{star2}
\partial_t E_{K, k }(f)   + \lambda D_{K, k}(f) \lesssim   \Vert [E, B]\Vert_{H^K_x} \Vert f \Vert_{H^K_x} +  \Vert f \Vert_{H^K_{x, v}}^2. 
\end{equation}
morevoer, for the term $\Vert f \Vert_{H^K_{x, v}}^2$, notice that for any $\alpha, \beta$ with $ |\beta| \ge 1, |\alpha|+|\beta|\le K$, and any $\eta>0$ small, by interpolation we have
\[
\Vert  \partial^\alpha_\beta  f \Vert_{L^2_{x, v}}^2 \le C \Vert \partial^\alpha f\Vert_{L^2_xH^\beta_v}^2 \le\eta  \Vert \partial^\alpha f \Vert_{L^2_xH^{\beta+s}_v}^2  + C_\eta \Vert \partial^\alpha f \Vert_{L^{2}_{x, v}}^2  \le \eta  D_{K, k} (f)+ C_{\eta} \Vert \partial^\alpha f \Vert_{L^{2}_{x, v}}^2,
\]
which implies that 
\[
C_k\Vert f \Vert_{H^K_{x, v}}^2  \le \frac {\lambda}  4  D_{K, k}(f) + C_k \Vert f \Vert_{H^K_xL^2_v}^2. 
\]
Thus for $2\le K\le 6$ we have
\begin{align}\label{star33}
\partial_t E_{K, k }(f)   + \lambda D_{K, k}(f) \lesssim  & \Vert [E, B]\Vert_{H^K_x}^2 +  \Vert f \Vert_{H^K_xL^2_v}^2. 
\end{align}
Therefore, if $\g(t)=(f_+(t),f_-(t), E(t), B(t))$ is the solution to \eqref{VMB3} and \eqref{VMB4}, we define the norm ${\vertiii}\cdot{\vertiii}_{K,k}$  for $K\ge 6$ by 
\begin{equation}
\label{revisednorm}
{\vertiii}g{\vertiii}_{K, k}^2:=E^{\bm g}_{K, k}(t) + \eta \int_{0}^\infty (S_{L_1} (\tau)\bm g, S_{L_1}(\tau)\bm g )_{H^{K}_xL^2_v}  d \tau.
\end{equation}
From Lemma \ref{lemsemigroup}, we have
\begin{equation*}
\begin{aligned}
&\int_{0}^\infty (S_{L_1} (\tau) \g(t), S_{L_1}(\tau) \g(t) )_{H^{K}_xL^2_v}  d \tau d\tau \lesssim   \Vert    \g(t) \Vert_{H^{K+1}_xL^2_{4}}^2   \int_0^\infty C(1+\tau)^{-2}d\tau  \lesssim \Vert    \g(t) \Vert_{H^{K+1}_xL^2_{4}}^2
\end{aligned}
\end{equation*}
Thus, the norm ${\vertiii}g{\vertiii}^2_{K,k}$ may not be equivalent to $E_{{K+1}, k}(t)$ because of the loss of one-order derivative, but it is larger than  $E_{K, k} (t)$. 
Recall from \eqref{equationmacro01} and \eqref{L1} that the solution $\g$ to VMB/VML system \eqref{VMB3} and \eqref{VMB4} can be written as the solution to 
\[
\partial_t \g  =L_1 \g +N\g,\quad \nabla_x \cdot E = \int_{\R^3} \big(f_+(v)-f_-(v)\big)\,dv   , \quad  \nabla_x \cdot B =0, \quad \g(0)=\g_0, 
\]
where $N(\g) = ((E + v \times B) \cdot \nabla_v f +Q(f_+ , f_+)+Q(f_-,f_+),\,-(E + v \times B) \cdot \nabla_v f +Q(f_- , f_-)+Q(f_+,f_-), 0, 0)$.
Note that $N(\g)_{E,B}=(0,0)$.
Thus, if $s \ge \frac 1 2$, we can easily compute that for any $\eta>0$,
\begin{align}\notag
\label{star3}
 \int_{0}^\infty (S_{L_1} (\tau) N \g, S_{L_1}(\tau) \g )_{H^K_xL^2_v}  d \tau 
&=  \int_{0}^\infty ((S_{L_1} (\tau) N g)_f, (S_{L_1}(\tau) g)_f )_{H^K_xL^2_v}  d \tau
\\(\text{because of}\,\, \eqref{equationl10})
&\notag\lesssim \Vert N g \Vert_{H^K_xH^{-s}_{10-\gamma/2}}\Vert f\Vert_{H^K_xH^{-s}_{ 10-\gamma/2}}
\\
&\notag\lesssim  \Vert (E + v \times B) \cdot \nabla_v f +Q(f, f) \Vert_{H^K_xH^{-s}_{10-\gamma/2}}\Vert f \Vert_{H^K_xL^{2}_{ 12}}
\\(\text{because of}\,\, \eqref{equationbols1})
&\notag\lesssim  \Vert f \Vert_{H^K_xL^{2}_{ 12}} \Vert f \Vert_{H^6_x L^{2}_{12}}( \Vert f \Vert_{H^K_xH^{s}_{12}}+  \Vert E \Vert_{H^K_x})
\\
&\lesssim \frac 1 {2 \eta} \lambda D_{K, k} + C \eta E_{6, k}( \Vert [E, B]\Vert_{H^K_x}^2 +  \Vert f \Vert_{H^K_xL^2_v}^2  ). 
\end{align}
The above proof relies on the crucial observation that the coordinates $E$ and $B$ of $N(g)$ are both set to 0. Since the dissipation is exclusively in $E$ and $B$ and not in $f$, there is no loss in the $x$ derivatives for $f$. This feature greatly simplifies the proof. So we deduce from \eqref{star1}, \eqref{star33}, \eqref{star3} that 
\[
\frac12\frac {d} {dt} {\vertiii}\g{\vertiii}_{K, k}^2   +\frac \lambda 2  D_{K, k} (f)  +\left(\frac \eta 2 -C - \eta^2 C E_{6, k}\right) (\Vert [E, B]\Vert_{H^K_x}^2 +  \Vert f \Vert_{H^K_xL^2_v}^2 ) \le 0. 
\]
Therefore, if we take $\eta\gg C$ large enough and $E_{6, k} \ll\frac 1{C^3}$ small enough, there exists $\lambda_0>0$, such that $\frac \eta 2 -C - \eta^2 C E_{6, k}\ge \lambda_0$.  Then we have that for $2\le K\le 6$,
\begin{equation}
\label{convergencestar}
\frac {d} {dt} {\vertiii}\g{\vertiii}_{K, k}^2  + \frac {\lambda_0} 2 G_{K, k} \le 0. 
\end{equation}
Choose $K=6$ in \eqref{convergencestar}, we have
$\frac {d} {dt} {\vertiii}\g{\vertiii}_{6, k}^2 \le 0$.
This implies that for any $t\in [0,T^*]$, we have
\[
E_{6, k}(t)   \le {\vertiii} \g{\vertiii}^2   \le E_{6, k}(0) + \eta \int_{0}^\infty (S_{L_1} (\tau) \g_0, S_{L_1}(\tau) \g_0 )_{H^6_xL^2_v}  d \tau \lesssim E_{7, k}(0).    
\]
Thus if we assume the initial data $E_{7, k}(0)  \le \epsilon_0 $ is small, we are able to close the first argument $E_{6, k } (t) \le \epsilon/2$ in the \emph{a priori} assumption \eqref{lifespan}. 

\smallskip 
Finally, we come to prove the convergence. For any $k\ge \frac 92\ge- \frac{3\gamma}{2}$, Take $K=2$ in \eqref{convergencestar}, we obtain
\begin{equation}
\label{conve1}
\frac {d} {dt}    {\vertiii}f{\vertiii}_{2,k-\frac 92}^2    +\frac {\lambda_0} 2 G_{2,   k -\frac 92} (f)    \le 0. 
\end{equation}
Next, from the Gagliardo-Nirenberg interpolation inequality, we have 
\begin{equation}
\label{conv1}
\Vert [E,B] \Vert_{H^3_x} \lesssim \Vert [E,B] \Vert ^\frac 1  4_{H^6_x}    \Vert [E,B] \Vert^\frac 3   4_{H^2_x}    \le E_{6}^{\frac 1 4}   D_{2}^ {\frac 3 4}. 
\end{equation}
Since $ k \ge \frac 92\ge\frac{3\gamma}{2}$, we obtain from interpolation that 
$\Vert  f \Vert_{L^2_xL^2_v} \lesssim \Vert  \langle v \rangle^{k} f \Vert_{L^2_xL^2_v} ^{\frac 1 4}    \Vert  \langle v \rangle^{\gamma/2} f \Vert_{L^2_xL^2_v} ^{\frac 3 4} $. Next, using the \emph{a priori} assumption $E_{6, k}(f)\le \varepsilon$, we obtain from \eqref{conv1} that
\[{\vertiii}f{\vertiii}_{2,k-\frac 92}^2  \lesssim  E_{3,k-\frac 92}(f) \lesssim E_{6, k}^{\frac 1 4} (f) G_{2,k-\frac 92}^{\frac 34}(f) \lesssim  G_{2,k-\frac 92}^{\frac 34}(f) .\] Together with \eqref{conve1}, and recalling that 
$E_{6, k}(f)\le \varepsilon$, there exists a constant $\lambda_1>0$ such that
$$\frac{d}{dt}  {\vertiii}f{\vertiii}_{2,k-\frac 92}^2    +\lambda_1  {\vertiii}f{\vertiii}_{2,k-3}^\frac 83 \le 0.$$
Since $E_{5, k} (f_0) \le \epsilon_0$, solving this ODE, we finally obtain that $$E_{2, k-\frac 92}(f)  \le {\vertiii}f{\vertiii}_{2,k-\frac 92}^2   \lesssim  \epsilon_0 (1+t)^{-3}.$$ 
This finally closes the \emph{a priori} assumption and implies 
that the lifespan time $T^*=\infty$. This completes the proof of Theorem \ref{maintheorem}.
\end{proof}

\medskip \noindent
{\bf Acknowledgments.} C. Cao is supported by grants from Department of applied mathematics, The Hong Kong Polytechnic University. D.-Q. Deng was partially supported by the National Research Foundation of Korea (NRF) grant funded by the Korea government (MSIT) No. RS-2023-00210484 and No. RS-2023-00212304. X.-Y. Li has been partially supported by the Basque Government through the BERC 2022- 2025 program and by the Spanish State Research Agency through BCAM Severo Ochoa excellence accreditation SEV-2017-0718 and through project PID2020-114189RB-I00 funded by Agencia Estatal de Investigaci´on (PID2020- 114189RB-I00 / AEI / 10.13039/501100011033).  He has also been supported by IMT (Institut de Mathematiqu\'es), Universit\'e de Toulouse III Paul Sabatier.\\


\bibliography{bibtex}
\bibliographystyle{amsplain}




\end{document}